\documentclass[10pt,aps,prb, preprint,reqno]{amsart}
\usepackage{amsmath}
\usepackage{amssymb}
\usepackage{yfonts}
\usepackage{hyperref}
\usepackage{mathrsfs}
\usepackage {latexsym}
\usepackage{graphics}
\usepackage{txfonts}
\usepackage{breqn}
\usepackage{cite}
\usepackage{color}
\usepackage[shortlabels]{enumitem}
\usepackage{wasysym}
\usepackage{tikz}
\usetikzlibrary{snakes}
\usetikzlibrary{arrows.meta}
\usetikzlibrary{decorations.pathmorphing}
\usetikzlibrary{er, positioning}

\newtheorem{theorem}{Theorem}[section]
\newtheorem{remark}{Remark}[section]
\newtheorem{lemma}[theorem]{Lemma}
\newtheorem{assume}[theorem]{Assumption}
\newtheorem{proposition}[theorem]{Proposition}
\newtheorem{corollary}[theorem]{Corollary}
\newtheorem{define}{Definition}[section]

\newtheorem*{proposition*}{Proposition}

\newcommand{\joinR}{\hspace{-.1em}}
\newcommand{\RomanI}{I}
\newcommand{\RomanII}{\mbox{\RomanI\joinR\RomanI}}
\newcommand{\RomanIII}{\mbox{\RomanI\joinR\RomanII}}
\newcommand{\RomanIV}{\mbox{\RomanI\joinR\RomanV}}
\newcommand{\RomanV}{V}

\DeclareMathOperator*{\Id}{Id}

\DeclareMathOperator*{\supp}{supp}
\DeclareMathOperator*{\divergence}{div}
\DeclareMathSymbol{:}{\mathord}{operators}{"3A}

\allowdisplaybreaks

\begin{document}
\title[Singular Burgers' equation]{Global unique solution to the perturbation of the Burgers' equation \\ forced by derivatives of space-time white noise}

\subjclass[2010]{35A02; 35R60; 76F30}
 
\author[Kazuo Yamazaki]{Kazuo Yamazaki}  
\address{Department of Mathematics, University of Nebraska, Lincoln, 243 Avery Hall, PO Box, 880130, Lincoln, NE 68588-0130, U.S.A.; Phone: 402-473-3731; Fax: 402-472-8466}
\email{kyamazaki2@unl.edu}
\date{}
\keywords{Anderson Hamiltonian; Burgers' equation; Global well-posedness; Paracontrolled distributions; Space-time white noise.}
\thanks{This work was supported by the Simons Foundation MPS-TSM-00962572.}

\begin{abstract}
We consider the one-dimensional Burgers' equation forced by fractional derivative of order $\frac{1}{2}$ applied on space-time white noise. Relying on the approaches of Anderson Hamiltonian from Allez and Chouk (2015, arXiv:1511.02718 [math.PR]) and two-dimensional Navier-Stokes equations forced by space-time white noise from Hairer and Rosati (2024, Annals of PDE, \textbf{10}, pp. 1--46), we prove the global-in-time existence and uniqueness of its mild and weak solutions.
\end{abstract}

\maketitle

\section{Introduction}\label{Section 1}

\subsection{Motivation from physics and real-world applications}\label{Section 1.1} 
Since the pioneering work \cite{LL57} on hydrodynamic fluctuations in 1957, many partial differential equations (PDEs) in fluid mechanics and mathematical physics have been investigated under the force by random noise, especially space-time white noise (STWN) \eqref{STWN}: ferromagnet \cite{MM75}; Kardar-Parisi-Zhang (KPZ) equation \eqref{KPZ} \cite{KPZ86}; magnetohydrodynamics (MHD) system \cite{CT92}; Navier-Stokes equations \cite{FNS77, YO86}; Rayleigh-B$\acute{\mathrm{e}}$nard equation \cite{ACHS81, GP75, HS92, SH77, ZS71}. 

The model of main interests in this manuscript is the Burgers' equation \eqref{Burgers'}, introduced by Bateman \cite{B15} in 1915 and later studied by Burgers \cite{B48} in 1948. It has rich applications in fluid mechanics, gas dynamics, and traffic flow. Furthermore, it has been investigated in great depth as the prototype for one-dimensional (1D) toy model of the Navier-Stokes equations \eqref{Navier-Stokes}, as well as conservation law that exhibit finite-time shock in the inviscid case. The Burgers' equation forced by STWN has also caught much attention from the physics community (e.g. \cite[Equation (2.8)]{FNS77}). Moreover, taking spatial derivatives on the KPZ equation \eqref{KPZ} leads to the vorticity-free velocity field solving the Burgers' equation forced by a derivative of the STWN (e.g. \cite[Equation (3)]{KPZ86} and \cite[Equation (B.1)]{BG97}). 

Global-in-time existence and uniqueness of a solution is a fundamental property of any system of PDEs and the roughness of the STWN makes its verification notoriously difficult. In this manuscript, we prove the global well-posedness of the 1D Burgers' equation forced by a fractional derivative of order $\frac{1}{2}$ applied on the STWN. Our proof follows the approach of \cite{HR24} on the 2D Navier-Stokes equations forced by STWN but without any derivatives, and generalizes \cite{AC15} on the 2D Anderson Hamiltonian with noise that is white only in space.

\subsection{Past relevant results and mathematical motivation}\label{Section 1.2}
We set up a minimum amount of notations before introducing the equations of our main interests. We define $\mathbb{N} \triangleq \{ 1, 2, \hdots \}$ and $\mathbb{N}_{0} \triangleq \mathbb{N} \cup \{0\}$, and work with a spatial variable $x \in \mathbb{T}^{d} = (\mathbb{R} \setminus \mathbb{Z})^{d}$ for $d \in \mathbb{N}$ with primary focus on $d = 1$. We abbreviate by $\partial_{t} \triangleq  \frac{\partial}{\partial t}, \partial_{i} \triangleq \partial_{x_{i}} \triangleq \frac{\partial}{\partial x_{i}}$ for $i \in \{1, \hdots, d\}$, and define $\mathbb{P}_{\neq 0} f \triangleq f - \fint_{\mathbb{T}^{d}} f(x) dx$ and $\mathbb{P}_{L}$ to be the Leray projection operator onto the space of divergence-free vector fields. We write $A \lesssim_{\alpha, \beta} B$ whenever there exists a constant $C = C(\alpha,\beta) \geq 0$ such that $A \leq CB$ and $A \approx_{\alpha,\beta} B$ in case $A \lesssim_{\alpha,\beta}B$ and $A \gtrsim_{\alpha,\beta}B$. We often write $A \overset{( \cdot)}{\lesssim} B$ whenever $A\lesssim B$ due to $(\cdot)$. We denote the Lebeague, homogeneous and inhomogeneous Sobolev spaces by $L^{p}, \dot{H}^{s}$, and $H^{s}$ for $p\in [1,\infty], s\in \mathbb{R}$ with corresponding norms of $\lVert \cdot\rVert_{L^{p}}, \lVert \cdot \rVert_{\dot{H}^{s}}$, and $\lVert \cdot \rVert_{H^{s}}$, respectively. Finally, we denote the Schwartz space and its dual by $\mathcal{S}$ and $\mathcal{S}'$, and Fourier transform of $f$ by $\mathcal{F}(f) = \hat{f}$. 

Let us fix a probability space $(\Omega, \mathcal{F}, \mathbb{P})$ so that the STWN $\xi$ can be introduced as a distribution-valued Gaussian field with a correlation of 
\begin{equation}\label{STWN}
\mathbb{E} [\xi(t,x) \xi(s,y)] = \delta(t-s) \delta (x-y) 
\end{equation} 
where $\mathbb{E}$ represents the mathematical expectation with respect to (w.r.t.) $\mathbb{P}$; i.e., 
\begin{align*}
\mathbb{E} [ \xi(\phi) \xi(\psi)] = \int_{\mathbb{R} \times \mathbb{T}^{d}} \phi(t,x) \psi(t,x) dx dt \hspace{3mm} \forall \hspace{1mm} \phi, \psi \in \mathcal{S} (\mathbb{R} \times \mathbb{T}^{d}). 
\end{align*}
We define $\Lambda^{\gamma} \triangleq (-\Delta)^{\frac{\gamma}{2}}$ as a fractional derivative of order $\gamma \in \mathbb{R}$, specifically a Fourier operator with a symbol $\lvert m \rvert^{\gamma}$ so that $\widehat{ \Lambda^{\gamma} f} (m) = \lvert m \rvert^{\gamma} \hat{f} (m)$. We also recall the H$\ddot{\mathrm{o}}$lder-Besov spaces $\mathcal{C}^{\gamma} \triangleq B_{\infty,\infty}^{\gamma}$ for $\gamma \in \mathbb{R}$ which is equivalent to the classical H$\ddot{\mathrm{o}}$lder spaces $C^{\alpha}$ whenever $\alpha \in (0,\infty) \setminus \mathbb{N}$ although $C^{k} \subsetneq \mathcal{C}^{k}$ for all $k \in \mathbb{N}$ (see \cite[p. 99]{BCD11}); we defer detailed definitions of Besov spaces to Section \ref{Section 3}. 

For general discussions, let us consider $m_{1}, m_{2} \in \mathbb{R}$ and the velocity and pressure fields $u: \hspace{1mm} \mathbb{R}_{\geq 0} \times \mathbb{T}^{d} \mapsto \mathbb{R}^{d}$ and $\pi: \hspace{1mm} \mathbb{R}_{\geq 0} \times \mathbb{T}^{d} \mapsto \mathbb{R}$ that satisfy 
\begin{equation}\label{Navier-Stokes}
\partial_{t} u + \divergence (u\otimes u) + \nabla \pi + \nu \Lambda^{m_{1}} u = \Lambda^{m_{2}} \xi, \hspace{1mm} \nabla\cdot u = 0, \hspace{1mm} \text{ for } t > 0, 
\end{equation} 
starting from given initial data $u^{\text{in}}(x) = u(0,x)$; the case $\nu > 0, \nu = 0$ represent the Navier-Stokes and the Euler equations, respectively. If we define a scaling (see \cite[Lemma 10.2]{H14})
\begin{equation}
\mathfrak{s} \triangleq (m_{1}, \underbrace{1, \hdots, 1}_{d\text{-many}}) \text{ so that } \lvert \mathfrak{s} \rvert \triangleq m_{1} + d,
\end{equation} 
then we have $\mathbb{P}$-almost surely ($\mathbb{P}$-a.s.) $\xi \in \mathcal{C}_{\mathfrak{s}}^{\alpha}$ for every $\alpha < - \frac{\lvert \mathfrak{s} \rvert}{2} = - \frac{m_{1} + d}{2}$ where $\mathcal{C}_{\mathfrak{s}}^{\alpha}$ is the scaled H$\ddot{\mathrm{o}}$lder-Besov space of space-time distributions introduced in \cite[Theorem 3.7]{H14}. Consequently, $\Lambda^{m_{2}} \xi \in \mathcal{C}_{x}^{\alpha}$ for $\alpha < - \frac{m_{1} + d}{2} - m_{2}$ so that $u \in \mathcal{C}_{x}^{\gamma}$ for $\gamma < \frac{m_{1}}{2} - \frac{d}{2} - m_{2}$ (e.g. \cite[Lemma 4.1]{BK17}). Due to the Bony's estimates (see Lemma \ref{Lemma 3.1}), a product $fg$ is well-defined if and only if $f \in \mathcal{C}_{x}^{\beta_{1}}, g \in \mathcal{C}_{x}^{\beta_{2}}$ for $\sum_{j=1}^{2} \beta_{j} > 0$. Therefore, $u \otimes u$ in \eqref{Navier-Stokes} becomes ill-defined if 
\begin{equation}\label{Bony's threshold}
\frac{m_{1}}{2} - \frac{d}{2} - m_{2} \leq 0;
\end{equation} 
let us refer to $m_{2} = \frac{m_{1}}{2} - \frac{d}{2}$ as the Bony's threshold. Such PDEs with ill-defined products due to rough stochastic forces are called singular stochastic PDEs (SPDEs). 

We notice that \eqref{Bony's threshold} in case $d = 2, m_{1} = 2, m_{2} = 0$ implies that the 2D Navier-Stokes equations forced by STWN is barely singular. In this case, Da Prato and Debussche \cite{DD02} decomposed \eqref{Navier-Stokes} to two parts: $u = X + v$ where 
\begin{subequations}\label{est 14}
\begin{align}
&\partial_{t} X = \nu \Delta X + \mathbb{P}_{L} \mathbb{P}_{\neq 0} \xi \hspace{11mm} \text{ for } t > 0, \hspace{3mm} X(0,x) = 0, \label{est 14a}\\
&\partial_{t} v + \mathbb{P}_{L} \divergence (X+ v)^{\otimes 2} = \nu \Delta v  \hspace{1mm} \text{ for } t > 0, \hspace{3mm} v(0,x) = u^{\text{in}}(x), \label{est 14b} 
\end{align}
\end{subequations}  
where $W^{\otimes 2} \triangleq W \otimes W$. Although $X^{\otimes 2}$ is ill-defined, Da Prato and Debussche were able to prove local solution theory using Wick products (global solution theory will be discussed shortly); similar approach was successfully applied on the 2D $\Phi^{4}$ model in  \cite{DD03b}. However, as can be seen from \eqref{Bony's threshold}, the singularity worsens in higher dimension and the approach of Da Prato and Debussche did not seem applicable to other singular SPDEs. 

One of the first breakthroughs on singular SPDEs that are not just barely singular was achieved by Hairer \cite{H13} on the 1D KPZ equation forced by STWN 
\begin{equation}\label{KPZ}
\partial_{t} h = \nu \Delta h + \lvert \nabla h \rvert^{2} + \xi 
\end{equation}
where $\nabla h \in \mathcal{C}^{\gamma} (\mathbb{T})$ for $\gamma < - \frac{1}{2}$ (the first was actually \cite{H11} on the Burgers' equation, to be described shortly). Using the theory of rough path due to Lyons \cite{L98}, Hairer was able to construct a local-in-time solution to the 1D KPZ equation (its global-in-time aspect will be discussed shortly). Subsequently, two breakthrough techniques were developed: the theory of paracontrolled distributions \cite{GIP15} by Gubinelli, Imkeller, and Perkoswki, and the theory of regularity structures \cite{H14} by Hairer. They led to many new developments; e.g. solution theory \cite{CC18} (also \cite{Y23a, ZZ15}) and strong Feller property \cite{HM18a} (also \cite{Y21a, ZZ20}). 

Yet, broadly stated, the solution theory that is directly attainable from either the theory of paracontrolled distributions or regularity structures is only local in time and in \emph{locally subcritical} case, which informally requires that the nonlinear term is smoother than the noise (see \cite[Assumption 8.3]{H14} for complete definition). Extending such a solution to be global-in-time has significant motivation, highlighted by recent developments on the stochastic quantization approach through the stochastic Yang-Mills equation \cite{CCHS22a, CCHS22b} due to Chandra, Chevyrev, Hairer, and Shen. Of most relevance, we first recall that Da Prato and Debussche in \cite{DD02} were able to take advantage of the explicit knowledge of an invariant measure ``$\mu = \times_{k\in\mathbb{Z}_{0}^{2}} \mathcal{N} (0, 1/2\lvert k \rvert^{2})$'' (see \cite[p. 185]{DD02}) due to the special property 
\begin{equation}\label{explicit knowledge}
\int_{\mathbb{T}^{2}} (u\cdot\nabla) u \cdot \Delta u dx = 0 
\end{equation} 
of the solution $u$ to the 2D Navier-Stokes equations, particularly its Gaussianity, and deduce the existence of path-wise unique solution globally in time starting from $\mu$-almost every initial data $u^{\text{in}}$ (see \cite[Theorem 5.1]{DD02}). Analogous result for the 2D $\Phi^{4}$ model was obtained also by Da Prato and Debussche in \cite[Theorem 4.2]{DD03b}, and later extended to the whole plane by Mourrat and Weber \cite{MW17}. Additionally, Gubinelli and Perkowski \cite{GP17} observed that the solution to the 1D KPZ equation constructed by Hairer \cite{H13} is global-in-time (see \cite[p. 170 and Corollary 7.5]{GP17}). Loosely speaking, the examples thus far are limited to singular SPDEs with explicit knowledge of invariant measure (2D Navier-Stokes equations due to \eqref{explicit knowledge}), favorable nonlinear term (2D $\Phi^{4}$ model with damping nonlinearity), or useful transformation (1D KPZ equation with Cole-Hopf transform). 

\subsection{Approach of \cite{HR24}}\label{Section 1.3}
We describe the approach of \cite{HR24} by Hairer and Rosati, the main source of our inspiration, which particularly constructed a global-in-time unique solution to the 2D Navier-Stokes equations forced by STWN without relying on its invariant measure in contrast to \cite{DD02} (\hspace{1sp}\cite{Y23b} in the case of the 2D MHD system forced by STWN for which an analogue of \eqref{explicit knowledge} fails). To understand the mechanism behind \cite{HR24}, it is instructive to recall the work of Tao \cite{T07, T09} on the global regularity result for deterministic logarithmically supercritical PDEs. Lions \cite{L69} proved that the Leray-Hopf weak solution to the deterministic $d$-dimensional Navier-Stokes equations (\eqref{Navier-Stokes} with zero noise) is unique as long as 
\begin{equation}\label{Lions' criticality} 
m_{1} \geq 1 + \frac{d}{2}, 
\end{equation} 
to which we refer as Lions' criticality; thus, we say that the deterministic $d$-dimensional Navier-Stokes equations is $L^{2}(\mathbb{T}^{d})$-subcritical, critical, and supercritical if $m_{1} > 1+ \frac{d}{2}, m_{1} = 1 + \frac{d}{2}$, and $m_{1} < 1+ \frac{d}{2}$, respectively. Tao in \cite{T09} proved the global well-posedness of the $d$-dimensional Navier-Stokes equations with diffusion of the form $\nu \mathcal{L} u$ such that 
\begin{equation}\label{est 261}
\widehat{\nu \mathcal{L} u}( k) = \frac{\nu \lvert k \rvert^{1 + \frac{d}{2}}}{\ln^{\frac{1}{2}}(2+ \lvert k \rvert^{2})} \hat{u}(k)
\end{equation} 
and thus in the \emph{logarithmically supercritical} case (see also \cite{BMR14, W11, Y14a, Y18}). A typical path toward the $H^{s}$-bound via Gronwall's inequality requires 
\begin{align}\label{classical Hs estimate}
\partial_{t} \lVert u(t) \rVert_{H^{s}}^{2} \lesssim a(t) \lVert u(t) \rVert_{H^{s}}^{2} \text{ where } a \in L^{1} (0, T) 
\end{align}
and naive energy estimates on the logarithmically supercritical Navier-Stokes equations to verify \eqref{classical Hs estimate} certainly fail. The key observation by Tao is that the following logarithmically worse bound still implies the $H^{s}$-bound of the solution: 
\begin{equation}\label{logarithmic worsening}
\partial_{t} \lVert u(t) \rVert_{H^{s}}^{2} \lesssim a(t) \lVert u(t) \rVert_{H^{s}}^{2} \ln ( e + \lVert u(t) \rVert_{H^{s}}^{2}) \hspace{3mm} \text{ where } a \in L^{1}(0,T). 
\end{equation} 
(see \cite[p. 362]{T09}). Loosely stated, to take advantage of this logarithmic affordability $\ln ( e + \lVert u(t) \rVert_{H^{s}}^{2})$ to handle the logarithmic worsening $\frac{1}{ \ln^{\frac{1}{2}} (2+ \lvert k \rvert^{2})}$ in \eqref{est 261}, Tao split the Fourier frequency of Littlewood-Paley decomposition with time-dependent cutoff as 
\begin{align*}
\sum_{k\geq -1} = \sum_{k: 2^{k} \leq e+ \lVert u(t) \rVert_{H^{s}}^{2}} + \sum_{k: 2^{k} > e+ \lVert u(t) \rVert_{H^{s}}^{2}}.
\end{align*}
Then the lower frequency side leads to $ln(e+ \lVert u(t) \rVert_{H^{s}}^{2})$ as $\sum_{k: 2^{k} \leq e+ \lVert u(t) \rVert_{H^{s}}^{2}} 1 \lesssim ln(e+ \lVert u(t) \rVert_{H^{s}}^{2})$. 

To prove the global solution theory for \eqref{Navier-Stokes} with $m_{1} = 2, m_{2} = 0$, \cite{HR24} defined 
\begin{equation}\label{est 5}
(\nabla_{\text{symm}} \phi)_{i,j} \triangleq \frac{1}{2} ( \partial_{i} \phi_{j} + \partial_{j} \phi_{i}) \hspace{3mm} \text{ and } \hspace{3mm}  \mathcal{A}_{t}^{\lambda} \triangleq \frac{ \nu \Delta \Id}{2} - \nabla_{\text{symm}} \mathcal{L}_{\lambda_{t}} X - r_{\lambda} (t) \Id 
\end{equation}
where $\mathcal{L}_{\lambda}$ is the projection onto the lower frequencies (see \eqref{est 3}) and $r_{\lambda}(t)$ is the renormalization constant (see  \eqref{est 63b}), which led to an $L^{2}(\mathbb{T}^{2})$-estimate of $w^{\mathcal{L}}$, the lower frequency part of a certain function $w$ (see \eqref{Define QH, wH, wL}):
\begin{subequations} 
\begin{align}
\partial_{t} \lVert w^{\mathcal{L}}(t) \rVert_{L^{2}}^{2} =& - \nu \lVert \nabla w^{\mathcal{L}}  (t)\rVert_{\dot{H}^{1}}^{2} + 2 \int_{\mathbb{T}^{2}} w^{\mathcal{L}} \cdot \left[ \frac{\nu \Delta \Id}{2} - \nabla_{\text{symm}} \mathcal{L}_{\lambda_{t}} X\right] w^{\mathcal{L}}(t) dx + \hdots  \label{est 16a}\\
\overset{\eqref{est 5}}{=}& - \nu \lVert \nabla w^{\mathcal{L}} (t)\rVert_{\dot{H}^{1}}^{2} + 2 \int_{\mathbb{T}^{2}} w^{\mathcal{L}} \cdot \mathcal{A}_{t}^{\lambda} w^{\mathcal{L}}(t) dx + 2 r_{\lambda_{t}} \lVert w^{\mathcal{L}}(t) \rVert_{L^{2}}^{2} + \hdots,  \label{est 16b}
\end{align}
\end{subequations} 
where we omitted other terms for simplicity of this discussion. By \cite[Proposition 6.1]{HR24}, which is due to results from \cite{AC15} concerning the 2D Anderson Hamiltonian (to be discussed shortly), we have 
\begin{align}\label{est 15} 
2 \int_{\mathbb{T}^{2}} w^{\mathcal{L}} \cdot \mathcal{A}_{t}^{\lambda} w^{\mathcal{L}} (t) dx \lesssim \lVert w^{\mathcal{L}}(t) \rVert_{L^{2}}^{2} 
\end{align}
(see \eqref{est 68}) so that the last remaining crucial ingredient that we need, similarly to \eqref{logarithmic worsening}, is 
\begin{equation}
r_{\lambda}(t) \lesssim \ln(\lambda_{t})
\end{equation} 
since $\ln(\lambda_{t}) \approx \ln(1+ \lVert w(T_{i}) \rVert_{L^{2}})$ for $T_{i}$ defined on \cite[p. 16]{HR24} (see Definition \ref{Definition 4.2}). By making these ideas rigorous, \cite{HR24} constructed a global-in-time unique solution to \eqref{Navier-Stokes} with $m_{1} = 2, m_{2} = 0$; in fact, they added another rough force to clarify that the approach of \cite{DD02} via invariant measure does not apply. We will follow the same suit in \eqref{Burgers'}-\eqref{define zeta}. Consequently, a logarithmic bound on the renormalization constant seems indispensable for the success of this scheme. Going through our computations, specifically \eqref{est 148}-\eqref{est 149}, \eqref{est 150}-\eqref{est 152}, indicates that the logarithmic growth of $r_{\lambda}(t)$ for \eqref{Navier-Stokes} requires
\begin{equation}\label{logarithmic renormalization criticality}
m_{2} = m_{1} - \frac{d+2}{2}. 
\end{equation} 

We now summarize our discussions thus far as follows.
\begin{enumerate}
\item On one hand, by Bony's threshold \eqref{Bony's threshold} we are interested in the singular case when $\frac{m_{1}}{2} - \frac{d}{2} \leq m_{2}$.
\item On the other hand, in terms of the energy estimates part of proof, Lions' criticality \eqref{Lions' criticality} suggests we should need $m_{1} \geq 1+ \frac{d}{2}$. 
\item Finally, the logarithmic renormalization criticality \eqref{logarithmic renormalization criticality} requires $m_{2} = m_{1} - \frac{d+2}{2}$. 
\end{enumerate} 
For this reason, in this manuscript we choose to focus on the 1D Burgers' equation with $m_{1} = 2, m_{2} = \frac{1}{2}$ to reach both Bony's and logarithmic renormalization criticalities. Furthermore, Hairer and Rosati \cite{HR24} chose an additional force $\zeta \in \mathcal{C}_{\mathfrak{s}}^{-2 + 3 \kappa} ( \mathbb{R} \times \mathbb{T}^{2}; \mathbb{R}^{2})$, intentionally not in the Cameron-Martin space $L_{\text{loc}}^{2}L_{x}^{2}$ (e.g. \cite[p. 32]{N95})  so that the law of the solution $u$ to the 2D Navier-stokes equations forced by STWN and $\zeta$ have no obvious link to the law of $X$ in \eqref{est 14a} and the approach via the explicit knowledge of invariant measure from \cite{DD02} becomes inapplicable. We follow the same suit and consider $\theta: \hspace{1mm} \mathbb{R}_{\geq 0} \times \mathbb{T} \mapsto \mathbb{R}$ that solves  
\begin{equation}\label{Burgers'}
\partial_{t} \theta + \frac{1}{2} \partial_{x} \theta^{2} - \nu \partial_{x}^{2} \theta = \mathbb{P}_{\neq 0} \zeta + \Lambda^{\frac{1}{2}} \xi \hspace{1mm} \text{ for } t > 0, 
\end{equation} 
starting from given initial data $\theta^{\text{in}}(x) = \theta(0,x)$, where $\nu > 0$ and 
\begin{equation}\label{define zeta}
\zeta \in \mathcal{C}_{\mathfrak{s}}^{-2 + 3 \kappa} ( \mathbb{R} \times \mathbb{T}; \mathbb{R})
\end{equation} 
is an additional perturbation similarly to \cite{HR24, Y23b}, and $\kappa \in (0,1)$ will be taken small.

The unforced deterministic Burgers' equation is \eqref{Burgers'} with $\xi \equiv \zeta \equiv 0$, and by now, the evolution of its solution is well understood. In case the diffusion ``$-\nu \partial_{x}^{2} \theta$'' in \eqref{Burgers'} replaced by $\nu \Lambda^{m_{1}} \theta$, the solution experiences finite-time shock for $m_{1} \in [0,1)$ and remains unique globally in time for $m_{1} \geq 1$ (e.g. \cite{ADV07, DDL09, KNS08}). 

In case the STWN $\xi$ is present in \eqref{Burgers'}, as can be seen from \eqref{Bony's threshold}, the product $\theta^{2}$ in \eqref{Burgers'} is well-defined in the case $d= 1, m_{1} = 2, m_{2} = 0$ and thus such 1D Burgers' equation with full Laplacian forced by STWN was proven to be globally well-posed by Bertini, Cancrini, and Jona-Lasinio \cite[Theorem 2.2]{BCJ94} via stochastic Cole-Hopf transform and Da Prato, Debussche, and Temam \cite[Theorem 3.1]{DDT94} by an approach akin to \eqref{est 14}. Yet, even in this well-posed case of $d = 1, m_{1} = 2, m_{2} = 0$, the solution $\theta \in C^{\gamma} (\mathbb{T})$ for $\gamma < \frac{1}{2}$ is not differentiable and consequently, inaccuracies in numerical approximations have been pointed out by Hairer and Voss \cite{HV11}. Additionally, Hairer \cite{H11} considered the 1D generalized Burgers' equation by replacing $\frac{1}{2} \partial_{x} \theta^{2}$ by $g(\theta) \partial_{x} \theta$ where $g$ is not a gradient type, the case in which one cannot shift a derivative in its weak formulation and even Young's integral becomes barely ill-defined due to $\theta \in \mathcal{C}_{x}^{\gamma}$ for $\gamma < \frac{1}{2}$. Hairer overcame this difficulty via the theory of rough paths and proved its global solution theory in \cite[Theorem 3.6]{H11} for $g$ that is sufficiently smooth; subsequently, Hairer and Weber \cite[Theorem 3.5]{HW13} extended this result to the case of multiplicative STWN by replacing $\xi$ by $f(\theta) \xi$ for $f$ that is sufficiently smooth. Relying on such solution theory, Hairer and Mass \cite{HM12} rigorously verified the numerical inaccuracies observed in \cite{HV11} with an explicit It$\hat{\mathrm{o}}$-Stratonovich type correction term in the case of gradient type nonlinearity with additive STWN; this was followed by Hairer, Mass, and Weber \cite{HMW14} in the case of non-gradient type with multiplicative STWN (also \cite{ZZ17, Y25a}). 

Concerning our proof of global solution theory of \eqref{Burgers'}, some results from \cite{AC15} will need to be extended to our case (see also \cite{GUZ20, L19a}). On one hand, the set up of \cite{AC15} was 2D and the force was $\tilde{\xi}$ that is white only in space so that $\tilde{\xi} \in \mathcal{C}_{x}^{\alpha}$ for $\alpha < - 1$. On the other hand, as can be seen from \eqref{est 5} and \eqref{est 16b}, $\nabla_{\text{symm}} X$ plays the role of $\tilde{\xi}$ here and they have the same regularity. Our Proposition \ref{Proposition 2.4} of independent interests states that the crucial results that we need from \cite{AC15} can indeed be extended to our setting.  

\section{Statement of main results}\label{Section 2}
We define $P_{t} \triangleq e^{\nu \partial_{x}^{2} t}$. Analogously to \eqref{est 14a} we consider 
\begin{equation}\label{Equation of X}
\partial_{t} X = \nu \partial_{x}^{2} X + \Lambda^{\frac{1}{2}} \xi \hspace{1mm} \text{ for } t > 0, \hspace{3mm} X(0,x) = 0. 
\end{equation} 
It follows that 
\begin{equation}\label{est 22}
X \in C( [0, \infty); C^{\gamma}(\mathbb{T})) \text{ for } \gamma < 0 \hspace{1mm} \mathbb{P}\text{-a.s.}
\end{equation} 
Then we define $v \triangleq \theta - X$ so that it solves 
\begin{equation}\label{Equation of v}
\partial_{t} v + \frac{1}{2} \partial_{x} (v+ X)^{2} = \nu \partial_{x}^{2} v + \mathbb{P}_{\neq 0} \zeta \hspace{1mm} \text{ for }\hspace{1mm}  t > 0, \hspace{3mm} v(0,x) = \theta^{\text{in}}(x). 
\end{equation} 
The following local solution theory in a mild formulation (see Definition \ref{Definition 4.1}) is classical and can be proven via the approach of \cite{DD02} similarly to \cite[Theorem 2.3]{HR24} and \cite[Proposition 2.1]{Y23b}.
\begin{proposition}\label{Proposition 2.1}
There exists a null set $\mathcal{N} \subset \Omega$ such that for all $\omega \in \Omega \setminus \mathcal{N}$ and $\kappa > 0$, the following holds. For any $\theta^{\text{in}} \in \mathcal{C}^{-1+ \kappa}(\mathbb{T})$ that is mean-zero, there exists a $T^{\max} (\omega, \theta^{\text{in}}) \in (0, \infty]$ and a unique maximal mild solution $v(\omega)$ to \eqref{Equation of v} on $[0, T^{\max} (\omega, \theta^{\text{in}}))$ such that $v(\omega, 0, x) = \theta^{\text{in}}(x)$. 
\end{proposition} 

We state our first main result. 
\begin{theorem}\label{Theorem 2.2} 
There exists a null set $\mathcal{N}' \subset \Omega$ such that for all $\omega \in\Omega \setminus \mathcal{N}'$, the  following holds. For any $\kappa > 0$ sufficiently small and $\theta^{\text{in}} \in \mathcal{C}^{-1+ \kappa} (\mathbb{T})$ that is mean-zero, $T^{\max}(\omega, \theta^{\text{in}})$ from Proposition \ref{Proposition 2.1} satisfies $T^{\max} (\omega, \theta^{\text{in}}) = \infty$ for all $\omega \in \Omega \setminus \mathcal{N}'$. 
\end{theorem}

We defer the definition of the high-low (HL) weak solution to \eqref{Equation of v} until Definition \ref{Definition 5.1} and state our next result. 
\begin{theorem}\label{Theorem 2.3} 
Consider the same null set $\mathcal{N}' \subset \Omega$ from Theorem \ref{Theorem 2.2}. For every $\omega \in \Omega \setminus \mathcal{N}'$, $\kappa > 0$ sufficiently small, and $\theta^{\text{in}} \in L^{2} (\mathbb{T})$ that is mean-zero, there exists a unique HL weak solution $v$ to \eqref{Equation of v} on $[0,\infty)$. 
\end{theorem} 
The regularity of our HL weak solution in Definition \ref{Definition 5.1} is better than those in \cite[Definition 7.1]{HR24} and \cite[Definition 5.1]{Y23b}; see Remark \ref{Remark 5.1}. 

As we mentioned, the key component of the proofs of Theorems \ref{Theorem 2.2}-\ref{Theorem 2.3} is the extension of  Anderson Hamiltonian in \cite{AC15} to our setting, namely 1D but a spatial derivative of order $\frac{1}{2}$ applied on STWN. For this purpose, we define 
\begin{equation}\label{est 271} 
\mathcal{E}^{-1-\kappa} \triangleq \mathcal{C}^{-1-\kappa}(\mathbb{T}) \times \mathcal{C}^{- 2 \kappa}(\mathbb{T}), 
\end{equation}
which is the special case of \eqref{Define E alpha} with $d = 1$ and ``$\alpha$'' = $-1-\kappa$, and the space of enhanced noise $\mathcal{K}^{-1-\kappa} \subset \mathcal{E}^{-1-\kappa}$ by 
\begin{equation}\label{est 21}
\mathcal{K}^{-1-\kappa} \triangleq \overline{\left\{ \left( \Theta_{1}, \Theta_{1} \circ \left( 1-  \nu \partial_{x}^{2} \right)^{-1} \Theta_{1} - c \right): \hspace{1mm} \Theta_{1} \in C^{\infty} ( \mathbb{T}), c \in \mathbb{R} \right\}}
\end{equation} 
where the closure is taken w.r.t. the $\mathcal{E}^{-1-\kappa}$-topology and $\circ$ indicates the Bony's resonant term (see \eqref{Define K alpha} and \eqref{paraproducts and resonant}). We note that in contrast to $\frac{ \nu \Delta \Id}{2}$ in \eqref{est 5} for the Navier-Stokes equations, we do not have a fraction of $\frac{1}{2}$ in \eqref{est 21}; this difference stems from $\frac{1}{2} \partial_{x} \theta^{2}$ in \eqref{Burgers'} in contrast to $\divergence (u\otimes u)$ in \eqref{Navier-Stokes}. In order to define an operator 
\begin{equation}\label{est 23}
\mathcal{U}(\Theta) \triangleq \nu \partial_{x}^{2} - \Theta_{1} \text{ for any } \Theta = (\Theta_{1}, \Theta_{2}) \in \mathcal{K}^{-1-\kappa} \text{ for some } \kappa  > 0 
\end{equation} 
(cf. \eqref{Define H}), we define the space of strongly paracontrolled distributions 
\begin{subequations}\label{est 24}
\begin{align}
& \chi_{\kappa}(\Theta) \triangleq \{ \phi \in H^{1-\kappa}: \hspace{1mm} \phi^{\sharp} \triangleq \phi + \phi \prec P \in H^{2-2\kappa} \} \hspace{1mm} \text{ where } \hspace{1mm} P \triangleq \left( 1- \nu \partial_{x}^{2} \right)^{-1} \Theta_{1}, \label{est 24a} \\
& \lVert \phi \rVert_{\chi_{\kappa}} \triangleq \lVert \phi \rVert_{H^{1-\kappa}} + \lVert \phi^{\sharp} \rVert_{H^{2-2\kappa}}, \label{est 24b}
\end{align}
\end{subequations}
(cf. \eqref{Define D eta gamma}-\eqref{D eta gamma norm}) where $\prec$ indicates the Bony's paraproduct term (see \eqref{paraproducts and resonant}).

\begin{proposition}\label{Proposition 2.4}
(Cf. \cite[Proposition 6.1]{HR24} and \cite[Proposition 5.3]{Y23b}) Define $\mathcal{K}^{-1-\kappa}$ by \eqref{est 21} and $P$ by \eqref{est 24a}. Then define $\mathcal{K} \triangleq \cup_{0 < \kappa < \kappa_{0}} \mathcal{K}^{-1-\kappa}$ and $C_{\text{op}}$ to be the space of all closed self-adjoint operators with the graph distance where the convergence in this distance is implied by the convergence in the resolvent sense. Then there exist $\kappa_{0} > 0$ and a unique map $\mathcal{U}: \hspace{1mm} \mathcal{K} \mapsto C_{\text{op}}$ such that the following statements hold.
\begin{enumerate}
\item For any $\Theta = (\Theta_{1}, \Theta_{2}) \in \left(C^{\infty} (\mathbb{T})\right)^{2} \cap \mathcal{K}$ and $\phi \in H^{2}(\mathbb{T})$, 
\begin{equation}\label{est 25}
\mathcal{U} (\Theta) \phi = \nu \partial_{x}^{2} \phi - \Theta_{1} \prec \phi - \Theta_{1} \succ \phi - \Theta_{1} \circ \phi^{\sharp} - \phi \prec \Theta_{2} - C^{0} (\phi, P, \Theta_{1})
\end{equation} 
where 
\begin{equation}\label{est 26}
C^{0} ( \phi, P, \Theta_{1}) \triangleq \Theta_{1} \circ ( \phi \prec P) - \phi \prec (P \circ \Theta_{1}).
\end{equation} 
In particular, if $\Theta_{2}= P \circ \Theta_{1}$, then $\mathcal{U} (\Theta) \phi = \nu \partial_{x}^{2} \phi - \Theta_{1} \phi$. 
\item For any $\{\Theta^{n} \}_{n\in\mathbb{N}} \subset \left( C^{\infty}(\mathbb{T}) \right)^{2}$ such that $\Theta^{n} \to \Theta$ in $\mathcal{K}^{-1-\kappa}$ as $n \nearrow + \infty$ for some $\kappa \in (0, \kappa_{0})$ and $\Theta \in \mathcal{K}^{-1-\kappa}$, $\mathcal{U} (\Theta^{n})$ converges to $\mathcal{U}(\Theta)$ in resolvent sense. Moreover, for any $\kappa \in (0, \kappa_{0})$, there exist two continuous maps $\mathbf{m}, \mathbf{c}: \hspace{1mm} \mathcal{K}^{-1-\kappa} \mapsto \mathbb{R}_{+}$ such that 
\begin{equation}\label{est 62}
[\mathbf{m}(\Theta), \infty) \subset \rho( \mathcal{U}(\Theta)) \hspace{3mm} \forall \hspace{1mm} \Theta \in \mathcal{K}^{-1-\kappa} 
\end{equation} 
where $\rho(\mathcal{U}(\Theta))$ is the resolvent set of $\mathcal{U}(\Theta)$ that satisfies for all $\phi \in L^{2} (\mathbb{T})$,  
\begin{equation}\label{est 27}
\lVert (\mathcal{U}(\Theta) + m)^{-1} \phi \rVert_{\chi_{\kappa}} \leq \mathbf{c}(\Theta) \lVert \phi \rVert_{L^{2}} \hspace{3mm} \forall \hspace{1mm} m \geq \mathbf{m}(\Theta). 
\end{equation} 
\end{enumerate} 
\end{proposition} 

\begin{remark}
Part of the motivation for this work was the recent developments on convex integration applied to singular SPDEs. Since the groundbreaking works of  De Lellis and Sz$\acute{\mathrm{e}}$kelyhidi Jr. \cite{DS09, DS13} on the Euler equations that were inspired by \cite{MS03}, many extensions and improvements were made leading particularly to non-uniqueness of weak solutions to the 3D Navier-Stokes equations \cite{BV19a} and the resolution of Onsager's conjecture \cite{I18}; we refer to \cite{BV19b} for further references. 

Starting with \cite{BFH20} by Breit, Feireisl, and Hofmanov$\acute{\mathrm{a}}$ and \cite{CFF19} by Chiodaroli, Feireisl, and Flandoli, the convex integration technique was applied to many SPDEs (e.g. \cite{Y25d} for a brief survey). In relevance to singular SPDEs, we highlight that \cite{HZZ21b} combined the theory of paracontrolled distributions and convex integration technique to construct infinitely many solutions to the 3D Navier-Stokes equations forced by STWN from one initial data; a special novelty therein was that the solution constructed was global in time, extending \cite{ZZ15} that constructed a local solution. Subsequently, the same authors in \cite{HZZ22a} incorporated the convex integration technique of \cite{CKL21} by Cheng, Kwon, and Li on the 2D surface quasi-geostrophic (SQG) equations to the case of random noise that is white-in-space and constructed infinitely many solutions from one initial data (also \cite{HLZZ23}). We also refer to \cite{LZ23} on the 2D Navier-Stokes equations forced by STWN, where the non-uniqueness of \cite{LZ23}   can be compared with the uniqueness results of \cite{HR24} (see \cite[pp. 1--2, Remarks 1.2 and 1.5]{LZ23} for such discussions). The novelty of  \cite{HZZ22a, HLZZ23} is that their constructions are not only global in time but in the locally critical and supercritical cases (recall Section \ref{Section 1.2}). 

Convex integration is more difficult in low spatial dimensions; e.g. the 1D Burgers' equation cannot be covered by \cite{IV15} of Isett and Vicol that demonstrated non-uniqueness of weak solutions to a wide class of active scalars because the velocity field therein needed to be divergence-free. To the best of the author's knowledge, the only application of the convex integration technique to the 1D Burgers' equation was \cite{VK18} by Vo and Kim that constructed $L^{\infty}$-solution to the 1D conservation laws that is nowhere continuous in the interior of its support. However, its proof adapted the partial differential inclusion approach of M$\ddot{u}$ller and $\check{\mathrm{S}}$ver$\acute{a}$k \cite{MS03} which has never been adapted to the stochastic setting and seems unfit for the diffusive or forced case. Instead, Theorems \ref{Theorem 2.2}-\ref{Theorem 2.3} constructed a globally unique solution even when forced by the derivatives of order $\frac{1}{2}$ applied on STWN. 
 \end{remark} 

\begin{remark}
As we described in Section \ref{Section 1.1}, the Burgers' equation forced by a spatial derivative of STWN appears naturally by applying a differentiation operator on the KPZ equation (\hspace{1sp}\cite[Remark 2.2]{HV11}, \cite[Equation (3)]{KPZ86} and \cite[Equation (B.1)]{BG97}. In fact, \cite[p. 899]{HV11} by Hairer and Voss stated ``it will follow from the argument that, if one considers driving noise that is slightly rougher than space-time white noise (taking a noise term equal to $(1 - \partial_{x}^{2})^{\alpha} dw(t)$ with $\alpha \in (0, 1/4)$ still yields a well-posed equation),'' and therefore, Theorems \ref{Theorem 2.2}-\ref{Theorem 2.3} represent precisely the endpoint of this statement. 
\end{remark} 

In Section \ref{Section 3} we set up notations and minimum preliminaries. In Sections \ref{Section 4}-\ref{Section 5}, we present our proofs of Theorems \ref{Theorem 2.2}-\ref{Theorem 2.3} assuming Proposition \ref{Proposition 2.4}. In Section \ref{Section 6}, we point out how the key results in \cite{AC15} apply to our setting under minimum modifications and thereby verify Proposition \ref{Proposition 2.4}. In the Appendices \ref{Appendix A} and \ref{Appendix B}, we leave further preliminaries and detailed computations for completeness, respectively. 

\section{Preliminaries}\label{Section 3} 
We set further notations needed for our proofs. We write $\lVert f \rVert_{C_{t,x}} \triangleq \sup_{s \in [0,t], x \in \mathbb{T}} \lvert f(t,x) \rvert$. 
\subsection{Besov spaces and Bony's paraproducts}\label{Subsection 3.2} 
We let $\chi$ and $\rho$ be smooth functions with compact support on $\mathbb{R}^{d}$ that are non-negative, and radial such that the support of $\chi$ is contained in a ball while that of $\rho$ in an annulus and 
\begin{align*}
& \chi(\xi) + \sum_{j\geq 0} \rho(2^{-j} \xi) = 1 \hspace{3mm} \forall \hspace{1mm} \xi \in \mathbb{R}^{d}, \\
&\supp (\chi) \cap  \supp (\rho(2^{-j} \cdot )) = \emptyset \hspace{1mm} \forall \hspace{1mm} j \in\mathbb{N}, \hspace{2mm} \supp (\rho(2^{-i}\cdot)) \cap \supp (\rho (2^{-j} \cdot)) = \emptyset \hspace{1mm} \text{ if } \lvert i-j \rvert > 1.
\end{align*}
We define $\rho_{j}(\cdot) \triangleq \rho(2^{-j} \cdot)$ and the Littlewood-Paley operators $\Delta_{j}$ for $j \in \mathbb{N}_{0} \cup \{-1\}$ by 
\begin{equation} 
\Delta_{j} f \triangleq 
\begin{cases}
\mathcal{F}^{-1} (\chi) \ast f & \text{ if } j = -1, \\
\mathcal{F}^{-1} (\rho_{j}) \ast f & \text{ if } j \in \mathbb{N}_{0},
\end{cases}
\end{equation}
and inhomogeneous Besov spaces $B_{p,q}^{s} \triangleq \{f \in \mathcal{S}': \hspace{1mm} \lVert f \rVert_{B_{p,q}^{s}} < \infty \}$ where 
\begin{equation}
\lVert f \rVert_{B_{p,q}^{s}} \triangleq  \lVert 2^{sm} \lVert \Delta_{m} f \rVert_{L_{x}^{p}} \rVert_{l_{m}^{q}} \hspace{3mm} \forall \hspace{1mm} p, q \in [1,\infty], s \in \mathbb{R}. 
\end{equation} 
We define the low-frequency cut-off operator $S_{i} f \triangleq \sum_{-1  \leq j \leq i-1} \Delta_{j} f$ and Bony's paraproducts and resonant  respectively as 
\begin{equation}\label{paraproducts and resonant}
f \prec g  \triangleq \sum_{i\geq -1} S_{i-1} f \Delta_{i} g   \text{ and }  f  \circ g  \triangleq \sum_{i\geq -1} \sum_{j: \lvert j\rvert \leq 1} \Delta_{i} f  \Delta_{i+j} g  
\end{equation}
so that $f g  = f  \prec g  + f  \succ g  + f  \circ g$, where $f  \succ g  = g  \prec f $ (see \cite[Sections 2.6.1 and 2.8.1]{BCD11}). We recall from \cite{BCD11} that there exist $N_{1}, N_{2} \in \mathbb{N}$ such that 
\begin{equation}\label{est 0} 
\Delta_{m} (f \prec g) = \sum_{j: \lvert j-m \rvert \leq N_{1}} (S_{j-1} f) \Delta_{j} g \hspace{2mm} \text{ and } \hspace{2mm} \Delta_{m} ( f \circ g) = \Delta_{m}\sum_{i\geq m-N_{2}}  \sum_{j: \lvert j \rvert \leq 1} \Delta_{i} f \Delta_{i+j} g. 
\end{equation} 
W record some special cases of Bony's estimates. 
\begin{lemma}\label{Lemma 3.1} 
\rm{(\hspace{1sp}\cite[Proposition 3.1]{AC15} and \cite[Lemma A.1]{HR24})} Let $\alpha, \beta \in \mathbb{R}$. 
\begin{enumerate} 
\item  Then 
\begin{subequations}\label{Bony in Sobolev}
\begin{align}
& \lVert f \prec g \rVert_{H^{\beta -\alpha}} \lesssim_{\alpha, \beta} \lVert f \rVert_{L^{2}} \lVert g \rVert_{\mathcal{C}^{\beta}}  \hspace{8mm}  \forall \hspace{1mm} f \in L^{2}, g \in \mathcal{C}^{\beta} \text{ if } \alpha > 0, \label{Bony 1} \\
& \lVert f \succ g \rVert_{H^{\alpha}} \lesssim_{\alpha} \lVert f \rVert_{H^{\alpha}} \lVert g \rVert_{L^{\infty}}  \hspace{11mm}  \forall \hspace{1mm} f \in H^{\alpha}, g \in L^{\infty},  \label{Bony 2} \\
& \lVert f \prec g \rVert_{H^{\alpha + \beta}} \lesssim_{\alpha, \beta} \lVert f \rVert_{H^{\alpha}} \lVert g \rVert_{\mathcal{C}^{\beta}} \hspace{7mm}  \forall \hspace{1mm} f \in H^{\alpha}, g \in \mathcal{C}^{\beta}  \text{ if } \alpha < 0,  \label{Bony 3}  \\
& \lVert f \succ g \rVert_{H^{\alpha + \beta}} \lesssim_{\alpha,\beta} \lVert f \rVert_{H^{\alpha}}  \lVert g \rVert_{\mathcal{C}^{\beta}}  \hspace{7mm}  \forall \hspace{1mm} f \in H^{\alpha}, g \in \mathcal{C}^{\beta} \text{ if } \beta < 0, \label{Bony 4} \\
& \lVert f \circ g \rVert_{H^{\alpha+ \beta}} \lesssim_{\alpha, \beta} \lVert f \rVert_{H^{\alpha}} \lVert g \rVert_{\mathcal{C}^{\beta}}  \hspace{8mm} \forall \hspace{1mm} f \in H^{\alpha}, g \in \mathcal{C}^{\beta}  \text{ if } \alpha + \beta > 0. \label{Bony 5} 
\end{align} 
\end{subequations} 
\item Let $p, q \in [1,\infty]$ such that $\frac{1}{r} = \frac{1}{p} + \frac{1}{q} \leq 1$. Then 
\begin{subequations}\label{Bony in Besov}
\begin{align}
& \lVert f \prec g \rVert_{B_{r, \infty}^{\alpha}} \lesssim_{\alpha,p,q,r} \lVert f \rVert_{L^{p}} \lVert g \rVert_{B_{q, \infty}^{\alpha}} \hspace{12mm} \forall \hspace{1mm} f \in L^{p}, g \in B_{q,\infty}^{\alpha}, \label{Bony 6}\\
& \lVert f \prec g \rVert_{B_{r,\infty}^{\alpha + \beta}} \lesssim_{\alpha, \beta, p, q, r} \lVert f \rVert_{B_{p,\infty}^{\beta}} \lVert g \rVert_{B_{q,\infty}^{\alpha}} \hspace{7mm} \forall \hspace{1mm} f \in B_{p,\infty}^{\beta}, g \in B_{q,\infty}^{\alpha} \text{ if } \beta < 0, \label{Bony 7}\\
& \lVert f \circ g \rVert_{B_{r,\infty}^{\alpha+ \beta}} \lesssim_{\alpha,\beta,p,q,r} \lVert f \rVert_{B_{p,\infty}^{\beta}} \lVert g \rVert_{B_{q,\infty}^{\alpha}} \hspace{8mm} \forall \hspace{1mm} f \in B_{p,\infty}^{\beta}, g \in B_{q,\infty}^{\alpha} \text{ if } \alpha + \beta > 0.\label{Bony 8}
\end{align}
\end{subequations}
\end{enumerate} 
One of the consequences is that $\lVert fg \rVert_{\mathcal{C}^{\min\{\alpha,\beta\}}} \lesssim \lVert f \rVert_{\mathcal{C}^{\alpha}} \lVert g \rVert_{\mathcal{C}^{\beta}}$ when $\alpha + \beta > 0$.  
\end{lemma} 
Next, we recall the useful product estimate in Sobolev spaces that was useful in the study of the 2D SQG equations (e.g. see \cite[p. 2818]{MX11}) and the 2D MHD system. 
\begin{lemma}\label{Product estimate lemma}
Let $d \in \mathbb{N}$ and $\sigma_{1}, \sigma_{2} < \frac{d}{2}$ satisfy $\sigma_{1} + \sigma_{2} > 0$. Then 
\begin{equation}\label{Product estimate}
\lVert fg \rVert_{\dot{H}^{\sigma_{1} + \sigma_{2} - \frac{d}{2}} (\mathbb{T}^{d})} \lesssim_{\sigma_{1}, \sigma_{2}} \lVert f \rVert_{\dot{H}^{\sigma_{1}} (\mathbb{T}^{d})} \lVert g \rVert_{\dot{H}^{\sigma_{2}} (\mathbb{T}^{d})} \hspace{3mm} \forall \hspace{1mm} f \in \dot{H}^{\sigma_{1}} (\mathbb{T}^{d}), g \in \dot{H}^{\sigma_{2}} (\mathbb{T}^{d}). 
\end{equation} 
\end{lemma} 
The statement and the proof of Lemma \ref{Product estimate lemma} in case $d= 2$ can be found in \cite[Lemma 2.5 and Appendix A.2]{Y14b}; as we could not locate the general statement for $d \in \mathbb{N}$, we sketch its proof in Section \ref{Section B.1} for completeness.  

Finally, we recall the following notations and lemma from \cite{HR24, Y23b}. 
\begin{define}\label{Definition 3.1} 
Let $\mathfrak{h}: \hspace{1mm} [0,\infty) \mapsto [0,\infty)$ be a smooth function such that 
\begin{equation}
\mathfrak{h}(r) \triangleq  
\begin{cases}
1 & \text{ if } r \geq 1, \\
0 & \text{ if } r \leq \frac{1}{2}, 
\end{cases} 
\hspace{5mm} \mathfrak{l} \triangleq 1- \mathfrak{h}.
\end{equation} 
Then, we consider for any $\lambda > 0$
\begin{equation}
\check{\mathfrak{h}}_{\lambda}(x) \triangleq \mathcal{F}^{-1} \left( \mathfrak{h} \left( \frac{ \lvert \cdot \rvert}{\lambda} \right) \right) (x), \hspace{5mm} \check{\mathfrak{l}}_{\lambda}(x) \triangleq \mathcal{F}^{-1} \left( \mathfrak{l} \left(\frac{ \lvert \cdot \rvert}{\lambda} \right) \right)(x), 
\end{equation} 
and then the projections onto higher and lower frequencies respectively by 
\begin{equation}\label{est 3}
\mathcal{H}_{\lambda}: \hspace{1mm} \mathcal{S}' \mapsto \mathcal{S}' \hspace{1mm} \text{ by } \hspace{1mm}  \mathcal{H}_{\lambda} f \triangleq \check{\mathfrak{h}}_{\lambda} \ast f \hspace{1mm} \text{ and } \hspace{1mm}  \mathcal{L}_{\lambda}: \hspace{1mm} \mathcal{S}' \mapsto \mathcal{S} \hspace{1mm}  \text{ by } \hspace{1mm} \mathcal{L}_{\lambda} f \triangleq f - \mathcal{H}_{\lambda} f = \check{\mathfrak{l}}_{\lambda} \ast f.  
\end{equation} 
\end{define} 
\begin{lemma}\label{Lemma 3.3} 
\rm{(Cf. \cite[Lemmas 4.2-4.3]{HR24} and \cite[Lemma 3.3]{Y23b})} For any $p, q \in [1,\infty]$, and $\alpha, \beta \in \mathbb{R}$ such that $\beta \geq \alpha$, 
\begin{equation}\label{Bernstein}
\lVert \mathcal{L}_{\lambda} f \rVert_{B_{p,q}^{\beta}} \lesssim \lambda^{\beta - \alpha} \lVert f \rVert_{B_{p,q}^{\alpha}} \hspace{2mm} \forall \hspace{1mm} f \in B_{p,q}^{\alpha} \text{ and } \lVert \mathcal{H}_{\lambda} f \rVert_{B_{p,q}^{\alpha}} \lesssim \lambda^{\alpha - \beta} \lVert f \rVert_{B_{p,q}^{\beta}} \hspace{2mm} \forall \hspace{1mm} f \in B_{p,q}^{\beta}.
\end{equation} 
\end{lemma} 

\section{Proof of Theorem \ref{Theorem 2.2}}\label{Section 4}
We define $w \triangleq v - Y$ where $Y$ solves 
\begin{equation}\label{Equation of Y} 
\partial_{t} Y + \frac{1}{2} \partial_{x} (2Y X + X^{2}) = \nu \partial_{x}^{2} Y + \mathbb{P}_{\neq 0} \zeta \hspace{1mm} \text{ for } \hspace{1mm} t > 0, \hspace{3mm} Y(0,x) = 0, 
\end{equation} 
so that $w$ solves 
\begin{equation}\label{Equation of w} 
\partial_{t} w + \frac{1}{2} \partial_{x} (w^{2} + 2wY + 2wX + Y^{2}) = \nu \partial_{x}^{2} w \hspace{1mm} \text{ for } \hspace{1mm} t > 0, \hspace{3mm} w(0,x) = \theta^{\text{in}}(x). 
\end{equation} 
Because $\zeta \in \mathcal{C}_{x}^{-2+ 3 \kappa}$ from \eqref{define zeta}, we expect $Y \in \mathcal{C}_{x}^{2\kappa}$. Taking $L^{2}(\mathbb{T})$-inner products with $w$ in \eqref{Equation of w}, we see that the potentially ill-defined terms are 
\begin{align*}
\int_{\mathbb{T}}w \left( - \nu \partial_{x}^{2} w + \frac{1}{2} \partial_{x} (2wX) \right)dx. 
\end{align*}

\begin{define}\label{Definition 4.1}
Recall $P_{t} = e^{\nu \partial_{x}^{2} t}$ from Section \ref{Section 2}. For any $\gamma > 0, T > 0$, and $\delta \in \mathbb{R}$, we define 
\begin{subequations}
\begin{align}
M_{T}^{\gamma} \mathcal{C}_{x}^{\delta} \triangleq& \{ f: \hspace{1mm} t \mapsto t^{\gamma} \lVert f(t) \rVert_{\mathcal{C}_{x}^{\delta}} \text{ is continuous over } [0,T], \lVert f \rVert_{\mathcal{M}_{T}^{\gamma} \mathcal{C}_{x}^{\delta}} < \infty \}\\
& \text{ where } \lVert f \rVert_{\mathcal{M}_{T}^{\gamma} \mathcal{C}_{x}^{\delta}} \triangleq \left\lVert t^{\gamma} \lVert f(t) \rVert_{\mathcal{C}_{x}^{\delta}} \right\rVert_{C_{T}}. 
\end{align} 
\end{subequations} 
Then $w \in \mathcal{M}_{T}^{\gamma} \mathcal{C}_{x}^{\delta}$ is a mild solution to \eqref{Equation of w} over $[0,T]$ if 
\begin{equation}
w(t) = P_{t} \theta^{\text{in}} - \int_{0}^{t} P_{t-s} \frac{1}{2} \partial_{x} (w^{2} + 2wY + 2wX + Y^{2})(s) ds. 
\end{equation} 
\end{define} 
For any $\lambda \geq 1$ and $t \in [0,\infty)$, we define our enhanced noise (recall \eqref{est 21}) by 
\begin{equation}\label{est 64}
t \mapsto ( \partial_{x} \mathcal{L}_{\lambda} X(t), \partial_{x} \mathcal{L}_{\lambda} X(t) \circ P^{\lambda}(t) - r_{\lambda}(t)), 
\end{equation} 
where 
\begin{subequations}\label{est 63}
\begin{align}
&P^{\lambda}(t,x) \triangleq \left(1- \nu \partial_{x}^{2} \right)^{-1} \partial_{x} \mathcal{L}_{\lambda} X(t,x)  \label{est 63a}\\
& r_{\lambda}(t) \triangleq \sum_{k\in \mathbb{Z} \setminus \{0\}} \mathfrak{l} \left( \frac{ \lvert k \rvert}{\lambda } \right)^{2} \left( \frac{1- e^{-2 \nu \lvert k \rvert^{2} t}}{2\nu} \right) \left(1+  \nu \lvert k \rvert^{2}  \right)^{-1} \lvert k \rvert.  \label{est 63b}
\end{align}
\end{subequations} 
For any $t \in [0, \infty)$, any $\kappa > 0$, and $\{ \lambda^{i} \}_{i\in\mathbb{N}}$ to be defined in Definition \ref{Definition 4.2}, we define 
\begin{equation}\label{Define Lt and Nt}
L_{t}^{\kappa} \triangleq 1 + \lVert X \rVert_{C_{t} \mathcal{C}_{x}^{-\kappa}} + \lVert Y \rVert_{C_{t} \mathcal{C}_{x}^{2\kappa}} \hspace{1mm} \text{ and } \hspace{1mm} N_{t}^{\kappa} \triangleq L_{t}^{\kappa} + \sup_{i \in \mathbb{N}} \lVert ( \partial_{x} \mathcal{L}_{\lambda^{i}} X) \circ P^{\lambda^{i}} - r_{\lambda^{i}} \rVert_{C_{t}\mathcal{C}_{x}^{-2\kappa}}. 
\end{equation} 
Here, $\lVert X \rVert_{C_{t} \mathcal{C}_{x}^{-\kappa}} + \sup_{i \in \mathbb{N}} \lVert ( \partial_{x} \mathcal{L}_{\lambda^{i}} X) \circ P^{\lambda^{i}} - r_{\lambda^{i}} \rVert_{C_{t}\mathcal{C}_{x}^{-2\kappa}}$ within $N_{t}^{\kappa}$ formally bounds the $\mathcal{E}^{-1-\kappa}$-norm of $(\partial_{x} \mathcal{L}_{\lambda^{i}} X,  (\partial_{x} \mathcal{L}_{\lambda^{i}} X) \circ P^{\lambda^{i}} - r_{\lambda^{i}}) \in \mathcal{K}^{-1-\kappa}$ for all $i \in \mathbb{N}$ (recall \eqref{est 271}-\eqref{est 21}).   

The following is a consequence of Proposition \ref{Proposition 4.13}.  
\begin{proposition}\label{Proposition 4.1} 
Let $(\Omega, \mathcal{F}, \mathbb{P})$ be a probability space on which the STWN $\xi$ satisfies \eqref{STWN}. Then there exists a null set $\mathcal{N}'' \subset \Omega$ such that $N_{t}^{\kappa} (\omega) < \infty$ for all $\omega \in \Omega \setminus \mathcal{N}''$ for all $t \geq 0$ and all $\kappa > 0$. 
\end{proposition} 
The following local solution theory can be proven via the approach from \cite{DD02}. 
\begin{proposition}\label{Proposition 4.2}
Fix $\kappa \in (0, \frac{2}{5})$ and then $\gamma = 1 - \frac{\kappa}{2} > 0$. Suppose that $X \in C( [0,\infty); \mathcal{C}^{-\kappa}(\mathbb{T}))$ and $Y \in C([0,\infty); \mathcal{C}^{2\kappa} (\mathbb{T}))$ $\mathbb{P}$-a.s. Then, for all $\theta^{\text{in}} \in \mathcal{C}^{-1+ 2 \kappa} (\mathbb{T})$ that is mean-zero, \eqref{Equation of w} has a unique mild solution $w \in \mathcal{M}_{T^{\max}}^{\frac{\gamma}{2}} \mathcal{C}_{x}^{\frac{3\kappa}{2}}$ where $T^{\max} ( \{ L_{t}^{\kappa} \}_{t\geq 0}, \theta^{\text{in}} ) \in (0,\infty]$. 
\end{proposition} 
\begin{assume}\label{Assumption 4.3}
Because Proposition \ref{Proposition 4.2} guarantees that $w \in \mathcal{M}_{T^{\max}}^{\frac{\gamma}{2}} \mathcal{C}_{x}^{\frac{3\kappa}{2}}$ for sufficiently small $\kappa > 0$ and $\gamma = 1 - \frac{\kappa}{2} > 0$, we assume hereafter that $\theta^{\text{in}} \in L^{2} (\mathbb{T})$ is mean-zero. 
\end{assume} 

Next, we define $Q$ to solve 
\begin{equation}\label{Define Q}
(\partial_{t} - \nu \partial_{x}^{2}) Q = 2X, \hspace{3mm} Q(0) = 0.
\end{equation} 
As $X \in C_{t} \mathcal{C}_{x}^{-\kappa}$ from Proposition \ref{Proposition 4.1} and \eqref{Define Lt and Nt}, we see that 
\begin{equation}\label{Estimate on Q}
\lVert Q(t) \rVert_{\mathcal{C}_{x}^{\gamma}}  \lesssim \lVert X \rVert_{C_{t} \mathcal{C}_{x}^{-\kappa}} t^{1- \frac{\gamma+ \kappa}{2}} < \infty \text{ for any } \gamma < 2 - \kappa. 
\end{equation}  
We now define $w^{\sharp}$ by 
\begin{equation}\label{Define w sharp}
w = - \frac{1}{2} \partial_{x} (w \prec Q) + w^{\sharp}. 
\end{equation} 
It follows from \eqref{Define Q} and \eqref{Equation of w} that 
\begin{equation}\label{Equation of w sharp}
\partial_{t} w^{\sharp} + \frac{1}{2} \partial_{x} (w^{2} + 2wY + 2wX - 2w \prec X - C^{\prec} (w, Q) + Y^{2}) = \nu \partial_{x}^{2} w^{\sharp} 
\end{equation} 
if we define 
\begin{subequations}\label{Define commutator} 
\begin{align}
C^{\prec} (w, Q) \triangleq& \partial_{t} (w \prec Q) - \nu \partial_{x}^{2} (w \prec Q) - w \prec \partial_{t} Q + \nu w \prec \partial_{x}^{2} Q \label{Define commutator a} \\
=& (\partial_{t} - \nu \partial_{x}^{2}) (w \prec Q) - w \prec ( \partial_{t} - \nu \partial_{x}^{2}) Q. \label{Define commutator b}
\end{align}
\end{subequations} 
Additionally, we define 
\begin{equation}\label{Define QH, wH, wL}
Q^{\mathcal{H}}(t) \triangleq \mathcal{H}_{\lambda_{t}}Q(t), \hspace{1mm} w^{\mathcal{H}}(t) \triangleq - \frac{1}{2} \partial_{x} (w \prec Q^{\mathcal{H}})(t), \hspace{1mm} w^{\mathcal{L}}(t) \triangleq w(t) - w^{\mathcal{H}}(t). 
\end{equation} 

\begin{define}\label{Definition 4.2} 
Fix any $\tau > 0$ and initial data $\theta^{\text{in}} \in L^{2} (\mathbb{T})$ that is mean-zero. Define a family of stopping times $\{T_{i} \}_{i \in \mathbb{N}_{0}}$ by 
\begin{equation}\label{Define Ti}
T_{0} \triangleq 0, \hspace{3mm} T_{i+1} (\omega, \theta^{\text{in}}) \triangleq \inf\{t \geq T_{i}: \hspace{1mm} \lVert w(t) \rVert_{L^{2}} \geq i + 1 \}  \wedge T^{\max} (\omega, \theta^{\text{in}}) 
\end{equation} 
with $T^{\max} (\omega, \theta^{\text{in}})$ from Proposition \ref{Proposition 2.1}. Set 
\begin{equation}\label{Define i0}
i_{0} (\theta^{\text{in}}) \triangleq \max\{ i \in \mathbb{N}_{0}: \hspace{1mm} i \leq \lVert \theta^{\text{in}} \rVert_{L^{2}} \} 
\end{equation} 
so that $T_{i} = 0$ if and only if $i \leq i_{0} (\theta^{\text{in}})$. Set 
\begin{equation}\label{est 29}
\lambda^{i} \triangleq (i+1)^{\tau} \hspace{1mm} \text{ and } \hspace{1mm} \lambda_{t} \triangleq 
\begin{cases}
(1+ \lceil \lVert \theta^{\text{in}} \rVert_{L^{2}} \rceil )^{\tau} & \text{ if } t = 0, \\
(1+ \lVert w(T_{i}) \rVert_{L^{2}})^{\tau} & \text{ if } t > 0 \text{ such that } t \in [T_{i}, T_{i+1}). 
\end{cases} 
\end{equation}  
As $\theta^{\text{in}} \in L^{2}(\mathbb{T})$, we have $i_{0} (\theta^{\text{in}}) < \infty$. Finally, $\lambda_{t} = \lambda^{i}$ for all $t \in [T_{i}, T_{i+1})$ such that $i > i_{0} (\theta^{\text{in}})$. 
\end{define}

\begin{proposition}\label{Proposition 4.4} 
Fix any $\kappa > 0, \tau > 0$ from Definition \ref{Definition 4.2}, $\mathcal{N}$ from Proposition \ref{Proposition 2.1}, $\mathcal{N}''$ from Proposition \ref{Proposition 4.1}, and define $N_{t}^{\kappa}$ from \eqref{Define Lt and Nt}. Then, for any $\delta \geq 0$ and $\omega \in \Omega \setminus ( \mathcal{N} \cup \mathcal{N} '')$, there exists a constant $C(\delta) > 0$ such that $w^{\mathcal{H}}$ satisfies 
\begin{equation}\label{Higher frequency estimate} 
\lVert w^{\mathcal{H}}(t, \omega) \rVert_{H^{1- 2 \kappa - \delta}} \leq C(\delta) (1+\lVert w(t,\omega) \rVert_{L^{2}})^{1- \tau \delta} N_{t}^{\kappa}(\omega) t^{\frac{\kappa}{4}} \hspace{3mm} \forall \hspace{1mm} t \in [0, T^{\max} (\omega, \theta^{\text{in}})). 
\end{equation} 
\end{proposition}

\begin{proof}[Proof of Proposition \ref{Proposition 4.4}]
We can compute starting from \eqref{Define QH, wH, wL}, 
\begin{align}
\lVert w^{\mathcal{H}}(t) \rVert_{H^{1- 2 \kappa - \delta}} \overset{\eqref{Bony 1}\eqref{Bernstein}}{\lesssim} \lVert w(t) \rVert_{L^{2}} \lambda_{t}^{-\delta} \lVert Q(t) \rVert_{\mathcal{C}^{2- \frac{3\kappa}{2}}}  \overset{\eqref{Estimate on Q} \eqref{Define QH, wH, wL}}{\lesssim} \lVert w(t) \rVert_{L^{2}} \lambda_{t}^{-\delta} N_{t}^{\kappa} t^{\frac{\kappa}{4}} \label{est 30}
\end{align}
where we can additionally bound $\lambda_{t}^{-\delta} \leq C(\delta) (1+ \lVert w(t) \rVert_{L^{2}})^{-\tau \delta}$ so that \eqref{Higher frequency estimate} follows. 
\end{proof} 

We fix $i \in \mathbb{N}$ such that $i > i_{0} ( \theta^{\text{in}})$, $t \in [T_{i}, T_{i+1})$, and compute using \eqref{Define QH, wH, wL}, \eqref{Equation of w} and \eqref{Define commutator}, 
\begin{equation}\label{est 32}
\partial_{t} w^{\mathcal{L}} + \frac{1}{2} \partial_{x} \left(w^{2} + 2wY + 2wX - 2 w \prec \mathcal{H}_{\lambda_{t}}X - C^{\prec} (w, Q^{\mathcal{H}}) + Y^{2} \right) = \nu \partial_{x}^{2} w^{\mathcal{L}}. 
\end{equation} 
Taking $L^{2}(\mathbb{T})$-inner products on \eqref{est 32} with $w^{\mathcal{L}}$ leads to 
\begin{equation}\label{est 61}
\partial_{t} \lVert w^{\mathcal{L}}(t) \rVert_{L^{2}}^{2} = \sum_{k=1}^{4} \RomanI_{k}
\end{equation} 
where 
\begin{subequations}\label{est 33} 
\begin{align}
\RomanI_{1} \triangleq& 2 \langle w^{\mathcal{L}}, \nu \partial_{x}^{2} w^{\mathcal{L}} -  \partial_{x} (w^{\mathcal{L}} \mathcal{L}_{\lambda_{t}}X) \rangle_{L^{2}} (t),  \label{Define I1} \\
\RomanI_{2} \triangleq& - \langle w^{\mathcal{L}}, \partial_{x} (2 w^{\mathcal{L}} \mathcal{H}_{\lambda_{t}}X - 2 w^{\mathcal{L}} \prec \mathcal{H}_{\lambda_{t}} X ) \rangle_{L^{2}} (t),  \label{Define I2}\\
\RomanI_{3} \triangleq& - \langle w^{\mathcal{L}}, \partial_{x} (2 w^{\mathcal{H}} X - 2 w^{\mathcal{H}} \prec \mathcal{H}_{\lambda_{t}}X) \rangle_{L^{2}} (t), \label{Define I3}\\
\RomanI_{4} \triangleq& - \langle w^{\mathcal{L}}, \partial_{x} (w^{2} + 2wY - C^{\prec} (w, Q^{\mathcal{H}}) + Y^{2} \rangle_{L^{2}} (t). \label{Define I4}
\end{align}
\end{subequations} 
For $\RomanI_{1}$, although we do not have divergence-free property in contrast to the case of the Navier-Stokes equations or the MHD system, we can write  
\begin{align*}
 - 2 \int_{\mathbb{T}} w^{\mathcal{L}} \partial_{x} (w^{\mathcal{L}} \mathcal{L}_{\lambda_{t}} X) dx = 2 \int_{\mathbb{T}} \partial_{x} w^{\mathcal{L}} w^{\mathcal{L}} \mathcal{L}_{\lambda_{t}}X dx  = - \int_{\mathbb{T}} (w^{\mathcal{L}})^{2} \partial_{x} \mathcal{L}_{\lambda_{t}} X dx. 
\end{align*} 
Therefore, we can define a time-dependent family of operators 
\begin{equation}\label{Define At}
\mathcal{A}_{t} \triangleq [ \nu \partial_{x}^{2} - \partial_{x} X(t) ] - \infty \hspace{3mm} \forall \hspace{1mm} t \geq 0 
\end{equation} 
as the limit $\lambda\nearrow + \infty$ of 
\begin{equation}\label{Define At lambda}
\mathcal{A}_{t}^{\lambda} \triangleq [ \nu \partial_{x}^{2} - \partial_{x} \mathcal{L}_{\lambda} X(t) ] - r_{\lambda}(t), 
\end{equation} 
and write 
\begin{equation}\label{est 60}
\RomanI_{1} = - \nu \lVert w^{\mathcal{L}}(t) \rVert_{\dot{H}^{1}}^{2} + \langle w^{\mathcal{L}}, \mathcal{A}_{t}^{\lambda_{t}} w^{\mathcal{L}} \rangle_{L^{2}} (t) + r_{\lambda}(t) \lVert w^{\mathcal{L}}(t) \rVert_{L^{2}}^{2}. 
\end{equation} 
We now estimate $\RomanI_{2}, \RomanI_{3}$, and $\RomanI_{4}$. 
\begin{proposition}\label{Proposition 4.5}
Let $t \in [T_{i}, T_{i+1})$ and fix $\lambda_{t}$ from \eqref{est 29} with $\tau \in [1, \infty)$. Then, for any $\kappa_{0} \in (0,1), \eta \in [ \frac{1+ \kappa_{0}}{2}, 1)$ and all $\kappa \in (0, \kappa_{0}]$, $\RomanI_{2}$ from \eqref{Define I2} satisfies 
\begin{equation}\label{est 272}
\lvert \RomanI_{2} \rvert \lesssim \lVert w^{\mathcal{L}}(t) \rVert_{\dot{H}^{\eta}} N_{t}^{\kappa}. 
\end{equation} 
\end{proposition} 

\begin{proof}[Proof of Proposition \ref{Proposition 4.5}]
Starting from \eqref{Define I2}, we rewrite 
\begin{equation*}
w^{\mathcal{L}} \mathcal{H}_{\lambda_{t}} X - w^{\mathcal{L}} \prec \mathcal{H}_{\lambda_{t}}X = w^{\mathcal{L}} \succ \mathcal{H}_{\lambda_{t}} X + w^{\mathcal{L}} \circ \mathcal{H}_{\lambda_{t}} X 
\end{equation*} 
and estimate using \eqref{Bony 4} and \eqref{Bony 5}, 
\begin{equation}
\lvert \RomanI_{2} \rvert \lesssim \lVert w^{\mathcal{L}}(t) \rVert_{\dot{H}^{\eta}} \lVert w^{\mathcal{L}} (t) \rVert_{\dot{H}^{1- \eta + \kappa}} \lVert \mathcal{H}_{\lambda_{t}} X(t) \rVert_{\mathcal{C}^{-\kappa}}  \overset{\eqref{Define Lt and Nt}}{\lesssim} \lVert w^{\mathcal{L}}(t) \rVert_{\dot{H}^{\eta}}^{2} N_{t}^{\kappa}. 
\end{equation} 
\end{proof} 

\begin{proposition}\label{Proposition 4.6}
Let $t \in [T_{i}, T_{i+1})$ and fix $\lambda_{t}$ from \eqref{est 29}. 
\begin{enumerate}
\item Let $\tau \in [2,\infty), \kappa_{0} \in (0, \frac{1}{6})$ and $\eta \in [\frac{1}{\tau} + 3 \kappa_{0}, 1)$. Then, for all $\kappa \in (0, \kappa_{0}]$, $\RomanI_{3}$ from \eqref{Define I3} satisfies 
\begin{equation}\label{est 38}
\lvert \RomanI_{3} \rvert \lesssim \lVert w^{\mathcal{L}}(t) \rVert_{\dot{H}^{\eta}} (N_{t}^{\kappa})^{2} \lambda_{t}^{1- \eta + 2 \kappa}. 
\end{equation} 
\item Let $\tau \in [\frac{25}{12}, \infty)$, $\kappa_{0} \in (0, \frac{1}{150})$, and $\eta \in (\frac{1}{2}, 1)$. Then, for all $\kappa \in (0, \kappa_{0}]$, $- \langle w^{\mathcal{L}}, \partial_{x} w^{2} \rangle_{L^{2}}$ of $\RomanI_{4}$ from \eqref{Define I4} satisfies 
\begin{equation}\label{est 39}
\lvert \langle w^{\mathcal{L}}, \partial_{x} w^{2} \rangle_{L^{2}} (t) \rvert \lesssim \lVert w^{\mathcal{L}}(t) \rVert_{\dot{H}^{\eta}} ( \lVert w^{\mathcal{L}}(t) \rVert_{\dot{H}^{\eta}} + N_{t}^{\kappa}) N_{t}^{\kappa}. 
\end{equation}  
\end{enumerate} 
\end{proposition} 

\begin{proof}[Proof of Proposition \ref{Proposition 4.6}]
We first decompose 
\begin{equation}\label{est 37}
\lvert \RomanI_{3} \rvert \leq \RomanI_{31} + \RomanI_{32} 
\end{equation} 
where 
\begin{equation}\label{Define I31 and I32}
\RomanI_{31} \triangleq 2 \lvert \langle w^{\mathcal{L}}, \partial_{x} (w^{\mathcal{H}} \mathcal{L}_{\lambda_{t}} X) \rangle_{L^{2}}(t) \rvert  \text{ and }  \RomanI_{32} \triangleq 2 \lvert \langle w^{\mathcal{L}}, \partial_{x} (w^{\mathcal{H}} \succ \mathcal{H}_{\lambda_{t}} X + w^{\mathcal{H}} \circ \mathcal{H}_{\lambda_{t}} X) \rangle_{L^{2}} (t) \rvert. 
\end{equation} 
For $\RomanI_{31}$, we first estimate 
\begin{subequations}\label{est 34} 
\begin{align}
&  \lVert w^{\mathcal{H}} \prec \mathcal{L}_{\lambda_{t}} X (t) \rVert_{\dot{H}^{1-\eta}} + \lVert w^{\mathcal{H}} \circ \mathcal{L}_{\lambda_{t}}X(t) \rVert_{\dot{H}^{1-\eta}}  \nonumber \\
& \hspace{17mm}  \overset{\eqref{Bony 3}\eqref{Bony 5}}{\lesssim} \lVert w^{\mathcal{H}}(t) \rVert_{H^{-\kappa}} \lVert \mathcal{L}_{\lambda_{t}} X(t) \rVert_{\mathcal{C}^{1+ \kappa - \eta}} 
\overset{\eqref{Bernstein} \eqref{Higher frequency estimate}\eqref{Define Lt and Nt}}{\lesssim}  \lambda_{t}^{1+ 2 \kappa - \eta} (N_{t}^{\kappa})^{2}, \\
&  \lVert w^{\mathcal{H}} \succ \mathcal{L}_{\lambda_{t}} X(t) \rVert_{\dot{H}^{1-\eta}}  \overset{\eqref{Bony 4}}{\lesssim} \lVert w^{\mathcal{H}} (t) \rVert_{H^{1-\eta + \kappa}} \lVert \mathcal{L}_{\lambda_{t}} X(t) \rVert_{\mathcal{C}^{-\kappa}}  \overset{\eqref{Higher frequency estimate} \eqref{Define Lt and Nt}}{\lesssim} (N_{t}^{\kappa})^{2}. 
\end{align}
\end{subequations} 
Applying \eqref{est 34} to \eqref{Define I31 and I32} gives us   
\begin{equation}\label{est 36}
\RomanI_{31} \lesssim \lVert w^{\mathcal{L}} (t) \rVert_{\dot{H}^{\eta}} \lVert w^{\mathcal{H}} \mathcal{L}_{\lambda_{t}} X(t) \rVert_{\dot{H}^{1-\eta}} \lesssim \lVert w^{\mathcal{L}} (t) \rVert_{\dot{H}^{\eta}} \lambda_{t}^{1- \eta + 2 \kappa} (N_{t}^{\kappa})^{2}.
\end{equation} 
On the other hand, we estimate from \eqref{Define I31 and I32}, 
\begin{align}
\RomanI_{32} \lesssim& \lVert w^{\mathcal{L}}(t) \rVert_{\dot{H}^{\eta}} [ \lVert w^{\mathcal{H}} \succ \mathcal{H}_{\lambda_{t}} X(t) \rVert_{\dot{H}^{1-\eta}} + \lVert w^{\mathcal{H}} \circ \mathcal{H}_{\lambda_{t}} X(t) \rVert_{\dot{H}^{1-\eta}} ]  \nonumber \\
& \hspace{5mm} \overset{\eqref{Bony 4}\eqref{Bony 5}}{\lesssim} \lVert w^{\mathcal{L}} (t) \rVert_{\dot{H}^{\eta}}  \lVert w^{\mathcal{H}}(t) \rVert_{H^{1- \eta + \kappa}} \lVert \mathcal{H}_{\lambda_{t}} X(t) \rVert_{\mathcal{C}^{-\kappa}} \overset{\eqref{Higher frequency estimate} \eqref{Define Lt and Nt}}{\lesssim}  \lVert w^{\mathcal{L}}(t) \rVert_{\dot{H}^{\eta}} (N_{t}^{\kappa})^{2}. \label{est 35}
\end{align}
Applying \eqref{est 36} and \eqref{est 35} to \eqref{est 37} gives us \eqref{est 38}. 

To prove \eqref{est 39}, we can write using $\int_{\mathbb{T}} w^{\mathcal{L}} w^{\mathcal{L}} \partial_{x} w^{\mathcal{L}} dx = 0$, 
\begin{align}
 \lvert \langle w^{\mathcal{L}}, \partial_{x} w^{2} \rangle_{L^{2}}(t) \rvert \overset{\eqref{Define QH, wH, wL}}{=}& 2 \left\lvert  \int_{\mathbb{T}} w^{\mathcal{L}} (w^{\mathcal{L}} \partial_{x} w^{\mathcal{H}} + w^{\mathcal{H}} \partial_{x} w^{\mathcal{L}} + w^{\mathcal{H}} \partial_{x} w^{\mathcal{H}} ) (t) dx \right\rvert  \nonumber \\
 & \hspace{12mm} = \left\lvert \int_{\mathbb{T}} [  \lvert w^{\mathcal{L}} \rvert^{2} \partial_{x} w^{\mathcal{H}} + w^{\mathcal{L}}  \partial_{x}\lvert w^{\mathcal{H}}\rvert^{2} ] (t) dx \right\rvert. \label{est 40} 
\end{align} 
We estimate the first term in \eqref{est 40} by Lemma \ref{Product estimate lemma} with $d = 1$ as follows: 
\begin{equation}\label{est 41} 
\left\lvert \int_{\mathbb{T}} \lvert w^{\mathcal{L}} \rvert^{2} \partial_{x} w^{\mathcal{H}}(t) dx \right\rvert \overset{\eqref{Higher frequency estimate}}{\lesssim}\left\lVert \lvert w^{\mathcal{L}} (t) \rvert^{2} \right\rVert_{\dot{H}^{\frac{1}{2} - \kappa}} (1+ \lVert w(t) \rVert_{L^{2}})^{1- \tau (\frac{1}{2} - 3 \kappa)} N_{t}^{\kappa} \overset{\eqref{Product estimate}}{\lesssim} \lVert w^{\mathcal{L}}(t) \rVert_{\dot{H}^{\eta}}^{2} N_{t}^{\kappa}.
\end{equation} 
We estimate the second term in \eqref{est 40} also by Lemma \ref{Product estimate lemma} with $d = 1$: 
\begin{align} 
\left\lvert \int_{\mathbb{T}} w^{\mathcal{L}} \partial_{x} \lvert w^{\mathcal{H}} \rvert^{2}(t) dx \right\rvert \lesssim& \lVert w^{\mathcal{L}}(t) \rVert_{\dot{H}^{\eta}} \left\lVert \lvert w^{\mathcal{H}}(t) \rvert^{2} \right\rVert_{\dot{H}^{1-\eta}} \nonumber \\
&\overset{\eqref{Product estimate}}{\lesssim} \lVert w^{\mathcal{L}}(t) \rVert_{\dot{H}^{\eta}}\lVert w^{\mathcal{H}}(t) \rVert_{\dot{H}^{\frac{3}{4} - \frac{\eta}{2}}}^{2} \overset{\eqref{Higher frequency estimate}}{\lesssim}  \lVert w^{\mathcal{L}}(t) \rVert_{\dot{H}^{\eta}} (N_{t}^{\kappa})^{2}.\label{est 42}
\end{align} 
Applying \eqref{est 41} and \eqref{est 42} to \eqref{est 40} verifies \eqref{est 39}. 
\end{proof} 

\begin{proposition}\label{Proposition 4.7} 
Let $t \in [T_{i}, T_{i+1})$ and fix $\lambda_{t}$ from \eqref{est 29}. Let $\tau \in [\frac{25}{12}, \infty), \kappa_{0} \in (0, \frac{1}{100})$, and $\eta \in [\frac{3}{4}, 1)$. Then, for all $\kappa \in (0, \kappa_{0}]$, $\RomanI_{4}$ from \eqref{Define I4} satisfies 
\begin{align}
\RomanI_{4} + \langle w^{\mathcal{L}}, \partial_{x} w^{2} \rangle_{L^{2}} (t) \lesssim& N_{t}^{\kappa} \lVert w^{\mathcal{L}}(t) \rVert_{\dot{H}^{1- \frac{3\kappa}{2}}} (\lVert w^{\mathcal{L}} (t) \rVert_{\dot{H}^{2\kappa}}   + N_{t}^{\kappa}) \nonumber\\
& \hspace{10mm} + (N_{t}^{\kappa})^{2} \lVert w^{\mathcal{L}}(t) \rVert_{\dot{H}^{\eta}} (\lVert w^{\mathcal{L}}(t) \rVert_{\dot{H}^{\eta}} + N_{t}^{\kappa}). \label{est 59}
\end{align} 
\end{proposition}

\begin{proof}[Proof of Proposition \ref{Proposition 4.7}]
By definition from \eqref{Define I4} we have 
\begin{equation}\label{est 43} 
\RomanI_{4} + \langle w^{\mathcal{L}}, \partial_{x} w^{2} \rangle_{L^{2}}(t)= - \langle w^{\mathcal{L}}, \partial_{x} (2 wY - C^{\prec}(w, Q^{\mathcal{H}}) + Y^{2} )\rangle_{L^{2}}(t). 
\end{equation} 
Concerning the first term in \eqref{est 43}, we estimate using \eqref{Bony 3}, \eqref{Bony 4}, \eqref{Bony 5}, and \eqref{Define Lt and Nt}, 
\begin{subequations}\label{est 44}
\begin{align}
& \lVert w \prec Y(t) \rVert_{\dot{H}^{\frac{3\kappa}{2}}} + \lVert w \circ Y(t) \rVert_{\dot{H}^{\frac{3\kappa}{2}}} \lesssim \lVert w(t) \rVert_{H^{-\frac{\kappa}{2}}} \lVert Y(t) \rVert_{\mathcal{C}^{2\kappa}} \lesssim \lVert w(t) \rVert_{H^{-\frac{\kappa}{2}}} N_{t}^{\kappa}, \\
&\lVert w \succ Y(t) \rVert_{\dot{H}^{\frac{3\kappa}{2}}} \lesssim \lVert w(t) \rVert_{H^{2\kappa}} \lVert Y(t) \rVert_{\mathcal{C}^{-\frac{\kappa}{2}}} \lesssim \lVert w(t) \rVert_{H^{2\kappa}} N_{t}^{\kappa}, 
\end{align}
\end{subequations} 
so that further application of \eqref{Higher frequency estimate} gives us 
\begin{equation}\label{est 56}
2 \lvert \langle w^{\mathcal{L}}, \partial_{x} (wY) \rangle_{L^{2}}(t) \rvert \lesssim N_{t}^{\kappa} \lVert w^{\mathcal{L}}(t) \rVert_{\dot{H}^{1- \frac{3\kappa}{2}}} \lVert w(t) \rVert_{H^{2\kappa}} \lesssim N_{t}^{\kappa} \lVert w^{\mathcal{L}}(t) \rVert_{\dot{H}^{1- \frac{3\kappa}{2}}} ( \lVert w^{\mathcal{L}}(t) \rVert_{H^{2\kappa}} + N_{t}^{\kappa}).
\end{equation} 
Before the commutator term, we estimate the third term in \eqref{est 43} via \eqref{Bony 3}, \eqref{Bony 5}, and \eqref{Define Lt and Nt}, 
\begin{equation}\label{est 58}
- \langle w^{\mathcal{L}}, \partial_{x} Y^{2} \rangle_{L^{2}} (t) \lesssim \lVert w^{\mathcal{L}}(t) \rVert_{\dot{H}^{1- \frac{3\kappa}{2}}} [ \lVert Y \prec Y (t) \rVert_{\dot{H}^{\frac{3\kappa}{2}}} + \lVert Y \circ Y (t) \rVert_{\dot{H}^{\frac{3\kappa}{2}}} ] \lesssim (N_{t}^{\kappa})^{2} \lVert w^{\mathcal{L}}(t) \rVert_{\dot{H}^{1- \frac{3\kappa}{2}}}. 
\end{equation} 
Next, concerning the commutator term in \eqref{est 43}, we see from \eqref{Define commutator a} that writing out $\partial_{t} (w \prec Q^{\mathcal{H}})$ gives us a cancellation of $w \prec \partial_{t} Q^{\mathcal{H}}$ that leads us to 
\begin{equation}\label{est 45} 
C^{\prec} (w, Q^{\mathcal{H}}) = \partial_{t} w \prec Q^{\mathcal{H}}  - \nu \partial_{x}^{2}(w \prec Q^{\mathcal{H}})  + \nu w \prec \partial_{x}^{2} Q^{\mathcal{H}}.
\end{equation} 
Additionally writing out $\partial_{x}^{2}(w \prec Q^{\mathcal{H}})$ gives us another cancellation of $\nu w \prec \partial_{x}^{2} Q^{\mathcal{H}}$ that leads to, due to \eqref{Equation of w},  
\begin{subequations}
\begin{align}
C^{\prec} (w, Q^{\mathcal{H}}) =& (\partial_{t} - \nu \partial_{x}^{2}) w \prec Q^{\mathcal{H}} - 2 \nu \partial_{x} w \prec \partial_{x} Q^{\mathcal{H}} \label{est 46a}\\
=&- \frac{1}{2}  \partial_{x} (w^{2} + 2wY + 2wX+ Y^{2})  \prec Q^{\mathcal{H}} - 2 \nu \partial_{x} w \prec \partial_{x} Q^{\mathcal{H}}. \label{est 46b}
\end{align}
\end{subequations} 
Consequently, the commutator term in \eqref{est 43} that we must handle is decomposed to 
\begin{equation}\label{est 55}
- \langle w^{\mathcal{L}}, \partial_{x} C^{\prec} (w, Q^{\mathcal{H}}) \rangle_{L^{2}} (t) =  \sum_{k=1}^{5} C_{k}
\end{equation} 
where 
\begin{subequations}
\begin{align}
& C_{1} \triangleq  \frac{1}{2} \langle w^{\mathcal{L}}, \partial_{x}  [ ( \partial_{x} w^{2} ) \prec Q^{\mathcal{H}} ] \rangle_{L^{2}} (t), \hspace{6mm} C_{2} \triangleq  \langle w^{\mathcal{L}}, \partial_{x} [ [ \partial_{x} (wY) ] \prec Q^{\mathcal{H}} ] \rangle_{L^{2}} (t), \label{Define C1 and C2} \\
& C_{3} \triangleq \langle w^{\mathcal{L}}, \partial_{x} [ [ \partial_{x} (wX) ] \prec Q^{\mathcal{H}} ] \rangle_{L^{2}} (t), \hspace{5mm} C_{4} \triangleq \frac{1}{2} \langle w^{\mathcal{L}}, \partial_{x} [ ( \partial_{x} Y^{2} ) \prec Q^{\mathcal{H}} ] \rangle_{L^{2}} (t),  \label{Define C3 and C4}\\
& C_{5} \triangleq 2 \nu \langle w^{\mathcal{L}}, \partial_{x} ( \partial_{x} w \prec \partial_{x} Q^{\mathcal{H}} ) \rangle_{L^{2}} (t). \label{Define C5}
\end{align}
\end{subequations}
We estimate $C_{1}$ from \eqref{Define C1 and C2} as follows: 
\begin{align} 
&C_{1} \lesssim  \lVert w^{\mathcal{L}}(t) \rVert_{\dot{H}^{\eta}} \lVert (\partial_{x} w^{2}) \prec Q^{\mathcal{H}} (t) \rVert_{\dot{H}^{1-\eta}} \label{est 51}\\
&\overset{\eqref{Bony 3}\eqref{Estimate on Q}\eqref{Product estimate}}{\lesssim} N_{t}^{\kappa} t^{\frac{\kappa}{4}} (1+ \lVert w(t) \rVert_{L^{2}})^{\tau (2 \kappa - \eta)} \lVert w^{\mathcal{L}}(t) \rVert_{\dot{H}^{\eta}} \lVert w(t) \rVert_{\dot{H}^{\frac{1}{4} - \frac{\kappa}{4}}}^{2} \nonumber\\
&\lesssim  N_{t}^{\kappa} (1+ \lVert w(t) \rVert_{L^{2}})^{\tau (2 \kappa - \eta)} \lVert w^{\mathcal{L}}(t) \rVert_{\dot{H}^{\eta}} ( \lVert w^{\mathcal{L}}(t) \rVert_{\dot{H}^{\frac{1}{4} - \frac{\kappa}{4}}} + N_{t}^{\kappa})^{2} \lesssim N_{t}^{\kappa} \lVert w^{\mathcal{L}}(t) \rVert_{\dot{H}^{\eta}}^{1+ \frac{1-\kappa}{2\eta}} + (N_{t}^{\kappa})^{3} \lVert w^{\mathcal{L}}(t) \rVert_{\dot{H}^{\eta}} \nonumber 
\end{align}
where the last inequality used the fact that $\tau (2 \kappa - \eta) + 2 ( 1 - \frac{1-\kappa}{4\eta}) \leq 0$ due to hypothesis. Concerning $C_{2}$ from \eqref{Define C1 and C2}, we first estimate using \eqref{Bony 3} and \eqref{Estimate on Q}, 
\begin{align}
C_{2}  \lesssim \lVert w^{\mathcal{L}}(t) \rVert_{\dot{H}^{\eta}} \lVert \partial_{x} (wY) \prec Q^{\mathcal{H}} (t) &\rVert_{\dot{H}^{1-\eta}} \lesssim \lVert w^{\mathcal{L}} (t) \rVert_{\dot{H}^{\eta}} N_{t}^{\kappa} \label{est 48} \\
& \times  [ \lVert w \prec Y \rVert_{H^{-\eta+ \frac{3\kappa}{2}}} + \lVert w \succ Y \rVert_{H^{-\eta + \frac{3 \kappa}{2}}} + \lVert w \circ Y \rVert_{H^{-\eta + \frac{3\kappa}{2}}} ] (t), \nonumber 
\end{align} 
where we further bound via \eqref{Bony 3}, \eqref{Bony 4}, \eqref{Bony 5}, and \eqref{Define Lt and Nt}, 
\begin{subequations}\label{est 47} 
\begin{align}
&\lVert w \prec Y(t) \rVert_{H^{-\eta + \frac{3\kappa}{2}}} \lesssim \lVert w (t) \rVert_{H^{-\eta - \frac{\kappa}{2}}} N_{t}^{\kappa},  \\
& \lVert w \succ Y(t) \rVert_{H^{-\eta + \frac{3\kappa}{2}}} + \lVert w \circ Y(t) \rVert_{H^{-\eta + \frac{3\kappa}{2}}}  \lesssim \lVert w(t) \rVert_{H^{\kappa}} N_{t}^{\kappa}.
\end{align}
\end{subequations} 
Applying \eqref{est 47} to \eqref{est 48}, and additionally \eqref{Higher frequency estimate} leads us to 
\begin{equation}\label{est 52}
C_{2} \lesssim \lVert w^{\mathcal{L}}(t) \rVert_{\dot{H}^{\eta}} (N_{t}^{\kappa})^{2} [ \lVert w^{\mathcal{L}}(t) \rVert_{H^{\kappa}} + N_{t}^{\kappa} ]. 
\end{equation} 
For $C_{3}$ from \eqref{Define C3 and C4}, we apply \eqref{Bony 3} and \eqref{Estimate on Q} to deduce 
\begin{align}
C_{3} \lesssim& \lVert w^{\mathcal{L}}(t) \rVert_{\dot{H}^{\eta}} \lVert [ \partial_{x} (w X)] \prec Q^{\mathcal{H}}(t) \rVert_{\dot{H}^{1-\eta}}  \nonumber\\
\lesssim& N_{t}^{\kappa}  t^{\frac{\kappa}{4}} \lVert w^{\mathcal{L}}(t) \rVert_{\dot{H}^{\eta}}  [ \lVert w \prec X \rVert_{H^{-\eta + \frac{3\kappa}{2}}} + \lVert w \succ X \rVert_{H^{-\eta + \frac{3\kappa}{2}}} + \lVert w \circ X \rVert_{H^{-\eta + \frac{3\kappa}{2}}}](t), \label{est 50} 
\end{align} 
where applications of \eqref{Bony 3}, \eqref{Bony 4}, \eqref{Bony 5}, and \eqref{Define Lt and Nt} further give us 
\begin{subequations}\label{est 49} 
\begin{align}
&\lVert w \prec X(t) \rVert_{H^{-\eta + \frac{3\kappa}{2}}}  \lesssim \lVert w(t) \rVert_{H^{-\eta + \frac{5\kappa}{2}}} N_{t}^{\kappa},  \\
&\lVert w \succ X(t) \rVert_{H^{-\eta + \frac{3\kappa}{2}}} + \lVert w \circ X(t) \rVert_{H^{-\eta + \frac{3\kappa}{2}}} \lesssim  \lVert w(t) \rVert_{H^{\frac{3\kappa}{2}}} N_{t}^{\kappa}.
\end{align} 
\end{subequations} 
Applying \eqref{est 49} to \eqref{est 50}, and additionally \eqref{Higher frequency estimate} gives us 
\begin{equation}\label{est 53}
C_{3}  \lesssim  (N_{t}^{\kappa})^{2} \lVert w^{\mathcal{L}}(t) \rVert_{\dot{H}^{\eta}} [ \lVert w^{\mathcal{L}}(t) \rVert_{H^{\frac{3\kappa}{2}}} + N_{t}^{\kappa} ].
\end{equation} 
Lastly, we estimate $C_{4}$ and $C_{5}$ from \eqref{Define C3 and C4} and \eqref{Define C5} using \eqref{Bony 3}, \eqref{Define Lt and Nt}, and \eqref{Estimate on Q}, 
\begin{subequations}\label{est 54}
\begin{align}
& C_{4}  \lesssim \lVert w^{\mathcal{L}}(t) \rVert_{\dot{H}^{\eta}} \lVert ( \partial_{x} Y^{2}) \prec Q^{\mathcal{H}}(t) \rVert_{\dot{H}^{1-\eta}}   \lesssim \lVert w^{\mathcal{L}}(t) \rVert_{\dot{H}^{\eta}} (N_{t}^{\kappa})^{3}, \label{est 54a} \\
&C_{5} \lesssim  \lVert w^{\mathcal{L}}(t) \rVert_{\dot{H}^{\eta}} \lVert \partial_{x} w \prec \partial_{x} Q^{\mathcal{H}} (t) \rVert_{\dot{H}^{1-\eta}} \lesssim \lVert w^{\mathcal{L}}(t) \rVert_{\dot{H}^{\eta}} N_{t}^{\kappa} [ \lVert w^{\mathcal{L}}(t) \rVert_{H^{\eta}} + N_{t}^{\kappa}]. \label{est 54b}
\end{align}
\end{subequations} 
Applying \eqref{est 51}, \eqref{est 52}, \eqref{est 53}, and \eqref{est 54} to \eqref{est 55} gives us 
\begin{equation}\label{est 57}
- \langle w^{\mathcal{L}}, \partial_{x} C^{\prec} (w, Q^{\mathcal{H}}) \rangle_{L^{2}} (t) \lesssim  (N_{t}^{\kappa})^{2} \lVert w^{\mathcal{L}}(t) \rVert_{\dot{H}^{\eta}} [ \lVert w^{\mathcal{L}}(t) \rVert_{\dot{H}^{\eta}} + N_{t}^{\kappa}].
\end{equation} 
Applying \eqref{est 56}, \eqref{est 57}, and \eqref{est 58} to \eqref{est 43} gives us the claimed \eqref{est 59}. 
\end{proof} 
As a consequence of \eqref{est 60}, Propositions \ref{Proposition 4.5}, \ref{Proposition 4.6}, and \ref{Proposition 4.7} with a choice of $\tau = 3, \kappa_{0} \in (0, \frac{1}{150})$, and $\eta = \frac{3}{4} + 2 \kappa$, applied to \eqref{est 61}, we obtain the following estimate.

\begin{corollary}\label{Corollary 4.8} 
Fix $\lambda_{t}$ from \eqref{est 29} with $\tau = 3$ and $\kappa_{0} \in (0, \frac{1}{150})$. Then there exists a constant $C > 0$ such that for all $\kappa \in (0, \kappa_{0}]$, all $i \in \mathbb{N}_{0}$, and all $t \in [T_{i}, T_{i+1})$, 
\begin{align}
\partial_{t} \lVert w^{\mathcal{L}} (t) \rVert_{L^{2}}^{2} \leq& - \nu \lVert w^{\mathcal{L}}(t) \rVert_{\dot{H}^{1}}^{2} + \langle w^{\mathcal{L}}, \mathcal{A}_{t}^{\lambda_{t}} w^{\mathcal{L}} \rangle_{L^{2}}(t) + r_{\lambda}(t) \lVert w^{\mathcal{L}} (t) \rVert_{L^{2}}^{2} \nonumber \\
&+ C \left( \lambda_{t}^{\frac{1}{4}} \lVert w^{\mathcal{L}} \rVert_{H^{1- \frac{3\kappa}{2}}} ( N_{t}^{\kappa})^{2} + (N_{t}^{\kappa})^{3} ( \lVert w^{\mathcal{L}}  \rVert_{H^{1- \frac{3\kappa}{2}}} + \lVert w^{\mathcal{L}} \rVert_{H^{1- \frac{3\kappa}{2}}}^{2}) \right) (t). \label{est 67}
\end{align}
\end{corollary} 

\begin{proposition}\label{Proposition 4.9}
Fix $\lambda_{t}$ from \eqref{est 29} with $\tau = 3$ and $\kappa_{0} \in (0, \frac{1}{150})$. Then there exists a constant $C_{1} > 0$ and increasing continuous functions $C_{2}, C_{3}: \hspace{1mm} \mathbb{R}_{\geq 0} \mapsto \mathbb{R}_{\geq 0}$, specifically
\begin{align}\label{est 69}
C_{2} (N_{t}^{\kappa}) \approx (N_{t}^{\kappa})^{\frac{12}{2+ 3 \kappa}} + (N_{t}^{\kappa})^{\frac{2}{\kappa}} + \mathbf{m}(N_{t}^{\kappa}) \hspace{1mm} \text{ and } \hspace{1mm} C_{3} (N_{t}^{\kappa}) \approx (N_{t}^{\kappa})^{2 + \frac{12}{2+ 3 \kappa}}, 
\end{align}
where $\mathbf{m}$ is the map from Proposition \ref{Proposition 2.4}, such that for all $\kappa \in (0, \kappa_{0}],$ all $i \in \mathbb{N}_{0}$ such that $i \geq i_{0} (\theta^{\text{in}})$, $i_{0} (\theta^{\text{in}})$ from  \eqref{Define i0}, and all $t \in [T_{i}, T_{i+1})$, 
\begin{align}
\partial_{t} \lVert w^{\mathcal{L}}(t) \rVert_{L^{2}}^{2} \leq& - \frac{\nu}{2} \lVert w^{\mathcal{L}}(t) \rVert_{\dot{H}^{1}}^{2} \nonumber \\
&+ \left( C_{1} \ln( \lambda_{t}) + C_{2} (N_{t}^{\kappa}) \right) [ \lVert w^{\mathcal{L}}(t) \rVert_{L^{2}}^{2} + \lVert w^{\mathcal{L}} (T_{i}) \rVert_{L^{2}}^{2} ] + C_{3} (N_{t}^{\kappa}), \label{est 71}
\end{align}
and consequently, 
\begin{subequations}\label{est 72}
\begin{align}
&\sup_{t \in [T_{i}, T_{i+1})} \lVert w^{\mathcal{L}}(t) \rVert_{L^{2}}^{2} + \frac{\nu}{2} \int_{T_{i}}^{T_{i+1}} \lVert w^{\mathcal{L}}(t) \rVert_{\dot{H}^{1}}^{2} dt \leq e^{\mu (T_{i+1} - T_{i})} [ 2 \lVert w^{\mathcal{L}} (T_{i}) \rVert_{L^{2}}^{2} + C_{3} (N_{T_{i+1}}^{\kappa})], \label{est 72a}\\
&\text{where } \mu \triangleq C_{1} \ln(\lambda_{T_{i}}) + C_{2} (N_{T_{i+1}}^{\kappa}).    \label{est 72b} 
\end{align}
\end{subequations}
\end{proposition} 

\begin{proof}[Proof of Proposition \ref{Proposition 4.9}]
We fix an arbitrary $t \in [T_{i}, T_{i+1})$. Applying Proposition \ref{Proposition 2.4} with ``$\mathcal{U}(\Theta)$'' = $\nu \partial_{x}^{2} - \partial_{x} \mathcal{L}_{\lambda}X$, relying on \eqref{est 62} and the fact that \eqref{est 63} shows $r_{\lambda}(t) \geq 0$ leads to 
\begin{equation}\label{est 68}
\langle w^{\mathcal{L}}, \mathcal{A}_{t}^{\lambda_{t}} w^{\mathcal{L}} \rangle_{L^{2}} (t) \leq \mathbf{m}(N_{t}^{\kappa}) \lVert w^{\mathcal{L}}(t)  \rVert_{L^{2}}^{2}. 
\end{equation} 
Additionally considering $\lambda_{t}^{\frac{1}{4}} \lesssim  \lVert w^{\mathcal{L}}(T_{i}) \rVert_{L^{2}} + N_{t}^{\kappa}$ for all $t \in [T_{i}, T_{i+1})$ that can be verified from \eqref{est 29} using \eqref{Higher frequency estimate}, we deduce from \eqref{est 67}, 
\begin{align}
\partial_{t} \lVert w^{\mathcal{L}}(t) &\rVert_{L^{2}}^{2} \overset{\eqref{est 68} \eqref{logarithmic bound}}{\leq} - \nu \lVert w^{\mathcal{L}}(t) \rVert_{\dot{H}^{1}}^{2} + \mathbf{m}(N_{t}^{\kappa}) \lVert w^{\mathcal{L}} (t) \rVert_{L^{2}}^{2} + C \ln (\lambda_{t}) \lVert w^{\mathcal{L}} (t) \rVert_{L^{2}}^{2} \nonumber\\
&+ C \left( ( \lVert w^{\mathcal{L}}(T_{i}) \rVert_{L^{2}} + N_{t}^{\kappa}) \lVert w^{\mathcal{L}}(t) \rVert_{H^{1- \frac{3\kappa}{2}}} (N_{t}^{\kappa})^{2} + (N_{t}^{\kappa})^{3} (\lVert w^{\mathcal{L}} (t) \rVert_{H^{1 - \frac{3\kappa}{2}}} + \lVert w^{\mathcal{L}} (t) \rVert_{H^{1- \frac{3\kappa}{2}}}^{2} )  \right) \nonumber  \\
\leq& -\frac{\nu}{2} \lVert w^{\mathcal{L}}(t) \rVert_{\dot{H}^{1}}^{2} + \left( C_{1} \ln( \lambda_{t}) + C_{2} (N_{t}^{\kappa}) \right) [ \lVert w^{\mathcal{L}}(t) \rVert_{L^{2}}^{2} + \lVert w^{\mathcal{L}}(T_{i}) \rVert_{L^{2}}^{2} ] + C_{3} (N_{t}^{\kappa}),  \label{est 70} 
\end{align}
which verifies \eqref{est 71}. Then, using the fact that $\ln(\lambda_{t}) = \ln ( \lambda_{T_{i}})$ for all $t \in [T_{i}, T_{i+1})$ and that $\mu \geq 1$, we can deduce for all $t \in [T_{i}, T_{i+1})$, 
\begin{equation}\label{est 75}
 \lVert w^{\mathcal{L}}(t) \rVert_{L^{2}}^{2}  + \frac{\nu}{2} \int_{T_{i}}^{t}   \lVert w^{\mathcal{L}}(s) \rVert_{\dot{H}^{1}}^{2} ds \leq e^{\mu (T_{i+1} - T_{i})} [ 2 \lVert w^{\mathcal{L}} (T_{i}) \rVert_{L^{2}}^{2} + C_{3} (N_{T_{i+1}}^{\kappa}) ];
\end{equation} 
taking supremum over all $t \in [T_{i}, T_{i+1})$ on the left hand side gives \eqref{est 72}. 
\end{proof}

\begin{proposition}\label{Proposition 4.10}
Fix $\lambda_{t}$ from \eqref{est 29} with $\tau = 3$ and $\kappa_{0} \in (0, \frac{1}{150})$. Consider $i \in \mathbb{N}$ such that $i \geq i_{0} (\theta^{\text{in}})$, $i_{0}(\theta^{\text{in}})$ from \eqref{Define i0}, and $t > 0$. If $T_{i+1} < T^{\max} \wedge t$, then for all $\kappa \in (0, \kappa_{0}]$, there exist
\begin{subequations}\label{est 79} 
\begin{align}
&C(N_{t}^{\kappa}) \triangleq 2C N_{t}^{\kappa} \\
\text{and } &\tilde{C} (N_{t}^{\kappa}) \triangleq \max\{ 3C_{1},   C_{2}(N_{t}^{\kappa}),  4 C N_{t}^{\kappa} + 2C^{2} (N_{t}^{\kappa})^{2} + C_{3} (N_{t}^{\kappa})\},
\end{align} 
\end{subequations} 
where $C$ is the universal constant from \eqref{Higher frequency estimate} and $C_{1}, C_{2}, C_{3}$ were given in the statement of Proposition \ref{Proposition 4.9}, such that 
\begin{equation}\label{est 77}
T_{i+1} - T_{i} \geq \frac{1}{\tilde{C}(N_{t}^{\kappa})(1+ \ln (3+i))} \ln \left( \frac{i^{2} + 2i - C(N_{t}^{\kappa})}{2i^{2} + \tilde{C}(N_{t}^{\kappa})} \right).
\end{equation} 
\end{proposition}

\begin{proof}[Proof of Proposition \ref{Proposition 4.10}]
First, from \eqref{est 75}, using the definition of $\mu$ from \eqref{est 72b} and that $N_{t}^{\kappa}$ is non-decreasing in $t$, we have 
\begin{equation}\label{est 76}
\frac{1}{C_{1} \ln(\lambda_{T_{i}}) + C_{2} (N_{t}^{\kappa})} \ln\left( \frac{ \lVert w^{\mathcal{L}} (T_{i+1} -) \rVert_{L^{2}}^{2}}{2 \lVert w^{\mathcal{L}}(T_{i}) \rVert_{L^{2}}^{2} + C_{3} (N_{t}^{\kappa})} \right) \leq T_{i+1} - T_{i}. 
\end{equation}
Moreover, applications of \eqref{Higher frequency estimate} give us 
\begin{equation}\label{est 78}
 \lVert w^{\mathcal{L}} (T_{i+1} - ) \rVert_{L^{2}}  \geq i + 1 - C(i+1)^{-1} N_{t}^{\kappa} \hspace{1mm} \text{ and } \hspace{1mm}  \lVert w^{\mathcal{L}}(T_{i}) \rVert_{L^{2}}   \leq i + \frac{C N_{t}^{\kappa}}{i}.
\end{equation} 
Applying \eqref{est 78} to \eqref{est 76} gives us \eqref{est 77} as desired. 
\end{proof}

\begin{proposition}\label{Proposition 4.11}
Fix $\lambda_{t}$ from \eqref{est 29} with $\tau = 3$ and $\kappa_{0} \in (0, \frac{1}{150})$. Then the following holds for any $\kappa \in (0, \kappa_{0})$ and $\epsilon \in (0, \kappa)$. Suppose that there exist $M > 1$ and $T > 0$ such that 
\begin{equation}\label{Define M}
\lVert w^{\mathcal{L}}(0) \rVert_{\dot{H}^{\epsilon}}^{2} + \sup_{t \in [0, T \wedge T^{\max} ]} \lVert w^{\mathcal{L}}(t) \rVert_{L^{2}}^{2} + \frac{\nu}{2} \int_{0}^{T \wedge T^{\max}} \lVert w^{\mathcal{L}}(s) \rVert_{\dot{H}^{1}}^{2} ds \leq M. 
\end{equation} 
Then there exists $C(T, M, N_{T}^{\kappa}) \in (0, \infty)$ such that 
\begin{equation}
\sup_{t \in [0, T \wedge T^{\max} ]} \lVert w^{\mathcal{L}}(t) \rVert_{\dot{H}^{\epsilon}}^{2} \leq C(T, M, N_{T}^{\kappa}). 
\end{equation} 
\end{proposition} 

Due to similarities to the estimates within the proof of Propositions \ref{Proposition 4.5}-\ref{Proposition 4.7}, we leave the proof of Proposition \ref{Proposition 4.11} in Section \ref{Section B.2} for completeness.

\begin{proposition}\label{Proposition 4.12} 
Suppose that $\theta^{\text{in}} \in L^{2} (\mathbb{T})$. If $T^{\max} < \infty$, then 
\begin{equation}
\limsup_{t \nearrow T^{\max}} \lVert w(t) \rVert_{L^{2}} = + \infty. 
\end{equation} 
\end{proposition} 
Due to similarities to the previous works \cite[Corollary 5.4]{HR24} and \cite[Proposition 4.11]{Y23b}, we lave this proof in Section \ref{Section B.3} for completeness.  

Next, we return to $\mathcal{A}_{t}$, $\mathcal{A}_{t}^{\lambda}$, and the enhanced noise in \eqref{Define At}, \eqref{Define At lambda}, and \eqref{est 64}. We define 
\begin{equation}\label{Define em and zeta}
e_{m}(x) \triangleq e^{i2\pi m x} \hspace{1mm} \text{ and } \zeta(t,m) \triangleq \lvert m \rvert^{\frac{1}{2}} \beta(t,m) 
\end{equation} 
where $\{\beta(m)\}_{m\in\mathbb{Z}}$ is a family of $\mathbb{C}$-valued two-sided Brownian motions such that 
\begin{equation}\label{correlation}
\mathbb{E} [\partial_{t} \beta(t,m) \partial_{t} \beta(s,m') ] = \delta(t-s) 1_{\{ m = -m'\}}, 
\end{equation} 
so that we may write 
\begin{equation}\label{est 148}
\Lambda^{\frac{1}{2}} \xi(t,x)  =  \sum_{m\in\mathbb{Z} \setminus \{0\}} e_{m}(x) \partial_{t} \zeta(t,m). 
\end{equation} 
Additionally, we define, with $\mathfrak{l}$ from Definition \ref{Definition 3.1}, 
\begin{equation}\label{Define F and F lambda}
F(t,m) \triangleq \int_{0}^{t} e^{- \nu \lvert m \rvert^{2} (t-s)} d \zeta(s,m) \hspace{1mm} \text{ and } \hspace{1mm} F^{\lambda}(t,m) \triangleq \int_{0}^{t} e^{-\nu \lvert m \rvert^{2} (t-s)} \mathfrak{l} \left( \frac{m}{\lvert \lambda \rvert} \right) d \zeta(s,m). 
\end{equation} 
Consequently, we can write the solution $X$ of \eqref{Equation of X} as 
\begin{equation}\label{Rewrite X}
X(t,x) = \sum_{m\in\mathbb{Z} \setminus \{0\}} e_{m}(x) F(t,m)
\end{equation} 
and hence 
\begin{subequations}\label{est 149} 
\begin{align}
&\mathcal{L}_{\lambda} X(t,x) \triangleq \sum_{m \in \mathbb{Z} \setminus \{0\}} e_{m}(x) F^{\lambda} (t,m), \label{est 149a}\\
\text{and }& \left( 1- \nu \partial_{x}^{2} \right)^{-1} \mathcal{L}_{\lambda} X(t,x) =\sum_{m\in\mathbb{Z} \setminus \{0\}} e_{m}(x) F^{\lambda}(t,m) \left( 1+ \nu \lvert m \rvert^{2} \right)^{-1}.   \label{est 149b}
\end{align}
\end{subequations} 

\begin{proposition}\label{Proposition 4.13} 
For any $\kappa > 0$, define $\mathcal{K}^{-1-\kappa}$ by \eqref{est 21}, $P^{\lambda}$, and $r_{\lambda}$ by \eqref{est 63}. Then, for any $t \geq 0$, there exists a distribution $\partial_{x} X \diamondsuit P_{t} \in \mathcal{C}^{-2\kappa}(\mathbb{T})$ such that 
\begin{equation}\label{est 159}
\left( \partial_{x} \mathcal{L}_{\lambda^{n}} X, \partial_{x} \mathcal{L}_{\lambda^{n}} X \circ P^{\lambda^{n}} - r_{\lambda^{n}} \right) \to \left( \partial_{x} X, \partial_{x} X \diamondsuit P \right)
\end{equation} 
as $n\nearrow + \infty$ both in $L^{p} \left(\Omega; C_{\text{loc}} (\mathbb{R}_{+}; \mathcal{K}^{-1-\kappa} ) \right)$ for any $p\in [1,\infty)$ and $\mathbb{P}$-a.s. Finally, there exists a constant $c > 0$ such that for all $\lambda \geq 1$, 
\begin{equation}\label{logarithmic bound}
r_{\lambda}(t) \leq c \ln (\lambda) \hspace{5mm} \text{ uniformly over all } t \geq 0. 
\end{equation} 
\end{proposition} 

\begin{proof}[Proof of Proposition \ref{Proposition 4.13}]
We focus on the task of proving the convergence of $( \partial_{x} \mathcal{L}_{\lambda^{n}} X) \circ P^{\lambda^{n}} - r_{\lambda^{n}} \to \partial_{x} X \diamondsuit P$ as $ n\nearrow + \infty$. We denote for brevity $X_{\lambda} \triangleq  \mathcal{L}_{\lambda}X$ and look at 
\begin{align}
& \partial_{x} X_{\lambda} \circ P^{\lambda}(t,x)  = - \sum_{k, k' \in \mathbb{Z}: k' \neq 0, k \neq k', \lvert c-d \rvert \leq 1} e^{i 2 \pi kx} \rho_{c} (k-k') \rho_{d} (k') \mathfrak{l} \left( \frac{ \lvert k-k'\rvert}{\lambda} \right) \mathfrak{l} \left( \frac{ \lvert k' \rvert}{\lambda} \right)   \nonumber \\
& \hspace{48mm} \times F(t, k-k') F(t,k') \left( 1+  \nu \lvert k'\rvert^{2} \right)^{-1} (k-k') k', \label{est 150} 
\end{align}
which can be verified using \eqref{est 63} and \eqref{est 149a}. We compute using \eqref{Define F and F lambda}, \eqref{Define em and zeta}, and \eqref{correlation}, 
\begin{equation}\label{est 151}
\mathbb{E} [ F(t, k-k') F(t, k')]  = \left( \frac{1- e^{-2 \nu \lvert k'\rvert^{2} t}}{2 \nu \lvert k' \rvert} \right) 1_{ \{ k-k' = - k' \}}.
\end{equation} 
Taking mathematical expectation w.r.t. $\mathbb{P}$ on \eqref{est 150} and applying \eqref{est 151} gives us 
\begin{equation}\label{est 152}
\mathbb{E} [ \partial_{x} X_{\lambda} \circ P^{\lambda} ](t,x) =  \sum_{k \in \mathbb{Z} \setminus \{0\}} \mathfrak{l} \left( \frac{ \lvert k \rvert}{\lambda} \right)^{2} \left( \frac{1- e^{-2 \nu \lvert k \rvert^{2} t}}{2 \nu} \right) \left(1+ \nu \lvert k \rvert^{2} \right)^{-1} \lvert k \rvert \overset{\eqref{est 63}}{=} r_{\lambda}(t). 
\end{equation} 
For brevity we define 
\begin{equation}\label{Define psi 0}
\psi_{0} (k, k') \triangleq \sum_{\lvert c-d \rvert \leq 1} \rho_{c} (k) \rho_{d} (k')
\end{equation} 
and compute using \eqref{Define F and F lambda}, \eqref{est 150}, \eqref{est 152}, \eqref{Define psi 0}, and Wick products (e.g. \cite[Theorem 3.12]{J97}), 
\begin{align}
&  \mathbb{E} \left[ \left \lvert \Delta_{m} \left(\partial_{x} \mathcal{L}_{\lambda} X \circ P^{\lambda} - r_{\lambda} \right)(t) \right\rvert^{2} \right] \nonumber\\
=&  \sum_{k, k' \in \mathbb{Z} \setminus \{0\}} \Bigg(e^{i 4 \pi (k+k') x} \rho_{m}^{2} (k+k') \lvert \psi_{0}(k, k') \rvert^{2} \mathfrak{l} \left( \frac{ \lvert k \rvert}{\lambda} \right)^{2} \mathfrak{l} \left( \frac{\lvert k'\rvert}{\lambda} \right)^{2}   (1+  \nu \lvert k' \rvert^{2})^{-1}  \nonumber\\
& \hspace{5mm} \times \int_{0}^{t} e^{-2 \nu \lvert k \rvert^{2} (t-s)} ds \int_{0}^{t} e^{-2\nu \lvert k'\rvert^{2} (t-s')} ds' \lvert k \rvert^{3} \lvert k'\rvert^{3} [ (1+ \nu \lvert k \rvert^{2})^{-1} + (1+ \nu \lvert k'\rvert^{2})^{-1} ] \Bigg). \label{est 157}
\end{align}
We deduce $m \lesssim c$ from $\rho_{m}(k), \rho_{c}(k-k'), \rho_{d}(k'),$ and $\lvert c-d \rvert \leq 1$ to estimate 
\begin{align} 
&  \mathbb{E} \left[ \left\lvert \Delta_{m} \left(\partial_{x} \mathcal{L}_{\lambda} X \circ P^{\lambda} - r_{\lambda} \right)(t) \right\rvert^{2} \right] \nonumber\\
\overset{\eqref{Define psi 0}}{\lesssim}& \sum_{k, k' \in \mathbb{Z} \setminus \{0\}} \rho_{m}^{2} (k) \left( \sum_{\lvert c-d \rvert \leq 1} \rho_{c}(k-k') \rho_{d} (k') \right)^{2} \mathfrak{l} \left( \frac{\lvert k-k' \rvert}{\lambda} \right)^{2} \mathfrak{l} \left( \frac{ \lvert k \rvert}{\lambda} \right)^{2} (1+  \nu \lvert k' \rvert^{2})^{-2} \lvert k-k' \rvert \lvert k' \rvert   \nonumber \\
\lesssim&  \sum_{k, k' \in \mathbb{Z} \setminus \{0\}: \lvert k \rvert \approx 2^{m}, \lvert k' \rvert \gtrsim 2^{m}} \lvert k' \rvert^{\frac{5}{2}} \left( \sum_{c: m \lesssim c} 2^{-\frac{c}{4}} \right)^{2} (1+ \lvert k' \rvert^{2})^{-2}  \lesssim 1.  \label{est 158}
\end{align}
Thanks to the Gaussian hypercontractivity theorem (e.g. \cite[Theorem 3.50]{J97}), we now conclude for any $p \in [2, \infty)$, 
\begin{align}
&\sup_{\lambda \geq 1} \mathbb{E} [ \lVert ( \partial_{x} \mathcal{L}_{\lambda} X \circ P^{\lambda}) (t) - r_{\lambda}(t) \rVert_{B_{p,p}^{-\kappa}}^{p} ] \nonumber \\
\lesssim& \sup_{ \lambda \geq 1} \sum_{m\geq -1} 2^{- \kappa m p} \int_{\mathbb{T}} \left\lVert \Delta_{m} \left( \partial_{x} \mathcal{L}_{\lambda} X \circ P^{\lambda}  - r_{\lambda} \right)(t) \right\rVert_{L_{\omega}^{2}}^{p} dx \overset{\eqref{est 158}}{\lesssim} 1.
\end{align}
This leads to the convergence claimed in \eqref{est 159} for all $p \in [1,\infty)$. The convergence $\mathbb{P}$-a.s. can be verified similarly; thus, we leave its proof in Section \ref{Section B.4} and conclude this proof of Proposition \ref{Proposition 4.13}. 
\end{proof}

\begin{proof}[Proof of Theorem \ref{Theorem 2.2}]
The proof of Theorem \ref{Theorem 2.2} now follows identically to the reasonings in \cite{HR24, Y23b}. If $T^{\max} < \infty$, then Proposition \ref{Proposition 4.12} gives us $\limsup_{t\nearrow T^{\max}} \lVert w(t) \rVert_{L^{2}} = + \infty$ which implies $T_{i} < T^{\max}$ for all $i \in \mathbb{N}$ by \eqref{Define Ti}. This allows us to sum \eqref{est 77} over all $i \in \mathbb{N}$ and reach a contradiction. We refer to \cite[p. 27]{HR24} and \cite[Section 6.2]{Y23b} for more details. 
\end{proof}

\section{Proof of Theorem \ref{Theorem 2.3}}\label{Section 5}
We finally define the HL-weak solution from Theorem \ref{Theorem 2.3}. 
\begin{define}\label{Definition 5.1}
Given any $\theta^{\text{in}} \in L^{2}(\mathbb{T})$ that is mean-zero and any $\kappa \in (0, \frac{1}{2})$, $v \in C([0, \infty); \mathcal{S}'(\mathbb{T}))$ is called a global high-low (HL) weak solution to \eqref{Equation of v} if $w = v - Y$ from \eqref{Equation of w}, with $Y$ that solves \eqref{Equation of Y}, satisfies the following statements. 
\begin{enumerate}
\item For any $T > 0$, there exists a $\lambda_{T} > 0$ such that for any $\lambda \geq \lambda_{T}$, there exists 
\begin{subequations}\label{est 127} 
\begin{align}
&w^{\mathcal{L},\lambda} \in L^{\infty} (0, T; L^{2}) \cap L^{2}(0, T; H^{1}), \label{est 127a}\\
& w^{\mathcal{H}, \lambda} \in L^{\infty} (0, T; L^{2}) \cap L^{2}(0, T; B_{\infty,2}^{1-2\kappa}), \label{est 127b}
\end{align}
\end{subequations} 
that satisfies 
\begin{equation}\label{est 129} 
w^{\mathcal{H}, \lambda}(t) = - \frac{1}{2} \partial_{x} (w \prec \mathcal{H}_{\lambda} Q) (t), \hspace{3mm} w^{\mathcal{L}, \lambda}(t) = w(t) - w^{\mathcal{H}, \lambda} (t) 
\end{equation} 
(cf. \eqref{Define QH, wH, wL}) for all $t \in [0,T]$ and for $Q$ defined in \eqref{Define Q}. 
\item $w$ solves \eqref{Equation of w} distributionally; i.e., for any $T> 0$ and $\phi \in C^{\infty} ([0,T] \times \mathbb{T}; \mathbb{R})$, 
\begin{align}
\langle w(T), \phi(T) \rangle_{L^{2}} - \langle w(0), \phi(0) \rangle_{L^{2}} =& \int_{0}^{T} \langle w, \partial_{t} \phi + \nu \partial_{x}^{2} \phi \rangle_{L^{2}}  + \frac{1}{2} \langle w, w \partial_{x} \phi \rangle_{L^{2}}  \nonumber \\
&+ \langle w, Y \partial_{x} \phi \rangle_{L^{2}} + \langle w, X \partial_{x} \phi \rangle_{L^{2}} + \frac{1}{2} \langle Y, Y \partial_{x} \phi \rangle_{L^{2}} dt. 
\end{align}
\end{enumerate} 
\end{define} 
\begin{remark}\label{Remark 5.1}
The analogous regularity of $w^{\mathcal{H}, \lambda}$ in \cite[Definition 7.1]{HR24} was  ``$L^{2}([0,T]; B_{4,\infty}^{1-\delta})$'' for all $\delta \in (0,1)$ while that of \cite[Definition 5.1]{Y23b} was $L^{2} ([0,T]; B_{p,2}^{1-2\kappa})$ for all $p \in [1, \frac{2}{\kappa}]$. In contrast, $w^{\mathcal{H}, \lambda} \in L^{2}(0, T; B_{\infty,2}^{1-2\kappa})$ in \eqref{est 127b} is better. 
\end{remark}

\begin{proposition}\label{Proposition 5.1}
Let $\mathcal{N} '' \subset \Omega$ be the null set from Proposition \ref{Proposition 4.1}. Then, for any $\omega \in \Omega \setminus \mathcal{N}''$ and $\theta^{\text{in}} \in L^{2}(\mathbb{T})$ that is mean-zero, there exists a global HL weak solution to \eqref{Equation of v}. 
\end{proposition}

\begin{proof}[Proof of Proposition \ref{Proposition 5.1}]
We define $X^{n} \triangleq \mathcal{L}_{n} X$ for $n \in \mathbb{N}_{0}$ and $X$ that solves \eqref{Equation of X}. We define $Y^{n}$ to be the corresponding solution to \eqref{Equation of Y} with $X$ therein replaced by $X^{n}$ and then define $w^{n}$ to be the solution to
\begin{equation}\label{Equation of wn} 
\partial_{t} w^{n} + \frac{1}{2} \partial_{x} ((w^{n})^{2} + 2w^{n}Y^{n} + 2w^{n}X^{n} + (Y^{n})^{2}) = \nu \partial_{x}^{2} w^{n} \hspace{1mm} \text{ for } t > 0, w^{n}(0,x) = \mathcal{L}_{n}\theta^{\text{in}}(x). 
\end{equation} 
Furthermore, similarly to $P^{\lambda}$ and $r_{\lambda}$ in \eqref{est 63} and $L_{t}^{\kappa}$ and $N_{t}^{\kappa}$ in \eqref{Define Lt and Nt}, we define 

\begin{subequations}
\begin{align}
&P^{\lambda,n}(t,x) \triangleq \left(1- \nu \partial_{x}^{2} \right)^{-1} \partial_{x} \mathcal{L}_{\lambda} X^{n}(t,x), \\
& r_{\lambda}(t) \triangleq \sum_{k\in \mathbb{Z} \setminus \{0\}} \mathfrak{l} \left( \frac{ \lvert k \rvert}{\lambda } \right) \mathfrak{l} \left( \frac{ \lvert k \rvert}{n} \right) \left( \frac{1- e^{-2 \nu \lvert k \rvert^{2} t}}{2\nu} \right) \left(1+ \nu \lvert k \rvert^{2} \right)^{-1} \lvert k \rvert, \hspace{3mm} r_{\lambda}^{n}(t) \leq c \ln (\lambda \wedge n),  
\end{align}
\end{subequations} 
where the last inequality is due to \eqref{logarithmic bound}, and 
\begin{subequations}\label{est 115}
\begin{align}
&L_{t}^{n,\kappa} \triangleq 1 + \lVert X^{n} \rVert_{C_{t} \mathcal{C}_{x}^{-\kappa}} + \lVert Y^{n} \rVert_{C_{t} \mathcal{C}_{x}^{2\kappa}}, \\
&N_{t}^{n,\kappa} \triangleq L_{t}^{n,\kappa} + \sup_{i \in \mathbb{N}} \lVert ( \partial_{x} \mathcal{L}_{\lambda^{i}} X^{n}) \circ P^{\lambda^{i},n} - r_{\lambda^{i}}^{n} \rVert_{C_{t}\mathcal{C}_{x}^{-2\kappa}}, \hspace{1mm} \text{ and } \hspace{1mm} \bar{N}_{t}^{\kappa} (\omega) \triangleq \sup_{n\in\mathbb{N}} N_{t}^{n, \kappa} (\omega), 
\end{align} 
\end{subequations} 
where $\{\lambda^{i}\}_{i\in\mathbb{N}}$ is from Definition \ref{Definition 4.2}. Consequently, $\lim_{n\nearrow \infty} N_{t}^{n,\kappa} (\omega) = N_{t}^{\kappa}(\omega)$ where $N_{t}^{\kappa}$ is from \eqref{Define Lt and Nt}, and $\bar{N}_{t}^{\kappa} (\omega) <\infty$ for all $\omega \in \Omega \setminus \mathcal{N}''$ where $\mathcal{N}''$ is the null set from Proposition \ref{Proposition 4.1}. Then, similarly to Definition \ref{Definition 4.2}, we define for all $i \in \mathbb{N}_{0}$, 
\begin{equation}\label{est 110} 
T_{0}^{n} \triangleq 0, \hspace{3mm} T_{i+1}^{n} (\omega, \theta^{\text{in}}) \triangleq \inf\{t \geq T_{i}^{n}: \hspace{1mm} \lVert w^{n} (t) \rVert_{L^{2}} \geq i + 1 \} 
\end{equation} 
and 
\begin{equation}\label{est 111}
\lambda_{0}^{n} \triangleq \lambda_{0}, \hspace{3mm} \lambda_{t}^{n} \triangleq (1+ \lVert w^{n} (T_{i}^{n}) \rVert_{L^{2}})^{3} \hspace{3mm} \text{ for } t > 0 \text{ such that } t \in [T_{i}^{n}, T_{i+1}^{n}). 
\end{equation} 
Similarly to \eqref{Define Q} we consider 
\begin{equation}\label{Equation of Qn}
( \partial_{t} - \nu \partial_{x}^{2}) Q^{n} = 2 X^{n}, \hspace{3mm} Q^{n} (0) = 0 
\end{equation} 
and define similarly to \eqref{Define QH, wH, wL}
\begin{equation}\label{est 112}
Q^{n, \mathcal{H}}(t) \triangleq \mathcal{H}_{\lambda_{t}^{n}} Q^{n}(t), \hspace{3mm} w^{n, \mathcal{H}}(t) \triangleq - \frac{1}{2} \partial_{x} (w^{n} \prec Q^{n,\mathcal{H}})(t), \hspace{3mm} w^{n,\mathcal{L}} (t) \triangleq w^{n}(t) - w^{n,\mathcal{H}}(t). 
\end{equation} 
Repeating the computations identically up to Proposition \ref{Proposition 4.9}, we can obtain $\kappa_{0} > 0$ sufficiently small so that there exist a constant $C_{1} > 0$ and increasing continuous functions $C_{2}, C_{3}: \hspace{1mm} \mathbb{R}_{\geq 0} \mapsto \mathbb{R}_{\geq 0}$ such that 
\begin{align*}
& \sup_{t \in [T_{i}^{n}, T_{i+1}^{n})} \lVert w^{n, \mathcal{L}}(t) \rVert_{L^{2}}^{2} + \frac{\nu}{2} \lVert w^{n, \mathcal{L}}  \rVert_{L^{2} (T_{i}^{n}, T_{i+1}^{n}; \dot{H}^{1} (\mathbb{T}))}^{2}   \\
\leq& e^{[C_{1} \ln( \lambda_{T_{i}^{n}}) + C_{2} (N_{T_{i+1}^{n}}^{n, \kappa})] (T_{i+1}^{n} -  T_{i}^{n})} [ 2 \lVert w^{n,\mathcal{L}} (T_{i}^{n}) \rVert_{L^{2}}^{2} + C_{3} (N_{T_{i+1}^{n}}^{n,\kappa}) ]  
\end{align*}
for all $\kappa \in (0, \kappa_{0}]$ and $i \geq i_{0} (\theta^{\text{in}})$. Similarly to Proposition \ref{Proposition 4.10} and the proof of Theorem \ref{Theorem 2.2}, we can then show uniformly over all $n \in \mathbb{N}$ and $i \geq i_{0} (\theta^{\text{in}})$ that 
\begin{equation}
T_{i+1}^{n} - T_{i}^{n} \geq \frac{1}{\tilde{C} ( \bar{N}_{T_{i+1}^{n}}^{\kappa}) (1+ \ln(3+ i))} \ln \left( \frac{ i^{2} + 2i - C( \bar{N}_{T_{i+1}^{n}}^{\kappa})}{2i^{2} + \tilde{C} (\bar{N}_{T_{i+1}^{n}}^{\kappa})} \right)
\end{equation} 
for constants $C( \bar{N}_{T_{i+1}^{n}}^{\kappa})$ and $\tilde{C} (\bar{N}_{T_{i+1}^{n}}^{\kappa})$ defined specifically in \eqref{est 79} and thus for every $T > 0, i \in \mathbb{N}$, and $i > i_{0} (\theta^{\text{in}})$, there exists $\mathfrak{t} (i, \bar{N}_{T}^{\kappa}) \in (0, T]$ such that 
\begin{equation}
\inf_{n\in\mathbb{N}_{0}} T_{i}^{n} \geq \mathfrak{t} (i, \bar{N}_{T}^{\kappa}), \hspace{3mm} \mathfrak{t} (i, \bar{N}_{T}^{\kappa}) = T \hspace{3mm} \forall \hspace{1mm} i \text{ sufficiently large}. 
\end{equation} 
Therefore, for all $T>  0$ and $\kappa > 0$ sufficiently small, there exists $C(T, \bar{N}_{T}^{\kappa}) > 0$ such that 
\begin{align}\label{est 114}
\sup_{n\in\mathbb{N}} \left( \lVert w^{n, \mathcal{L}} \rVert_{C_{T}L^{2}}^{2} + \frac{\nu}{2}  \lVert w^{n,\mathcal{L}} \rVert_{L^{2} (0,T; \dot{H}^{1}(\mathbb{T}))}^{2}  \right)  \leq C(T, \bar{N}_{T}^{\kappa}).
\end{align}
Moreover, we can find $\bar{\lambda}_{T} > 0$, in accordance to Definition \ref{Definition 5.1} (1), such that 
\begin{equation}\label{est 113}
\lambda_{t}^{n} \leq \bar{\lambda}_{T} \hspace{3mm} \forall \hspace{1mm} t \in [0, T], \hspace{1mm} n \in \mathbb{N}.
\end{equation} 
Thus, extending $w^{n,\mathcal{H}}$ and $w^{n,\mathcal{L}}$ in \eqref{est 112} to 
\begin{equation}\label{est 124}
w^{n,\mathcal{H}, \lambda} \triangleq - \frac{1}{2} \partial_{x} (w^{n} \prec \mathcal{H}_{\lambda} Q^{n}), \hspace{3mm} w^{n,\mathcal{L},\lambda} \triangleq w^{n} - w^{n,\mathcal{H},\lambda} \hspace{3mm} \forall \hspace{1mm} \lambda \geq \bar{\lambda}_{T}, 
\end{equation}  
we see that for all $\lambda \geq \bar{\lambda}_{T}$ and hence $\lambda \geq \lambda_{t}^{n}$ for all $t \in [0,T]$ and $n \in \mathbb{N}$ due to \eqref{est 113}, following the previous computations leads us now to 
\begin{equation}\label{est 128}
\sup_{n\in\mathbb{N}} \left( \lVert w^{n,\mathcal{L},\lambda} \rVert_{C_{T}L^{2}}^{2} + \frac{\nu}{2} \lVert w^{n,\mathcal{L},\lambda} \rVert_{L^{2} (0,T; \dot{H}^{1}(\mathbb{T}))}^{2} \right) \leq C(\lambda, T, \bar{N}_{T}^{\kappa}). 
\end{equation} 
Next, for all $\alpha < 1 - 2 \kappa$ and $\kappa \in (0, \frac{1}{3}]$, we obtain due to \eqref{est 112}, \eqref{Higher frequency estimate}, and \eqref{est 114}, 
\begin{equation}\label{est 117}
\sup_{n\in\mathbb{N}} \left( \lVert w^{n} \rVert_{C_{T}L^{2}}^{2} + \frac{\nu}{2} \lVert w^{n}  \rVert_{L^{2} (0,T; \dot{H}^{\alpha} (\mathbb{T}))}^{2} \right) \leq C( T,\bar{N}_{T}).
\end{equation} 
Furthermore, applying \eqref{est 112}, \eqref{Bony 3}, and \eqref{est 115} gives us 
\begin{equation}\label{est 116}
\lVert w^{n}(t) \rVert_{H^{1- \frac{3\kappa}{2}}} \lesssim \lVert w^{n}(t) \rVert_{L^{2}} \bar{N}_{t}^{\kappa} + \lVert w^{n,\mathcal{L}}(t) \rVert_{H^{1- \frac{3\kappa}{2}}}. 
\end{equation} 
As a consequence of \eqref{est 116}, \eqref{est 114}, and \eqref{est 117}, for all $\kappa \in (0, \frac{1}{3}]$ sufficiently small
\begin{equation}\label{est 118}
\sup_{n\in\mathbb{N}} \lVert w^{n} \rVert_{L^{2} (0, T; H^{1- \frac{3\kappa}{2}})}^{2} \leq C(T, \bar{N}_{T}^{\kappa}). 
\end{equation} 
Consequently, for some $N_{1} \in \mathbb{N}$ from \eqref{est 0}, all $\kappa \in (0,\frac{1}{2}]$, due to Bernstein's inequality
\begin{align}
 \lVert w^{n,\mathcal{H},\lambda} \rVert_{L^{2}(0,T; B_{\infty,2}^{1-2\kappa})}^{2} \overset{\eqref{est 0}}{\lesssim}&  \int_{0}^{T} \sum_{m\geq -1} \left\lvert 2^{m(2-2\kappa)} \sum_{l: l \leq m + N_{1} -2} \lVert \Delta_{l} w^{n} (t) \rVert_{L^{\infty}} \lVert \Delta_{m} \mathcal{H}_{\lambda} Q^{n} (t) \rVert_{L^{\infty}} \right\rvert^{2} dt  \nonumber \\
\lesssim& \int_{0}^{T} \lVert Q^{n} (t) \rVert_{\mathcal{C}^{2- \frac{3\kappa}{2}}}^{2} \lVert 2^{-m (\frac{\kappa}{2})} \rVert_{l^{1}}^{2} \lVert w^{n} (t) \rVert_{H^{\frac{1}{2} - \frac{\kappa}{2}}}^{2} dt \overset{\eqref{est 118}}{\leq} C(T, \bar{N}_{T}^{\kappa}).  \label{est 126}
\end{align}
Next, we estimate for any $\kappa \in (0, \frac{1}{6})$, via Sobolev embeddings and H$\ddot{\mathrm{o}}$lder's inequality, 
\begin{subequations}\label{est 119} 
\begin{align}
& \lVert (w^{n})^{2} (t) \rVert_{H^{-3\kappa}} \lesssim \lVert w^{n} (t) \rVert_{L^{\frac{4}{1+ 6 \kappa}}}^{2} \lesssim  \lVert w^{n} (t) \rVert_{L^{2}}^{\frac{3-6\kappa}{2(1-3\kappa)}} \lVert w^{n} (t) \rVert_{\dot{H}^{1- 3\kappa}}^{\frac{1-6\kappa}{2(1-3\kappa)}}, \\
& \lVert w^{n} Y^{n}(t) \rVert_{H^{-3\kappa}} \lesssim \lVert w^{n}(t) \rVert_{L^{2}} \lVert Y^{n} (t) \rVert_{L^{\frac{1}{3\kappa}}}  \overset{\eqref{est 115}}{\lesssim} \lVert w^{n} (t) \rVert_{L^{2}} \bar{N}_{T}^{\kappa},  \\
& \lVert w^{n} X^{n} (t) \rVert_{H^{-3\kappa}} \overset{\eqref{Bony 3} \eqref{Bony 4} \eqref{Bony 5}  \eqref{est 115}}{\lesssim} \bar{N}_{T}^{\kappa} \lVert w^{n} (t) \rVert_{H^{2\kappa}}, \\
&\lVert (Y^{n})^{2} (t) \rVert_{H^{-3\kappa}}  \overset{\eqref{Bony 3} \eqref{Bony 5}}{\lesssim} \lVert Y^{n}(t) \rVert_{H^{-2\kappa}} \lVert Y^{n} (t) \rVert_{\mathcal{C}^{-\kappa}} + \lVert Y^{n}(t) \rVert_{L^{2}} \lVert Y^{n}(t) \rVert_{\mathcal{C}^{2\kappa}} \overset{\eqref{est 115}}{\lesssim} (\bar{N}_{T}^{\kappa})^{2}.
\end{align}
\end{subequations} 
By an application of \eqref{est 119} directly on \eqref{Equation of wn}, we can deduce 
\begin{align}
\lVert \partial_{t} w^{n}(t) \rVert_{H^{-1 - 3 \kappa}} \lesssim& \lVert w^{n} (t) \rVert_{H^{1-3\kappa}}  \nonumber \\
&+ \lVert w^{n} (t) \rVert_{L^{2}}^{\frac{3-6\kappa}{2(1-3\kappa)}} \lVert w^{n} (t) \rVert_{\dot{H}^{1-3\kappa}}^{\frac{1-6\kappa}{2(1-3\kappa)}} + \bar{N}_{T}^{\kappa} \left( \lVert w^{n} (t) \rVert_{L^{2}}^{\frac{1-5\kappa}{1-3\kappa}} \lVert w^{n}(t) \rVert_{\dot{H}^{1-3\kappa}}^{\frac{2\kappa}{1-3\kappa}} + 1\right).\label{est 120} 
\end{align} 
Therefore, applying \eqref{est 117} and \eqref{est 118} to \eqref{est 120} gives us 
\begin{equation}\label{est 121}
 \sup_{n \in \mathbb{N}} \lVert \partial_{t} w^{n} \rVert_{L^{2} (0, T; H^{-1-3\kappa})}^{2} \leq C(T, \bar{N}_{T}^{\kappa}).
\end{equation} 
Considering \eqref{est 117}, \eqref{est 118} and now \eqref{est 121}, weak compactness, Banach-Alaoglu theorem, and Lions-Aubins compactness type result (e.g. \cite[Lemma 4 (1)]{S90}) allow us to find a subsequence $\{w^{n_{k}} \}_{k\in\mathbb{N}} \subset \{w^{n}\}_{n\in\mathbb{N}}$ and a limit $w \in L^{\infty} (0, T; L^{2} (\mathbb{T})) \cap L^{2} (0, T; H^{1- \frac{3\kappa}{2}}(\mathbb{T}))$ such that 
\begin{subequations}\label{convergence} 
\begin{align}
& w^{n_{k}} \overset{\ast}{\rightharpoonup} w \hspace{3mm} \text{weak}^{\ast} \text{ in } L^{\infty} (0, T; L^{2} (\mathbb{T})), \label{convergence 1} \\
& w^{n_{k}} \rightharpoonup w \hspace{3mm} \text{weakly in } L^{2} (0, T; H^{1- \frac{3\kappa}{2}} (\mathbb{T})),  \label{convergence 2} \\
& w^{n_{k}} \to w \hspace{3mm} \text{strongly in } L^{2} (0, T; H^{\beta} (\mathbb{T})) \hspace{3mm} \forall \hspace{1mm} \beta \in \left(-1-3\kappa, 1- \frac{3\kappa}{2} \right).  \label{convergence 3} 
\end{align}
\end{subequations} 
With these convergence results, it is a classical to prove that $w$ is a weak solution to \eqref{Equation of w}. Additionally, $Q^{n_{k}} \to Q$ strongly in $C([0,T]; \mathcal{C}^{2-3\kappa})$ as $k\to\infty$ so that $w^{n_{k}, \mathcal{H}, \lambda} \to w^{\mathcal{H}, \lambda}$ strongly in $L^{2} (0, T; H^{1-4\kappa})$ as $k\to\infty$. Moreover, from \eqref{est 126} we see that $w^{\mathcal{H}, \lambda} \in L^{2}(0, T; B_{\infty,2}^{1-2\kappa})$ as desired in \eqref{est 127b}. The fact that $w^{\mathcal{H}, \lambda} \in L^{\infty} (0, T; L^{2})$ as desired in \eqref{est 127b} follows from \eqref{est 128}, \eqref{est 117}, and the fact that $w^{n, \mathcal{H}, \lambda} = w^{n} - w^{n, \mathcal{L}, \lambda}$ from \eqref{est 124}.  Finally, \eqref{est 128} also implies \eqref{est 127a}. 
\end{proof} 

\begin{proposition}\label{Proposition 5.2}  
Let $\mathcal{N}''$ be the null set from Proposition \ref{Proposition 4.1}. Then, for any $\omega \in \Omega \setminus \mathcal{N}''$ and any $\theta^{\text{in}} \in L^{2} (\mathbb{T})$ that is mean-zero, there exists at most one HL weak solution starting from $\theta^{\text{in}}$. 
\end{proposition} 

\begin{proof}[Proof of Proposition \ref{Proposition 5.2}]
Suppose that $v \triangleq w + Y, \bar{v} \triangleq \bar{w} + Y$ are two HL weak solutions. 
We define 
\begin{equation}\label{est 130}
z \triangleq w - \bar{w}, \hspace{3mm} z^{\mathcal{L}, \lambda} \triangleq w^{\mathcal{L}, \lambda} - \bar{w}^{\mathcal{L}, \lambda}, \hspace{1mm} \text{ and } \hspace{1mm}  \lVert f \rVert_{\lambda} \triangleq \lVert f^{\mathcal{L},\lambda} \rVert_{H^{1}} + \lVert f^{\mathcal{H}, \lambda} \rVert_{B_{\infty, 2}^{1-2\kappa}}.
\end{equation} 
It follows that $z^{\mathcal{L},\lambda}$ satisfies 
\begin{align}
\partial_{t} z^{\mathcal{L}, \lambda} - \nu \partial_{x}^{2} & z^{\mathcal{L}, \lambda} = - \partial_{x} ( z^{\mathcal{L},\lambda} \mathcal{L}_{\lambda} X) - \partial_{x}( z^{\mathcal{L},\lambda} \mathcal{H}_{\lambda} X - z^{\mathcal{L},\lambda} \prec \mathcal{H}_{\lambda} X)  \nonumber \\
& - \partial_{x} \left( z^{\mathcal{H},\lambda} X - z^{\mathcal{H},\lambda} \prec \mathcal{H}_{\lambda} X - \frac{1}{2} C^{\prec} (z, \mathcal{H}_{\lambda} Q) \right) - \frac{1}{2} \partial_{x} (w^{2} - \bar{w}^{2}+ 2zY). \label{est 131}
\end{align} 
Taking $L^{2}(\mathbb{T})$-inner products on \eqref{est 131} with $z^{\mathcal{L},\lambda}$ leads to 
\begin{equation}\label{est 145} 
\frac{1}{2} \partial_{t} \lVert z^{\mathcal{L}, \lambda} (t) \rVert_{L^{2}}^{2} = \sum_{k=1}^{4}\RomanIII_{k}, 
\end{equation} 
where 
\begin{subequations}
\begin{align}
\RomanIII_{1} \triangleq& - \langle z^{\mathcal{L},\lambda}, - \nu \partial_{x}^{2} z^{\mathcal{L},\lambda} + \partial_{x} (z^{\mathcal{L},\lambda} \mathcal{L}_{\lambda} X) \rangle_{L^{2}} (t), \label{Define III1}\\
\RomanIII_{2} \triangleq& - \langle z^{\mathcal{L},\lambda}, \partial_{x} (z^{\mathcal{L},\lambda} \mathcal{H}_{\lambda} X - z^{\mathcal{L},\lambda} \prec \mathcal{H}_{\lambda} X ) \rangle_{L^{2}} (t), \label{Define III2}\\
\RomanIII_{3} \triangleq& - \left\langle z^{\mathcal{L},\lambda}, \partial_{x} \left(z^{\mathcal{H},\lambda} X - z^{\mathcal{H},\lambda} \prec \mathcal{H}_{\lambda} X - \frac{1}{2} C^{\prec} (z, \mathcal{H}_{\lambda} Q)  \right) \right\rangle_{L^{2}} (t), \label{Define III3} \\
\RomanIII_{4} \triangleq& - \left\langle z^{\mathcal{L},\lambda}, \frac{1}{2} \partial_{x} (w^{2} - \bar{w}^{2} + 2z Y) \right\rangle_{L^{2}} (t). \label{Define III4}
\end{align}
\end{subequations} 
We estimate for any $\kappa \in (0, \frac{1}{2})$, 
\begin{subequations}\label{est 132}
\begin{align}
&\RomanIII_{1} \lesssim  \lVert \partial_{x} z^{\mathcal{L},\lambda} (t) \rVert_{L^{2}} \lVert z^{\mathcal{L},\lambda} (t) \rVert_{L^{2}} \lVert \mathcal{L}_{\lambda} X(t) \rVert_{L^{\infty}} \nonumber \\
& \hspace{40mm}  \leq - \frac{15 \nu}{16} \lVert z^{\mathcal{L},\lambda}(t) \rVert_{\dot{H}^{1}}^{2} + C(\lambda, N_{T}^{\kappa}) \lVert z^{\mathcal{L},\lambda}(t) \rVert_{L^{2}}^{2}, \label{est 132a} \\
&\RomanIII_{2}  \overset{\eqref{Bony 4} \eqref{Bony 5}}{\lesssim} \lVert z^{\mathcal{L},\lambda} (t) \rVert_{\dot{H}^{1}} \left(  \lVert z^{\mathcal{L},\lambda}(t) \rVert_{H^{\kappa}} \lVert \mathcal{H}_{\lambda} X(t) \rVert_{\mathcal{C}^{-\kappa}} + \lVert z^{\mathcal{L},\lambda}(t) \rVert_{H^{2\kappa}} \lVert \mathcal{H}_{\lambda} X(t) \rVert_{\mathcal{C}^{-\kappa}} \right)  \nonumber \\
& \hspace{40mm} \leq \frac{\nu}{16} \lVert z^{\mathcal{L},\lambda} (t) \rVert_{\dot{H}^{1}}^{2} +C (N_{T}^{\kappa}) \lVert z^{\mathcal{L},\lambda}(t) \rVert_{L^{2}}^{2}. \label{est 132b}  
\end{align}
\end{subequations} 
We use the reformulation of the commutator from \eqref{est 46a} and estimate $\RomanIII_{3}$ from \eqref{Define III3} as 
\begin{align}
\RomanIII_{3}  \lesssim& \lVert z^{\mathcal{L},\lambda}(t) \rVert_{\dot{H}^{1}} \Bigg[ \lVert  z^{\mathcal{H},\lambda} X - z^{\mathcal{H},\lambda}\prec \mathcal{H}_{\lambda} X \rVert_{L^{2}}  \nonumber \\
& \hspace{25mm} + \lVert ( \partial_{t} - \nu \partial_{x}^{2}) z \prec \mathcal{H}_{\lambda} Q \rVert_{L^{2}} + \lVert \partial_{x} z \prec \partial_{x} \mathcal{H}_{\lambda} Q \rVert_{L^{2}} \Bigg](t). \label{est 133} 
\end{align}
We estimate the first term in \eqref{est 133} as follows:
\begin{align}
& \lVert ( z^{\mathcal{H}, \lambda} X - z^{\mathcal{H},\lambda} \prec \mathcal{H}_{\lambda} X ) (t)  \rVert_{L^{2}} = \lVert ( z^{\mathcal{H},\lambda} \mathcal{L}_{\lambda} X + z^{\mathcal{H},\lambda} \succ \mathcal{H}_{\lambda} X + z^{\mathcal{H},\lambda} \circ \mathcal{H}_{\lambda} X ) (t) \rVert_{L^{2}}  \nonumber \\
\overset{\eqref{Bony 4} \eqref{Bony 5}}{\lesssim}& \lVert z^{\mathcal{H},\lambda}(t) \rVert_{L^{2}} \lVert \mathcal{L}_{\lambda} X(t) \rVert_{L^{\infty}}+ \lVert z^{\mathcal{H},\lambda}(t) \rVert_{H^{\kappa}} \lVert \mathcal{H}_{\lambda} X(t) \rVert_{\mathcal{C}^{-\kappa}} + \lVert z^{\mathcal{H},\lambda} (t) \rVert_{H^{2\kappa}} \lVert \mathcal{H}_{\lambda} X(t) \rVert_{\mathcal{C}^{-\kappa}}  \nonumber \\
& \hspace{30mm} \lesssim \left( \lVert z^{\mathcal{H},\lambda}(t) \rVert_{L^{2}} \lambda^{2\kappa} + \lVert z^{\mathcal{H}, \lambda} (t) \rVert_{H^{2\kappa}}  \right)N_{T}^{\kappa}.  \label{est 135} 
\end{align}
To treat the second term in \eqref{est 133}, we first rewrite by using \eqref{Equation of w} and \eqref{est 130}, for all $\kappa \in (0, \frac{1}{3})$, 
\begin{equation}\label{est 134} 
( \partial_{t} - \nu \partial_{x}^{2}) z = - \frac{1}{2} \partial_{x} \left( z w + \bar{w} z + 2z Y + 2z X  \right) 
\end{equation} 
to estimate by 
\begin{align}
& \lVert ( \partial_{t} - \nu \partial_{x}^{2}) z \prec \mathcal{H}_{\lambda} Q (t) \rVert_{L^{2}} \overset{\eqref{Bony 3} \eqref{est 134}}{\lesssim}  \lVert \partial_{x} ( zw + \bar{w} z + 2zY + 2zX) (t) \rVert_{H^{-2+ \frac{3\kappa}{2}}} \lVert Q(t) \rVert_{\mathcal{C}^{2- \frac{3\kappa}{2}}} \nonumber \\
& \hspace{15mm} \overset{\eqref{Bony 3} \eqref{Bony 4} \eqref{Bony 5}}{\lesssim} N_{T}^{\kappa} \lVert z(t) \rVert_{L^{2}} ( \lVert w(t) \rVert_{L^{2}} + \lVert \bar{w}(t) \rVert_{L^{2}}) + \lVert z(t) \rVert_{H^{2\kappa}} (N_{T}^{\kappa})^{2}.\label{est 136} 
\end{align}
Lastly, we estimate the third term in \eqref{est 133} as follows: for all $\kappa \in (0, \frac{2}{3})$, 
\begin{equation}\label{est 137}
\lVert \partial_{x} z \prec \partial_{x} \mathcal{H}_{\lambda} Q(t) \rVert_{L^{2}} 
\overset{\eqref{Bony 3}}{\lesssim} \lVert \partial_{x} z(t) \rVert_{H^{-1+ \frac{3\kappa}{2}}} \lVert \partial_{x} \mathcal{H}_{\lambda} Q(t) \rVert_{\mathcal{C}^{1- \frac{3\kappa}{2}}}  \overset{\eqref{Estimate on Q}}{\lesssim} \lVert z (t) \rVert_{H^{\frac{3\kappa}{2}}}  N_{T}^{\kappa}. 
\end{equation} 
Applying \eqref{est 135}, \eqref{est 136}, and \eqref{est 137} to \eqref{est 133} gives us 
\begin{equation}\label{est 146}
\RomanIII_{3}\leq \frac{\nu}{16} \lVert z^{\mathcal{L},\lambda}(t) \rVert_{\dot{H}^{1}}^{2} + C(N_{T}^{\kappa},\lambda) [ \lVert z^{\mathcal{H},\lambda} (t) \rVert_{H^{2\kappa}}^{2} + \lVert z(t) \rVert_{H^{2\kappa}}^{2} ( \lVert w(t) \rVert_{L^{2}} + \lVert \bar{w} (t) \rVert_{L^{2}} + 1)^{2} ].
\end{equation}
At last, we work on $\RomanIII_{4}$ from \eqref{Define III4} as follows: 
\begin{equation}\label{est 143}
\RomanIII_{4}  \lesssim \lVert z^{\mathcal{L},\lambda}(t) \rVert_{\dot{H}^{1}} \left( \lVert zw \rVert_{L^{2}} + \lVert \bar{w} z \rVert_{L^{2}} + \lVert z \rVert_{L^{2}} N_{T}^{\kappa} \right) (t). 
\end{equation} 
We expand $z, w, \bar{w},$ and $\bar{z}$ and first estimate the product of lower-order terms as follows: via H$\ddot{\mathrm{o}}$lder's inequality, Sobolev embedding $H^{1}(\mathbb{T}) \hookrightarrow L^{\infty} (\mathbb{T})$, and \eqref{est 130}
\begin{equation}\label{est 141}
\lVert z^{\mathcal{L},\lambda} w^{\mathcal{L},\lambda}  \rVert_{L^{2}} + \lVert \bar{w}^{\mathcal{L},\lambda} z^{\mathcal{L},\lambda} \rVert_{L^{2}}  \lesssim \lVert z^{\mathcal{L}, \lambda} \rVert_{L^{2}} ( \lVert w \rVert_{\lambda} + \lVert \bar{w} \rVert_{\lambda}). 
\end{equation} 
Next, among the products of higher and lower order terms, we focus on those with $z^{\mathcal{L},\lambda}$: 
\begin{equation}\label{est 142} 
 \lVert z^{\mathcal{L},\lambda} w^{\mathcal{H},\lambda} \rVert_{L^{2}} + \lVert \bar{w}^{\mathcal{H},\lambda} z^{\mathcal{L}, \lambda} \rVert_{L^{2}}   \overset{\eqref{est 130}}{\lesssim} \lVert z^{\mathcal{L},\lambda} \rVert_{L^{2}} ( \lVert w \rVert_{\lambda} + \lVert \bar{w} \rVert_{\lambda}).
\end{equation} 
Next, we estimate all of the rest of the products via embeddings of $H^{1}(\mathbb{T}) \hookrightarrow L^{\infty}(\mathbb{T})$ and $B_{\infty,2}^{1-2\kappa}(\mathbb{T}) \hookrightarrow B_{\infty,1}^{0}(\mathbb{T}) \hookrightarrow L^{\infty}(\mathbb{T})$, 
\begin{align}
&\lVert z^{\mathcal{H}, \lambda} w^{\mathcal{L},\lambda} \rVert_{L^{2}} + \lVert \bar{w}^{\mathcal{L}, \lambda} z^{\mathcal{H},\lambda} \rVert_{L^{2}} + \lVert z^{\mathcal{H},\lambda} w^{\mathcal{H},\lambda} \rVert_{L^{2}} + \lVert \bar{w}^{\mathcal{H},\lambda} z^{\mathcal{H},\lambda} \rVert_{L^{2}} \nonumber \\
\lesssim& \lVert z^{\mathcal{H}, \lambda} \rVert_{L^{2}} [ \lVert w^{\mathcal{L},\lambda} \rVert_{H^{1}} + \lVert \bar{w}^{\mathcal{L},\lambda} \rVert_{H^{1}} + \lVert w^{\mathcal{H},\lambda} \rVert_{B_{\infty, 2}^{1-2\kappa}} + \lVert \bar{w} \rVert_{B_{\infty,2}^{1-2\kappa}}] \lesssim \lVert z^{\mathcal{H},\lambda} \rVert_{L^{2}}  ( \lVert w \rVert_{\lambda} + \lVert \bar{w} \rVert_{\lambda}). \label{est 138}
\end{align} 
Next, for all $s \in [0, 1- 2 \kappa)$ we can compute for $\lambda \geq \bar{\lambda} \vee \lambda_{T}$ where $\lambda_{T}$ is from Definition \ref{Definition 5.1} (1) and $\bar{\lambda}$ is taken large, 
\begin{equation*}
\lVert z(t) \rVert_{H^{s}} \overset{\eqref{Bony 3}}{\leq} \lVert z^{\mathcal{L},\lambda}(t) \rVert_{H^{s}}  + C \lVert z \rVert_{H^{s}} \lambda^{-\frac{\kappa}{2}} \lVert Q (t) \rVert_{\mathcal{C}_{x}^{2-  \frac{3\kappa}{2}}}  \overset{\eqref{Estimate on Q}}{\leq}  \lVert z^{\mathcal{L},\lambda}(t) \rVert_{H^{s}}  + \frac{1}{2}\lVert z \rVert_{H^{s}} 
\end{equation*} 
and consequently
\begin{align}\label{est 139} 
\lVert z \rVert_{H^{s}} \leq 2 \lVert z^{\mathcal{L}, \lambda} \rVert_{H^{s}} \hspace{3mm} \forall \hspace{1mm} s \in [0, 1- 2 \kappa).
\end{align}
Applying \eqref{est 139} with $s = 0$ to \eqref{est 138} leads us to 
\begin{equation}\label{est 140}
\lVert z^{\mathcal{H}, \lambda} w^{\mathcal{L},\lambda} \rVert_{L^{2}} + \lVert \bar{w}^{\mathcal{L}, \lambda} z^{\mathcal{H},\lambda} \rVert_{L^{2}} + \lVert z^{\mathcal{H},\lambda} w^{\mathcal{H},\lambda} \rVert_{L^{2}} + \lVert \bar{w}^{\mathcal{H},\lambda} z^{\mathcal{H},\lambda} \rVert_{L^{2}}  \lesssim \lVert z^{\mathcal{L},\lambda} \rVert_{L^{2}}  ( \lVert w \rVert_{\lambda} + \lVert \bar{w} \rVert_{\lambda}). 
\end{equation} 
Due to \eqref{est 141}, \eqref{est 142}, and \eqref{est 140}, we are able to deduce 
\begin{equation}\label{est 144} 
\lVert zw \rVert_{L^{2}}+ \lVert \bar{w} z \rVert_{L^{2}} \lesssim \lVert z^{\mathcal{L},\lambda} \rVert_{L^{2}}  ( \lVert w \rVert_{\lambda} + \lVert \bar{w} \rVert_{\lambda}). 
\end{equation} 
Applying \eqref{est 144} and \eqref{est 139} with $s = 0$ to \eqref{est 143} allows us to conclude 
\begin{equation}\label{est 147}
\RomanIII_{4}  \leq \frac{\nu}{16} \lVert z^{\mathcal{L},\lambda} (t) \rVert_{\dot{H}^{1}}^{2} + C \lVert z^{\mathcal{L},\lambda}(t) \rVert_{L^{2}}^{2} \left( \lVert w(t) \rVert_{\lambda}^{2} + \lVert \bar{w}(t) \rVert_{\lambda}^{2} + (N_{T}^{\kappa} )^{2} \right).
\end{equation} 
At last, we obtain for 
\begin{equation}\label{est148} 
M_{T} \triangleq \lVert w^{\mathcal{L},\lambda} \rVert_{L_{T}^{\infty} L_{x}^{2}}^{2} + \lVert w^{\mathcal{H},\lambda} \rVert_{L_{T}^{\infty} L_{x}^{2}}^{2} + \lVert \bar{w}^{\mathcal{L},\lambda} \rVert_{L_{T}^{\infty} L_{x}^{2}}^{2} + \lVert \bar{w}^{\mathcal{H},\lambda} \rVert_{L_{T}^{\infty} L_{x}^{2}}^{2} 
\end{equation} 
which is finite by Definition \ref{Definition 5.1} (1), 
\begin{align}
&\partial_{t} \lVert z^{\mathcal{L},\lambda}(t) \rVert_{L^{2}}^{2} + \nu\lVert z^{\mathcal{L},\lambda}(t) \rVert_{\dot{H}^{1}}^{2}  \nonumber \\
\leq&C(\lambda, N_{T}^{\kappa}) \lVert z^{\mathcal{L},\lambda} (t) \rVert_{L^{2}}^{2} \left( ( \lVert w \rVert_{L^{2}} + \lVert \bar{w} \rVert_{L^{2}} + 1)^{\frac{2}{1-2\kappa}}  + \lVert w \rVert_{\lambda}^{2} + \lVert \bar{w} \rVert_{\lambda}^{2} + (N_{T}^{\kappa})^{2} \right), \label{est 273}
\end{align} 
where we used $\lVert z^{\mathcal{H},\lambda} \rVert_{H^{2\kappa}}^{2} \leq 9 \lVert z^{\mathcal{L},\lambda} \rVert_{H^{2\kappa}}^{2}$ for $\kappa \in (0, \frac{1}{4})$ due to \eqref{est 139}. This completes the proof of Proposition \ref{Proposition 5.2}.  
\end{proof} 

The proof of Theorem \ref{Theorem 2.3} now follows similarly to the proof of Theorem \ref{Theorem 2.2}. 

\section{Proof of Proposition \ref{Proposition 2.4}}\label{Section 6}
In this section we discuss the proof of Proposition \ref{Proposition 2.4}. It is an analogue of \cite[Proposition 6.1]{HR24}, the proof of which was referred to    \cite{AC15} such as Proposition 4.13 therein. We point out in this section that the relevant works of \cite{AC15} that imply our Proposition \ref{Proposition 2.4} depend on spatial dimensions only \emph{initially}. In other words, spatial dimensions play the role in determining the regularity of the force initially, and once the force is fixed with such a regularity, the proofs of relevant results are independent of spatial dimensions. We state the results with only a sketch of their proofs, especially where we diverge a bit from \cite{AC15}, due to the overall similarities with \cite{AC15}. We consider the case $x \in \mathbb{T}^{d}$ for $d\in \mathbb{N}$, a temporal variable, if there is any, and a parameter $\kappa > 0$ sufficiently small all considered fixed throughout this Section \ref{Section 6}.

\begin{define}\rm{(Cf. \cite[Definition 4.5]{AC15} and \cite[Definition 2.2]{GUZ20})}
We define 
\begin{subequations}
\begin{align}
&\sigma(D) f \triangleq \mathcal{F}^{-1} (\sigma \mathcal{F} f),   \label{Define sigma D}  \\
& \mathcal{E}^{\alpha} \triangleq \mathcal{C}^{\alpha} \times \mathcal{C}^{2 \alpha + 2}, \hspace{2mm} \alpha \in \mathbb{R}, \label{Define E alpha}
\end{align}
\end{subequations} 
and the space of enhanced noise,  
\begin{equation}\label{Define K alpha}
\mathcal{K}^{\alpha} \triangleq \overline{\{ (\eta, -\eta \circ \sigma(D) \eta -c ): \hspace{1mm} \eta \in C^{\infty}, c \in \mathbb{R} \}}, \hspace{1mm} \text{ where } \hspace{1mm} \sigma(D) \triangleq - ( 1-  \Delta)^{-1}, 
\end{equation} 
where the closure is taken w.r.t. $\mathcal{E}^{\alpha}$-topology (cf. \eqref{est 21}). A general element of $\mathcal{K}^{\alpha}$ will be denoted by $\Theta \triangleq (\Theta_{1}, \Theta_{2})$. If $\eta \in \mathcal{C}^{\alpha}$ is $\Theta_{1}$, then $\Theta$ is said to be an enhancement (lift) of $\eta$. 
\end{define} 
The case $d= 1$, $\alpha$ = $-1 - \kappa$, and $\nu =1$ applies to Proposition \ref{Proposition 2.4}; recall $\partial_{x} X \in C_{t} \mathcal{C}_{x}^{-1-\kappa}$ $\mathbb{P}$-a.s. due to \eqref{Define Lt and Nt} and Proposition \ref{Proposition 4.1}.

\begin{define}\label{Definition 6.2}
\rm{(Cf. \cite[Definition 4.1]{AC15} and \cite[Definition 2.4]{GUZ20})} Let $\alpha < -1$ and $\eta \in \mathcal{C}^{\alpha}$. For $\gamma \leq \alpha + 2$, we define the space of distributions which are paracontrolled by $\sigma(D) \eta$ as 
\begin{equation}\label{Define D eta gamma}
\mathcal{D}_{\eta}^{\gamma} \triangleq \{ f \in H^{\gamma}: \hspace{1mm} f^{\sharp} \triangleq f - f \prec \sigma(D) \eta \in H^{2\gamma}  \}. 
\end{equation} 
The space $\mathcal{D}_{\eta}^{\gamma}$, equipped with the following scalar product, is a Hilbert space: 
\begin{equation}\label{D eta gamma norm}
\langle f,g \rangle_{\mathcal{D}_{\eta}^{\gamma}} \triangleq \langle f,g \rangle_{H^{\gamma}} + \langle f^{\sharp}, g^{\sharp} \rangle_{H^{2\gamma}} \hspace{3mm} \forall \hspace{1mm} f, g \in \mathcal{D}_{\eta}^{\gamma}. 
\end{equation} 
\end{define} 

The following definition, especially the product $\eta \circ f$ within $\eta f$, can be justified by Proposition \ref{Proposition 6.1} (1). 
\begin{define}\label{Definition 6.3}
\rm{(Cf. \cite[Definition 4.10]{AC15})} Let $\alpha \in (-\frac{4}{3}, -1)$, $\gamma \in (-\frac{\alpha}{2}, \alpha + 2]$, and $\Theta = (\eta, \Theta_{2}) \in \mathcal{K}^{\alpha}$. We define the linear operator 
\begin{equation}\label{Define H}
\mathcal{H}: \hspace{1mm} \mathcal{D}_{\eta}^{\gamma} \mapsto H^{\gamma -2} \text{ by } \mathcal{H}f \triangleq  \Delta f - \eta f \text{ where } \eta f = \eta \prec f + \eta \succ f + \eta \circ f  
\end{equation} 
(cf. \eqref{est 23}). 
\end{define} 

\begin{proposition}\label{Proposition 6.1}\rm{(Cf. \cite[Proposition 4.8]{AC15} and \cite[Corollary 4.9]{AC15})}
\indent
\begin{enumerate}
\item Let $\alpha \in (-\frac{4}{3}, -1)$ and $\gamma \in (-\frac{\alpha}{2}, \alpha + 2]$. Denote $\Theta = (\eta, \Theta_{2}) \in \mathcal{K}^{\alpha}$ as an enhancement of $\eta \in \mathcal{C}^{\alpha}$, and $f \in \mathcal{D}_{\eta}^{\gamma}$. Then we can define 
\begin{equation}\label{est 165}
f \circ \eta = f \Theta_{2} + \mathcal{R} (f, \sigma(D) \eta, \eta) + f^{\sharp} \circ \eta 
\end{equation} 
and we have the following bound:  
\begin{equation}\label{est 163} 
\lVert f \circ \eta \rVert_{H^{2 \alpha + 2 - \kappa}} \lesssim \lVert f \rVert_{\mathcal{D}_{\eta}^{\gamma}} \lVert \Theta \rVert_{\mathcal{E}^{\alpha}} ( 1+ \lVert \Theta \rVert_{\mathcal{E}^{\alpha}}). 
\end{equation} 
\item Let $\alpha \in (-\frac{4}{3}, -1), \gamma \in (\frac{2}{3}, \alpha + 2)$, $\Theta = (\eta, \Theta_{2}) \in \mathcal{K}^{\alpha}$, and $\{ \Theta^{n} \}_{n \in\mathbb{N}}$ where 
\begin{equation}
\Theta^{n} \triangleq ( \eta_{n}, -\eta_{n} \circ ( \sigma(D) \eta_{n}) - c_{n}), \hspace{3mm} c_{n} \in \mathbb{R}, 
\end{equation} 
a family of smooth functions such that 
\begin{equation}\label{est 169}
\Theta^{n} \to \Theta \hspace{1mm} \text{ in }\mathcal{E}^{\alpha} \hspace{1mm} \text{ as } n \nearrow \infty. 
\end{equation} 
Let $f^{n}$ be a smooth approximation of $f \in \mathcal{D}_{\eta}^{\gamma}$ such that 
\begin{equation}\label{est 170}
\lVert f - f_{n} \rVert_{H^{\gamma}} + \lVert f_{n}^{\sharp} - f^{\sharp} \rVert_{H^{2\gamma}} \to 0 \hspace{1mm} \text{ as } n \nearrow \infty, \hspace{1mm} \text{ where } f_{n}^{\sharp} \triangleq f_{n} - f_{n} \prec \sigma(D) \eta_{n}. 
\end{equation} 
Then, for all $\kappa > 0$, 
\begin{equation}\label{est 171} 
\lVert f_{n} \circ \eta_{n} - f \circ \eta \rVert_{H^{2 \alpha + 2 - \kappa}} \to 0 \hspace{3mm} \text{ as } n \nearrow \infty. 
\end{equation} 
\end{enumerate} 
\end{proposition}

\begin{proof}[Proof of Proposition \ref{Proposition 6.1}]
(1) To estimate the first term in \eqref{est 165}, we realize that $2 \alpha + 2 + \gamma> 0$ due to the hypotheses that $\gamma > - \frac{\alpha}{2}$ and $\alpha > - \frac{4}{3}$ and consequently due to \eqref{Bony 3}-\eqref{Bony 5}, 
\begin{equation}\label{est 166} 
\lVert f \Theta_{2} \rVert_{H^{2 \alpha + 2 - \kappa}} \lesssim \lVert f \prec \Theta_{2} \rVert_{H^{2 \alpha + 2 - \kappa}} + \lVert f \succ \Theta_{2} \rVert_{H^{2 \alpha + 2}} + \lVert f \circ \Theta_{2} \rVert_{H^{2 \alpha + 2 + \gamma}} \lesssim \lVert f \rVert_{\mathcal{D}_{\eta}^{\gamma}} \lVert \Theta \rVert_{\mathcal{E}^{\alpha}}.
\end{equation} 
We estimate the second and third terms in \eqref{est 165} by 
\begin{subequations}
\begin{align}
&\lVert \mathcal{R} (f, \sigma(D)\eta, \eta) \rVert_{H^{2 \alpha + 2}}   \overset{\eqref{Estimate on R}}{\lesssim} \lVert f \rVert_{H^{\gamma}} \lVert \sigma(D) \eta \rVert_{\mathcal{C}^{\alpha + 2}} \lVert \eta \rVert_{\mathcal{C}^{\alpha}}  \overset{\eqref{D eta gamma norm} \eqref{Define E alpha}}{\lesssim} \lVert f \rVert_{\mathcal{D}_{\eta}^{\gamma}} \lVert \Theta \rVert_{\mathcal{E}^{\alpha}}^{2}, \label{est 167}\\
&\lVert f^{\sharp} \circ \eta \rVert_{H^{2 \alpha + 2}} \lesssim \lVert f^{\sharp} \circ \eta \rVert_{H^{\alpha + 2 \gamma}} \overset{\eqref{Bony 5}}{\lesssim} \lVert f^{\sharp} \rVert_{H^{2\gamma}} \lVert \eta \rVert_{\mathcal{C}^{\alpha}}  \overset{\eqref{Define E alpha} \eqref{D eta gamma norm}}{\lesssim} \lVert f \rVert_{\mathcal{D}_{\eta}^{\gamma}} \lVert\Theta \rVert_{\mathcal{E}^{\alpha}}. \label{est 168}
\end{align}
\end{subequations}
Considering \eqref{est 166}, \eqref{est 167}, and \eqref{est 168} in \eqref{est 165} verifies \eqref{est 163}. 

(2) Similar computations to the proof of part (1) leads to 
\begin{align}
& \lVert f_{n} \circ \eta_{n} - f \circ \eta \rVert_{H^{2 \alpha + 2 - \kappa}}  \nonumber \\
\overset{\eqref{est 165}}{\lesssim}& \lVert (f_{n} - f) \circ \eta_{n} \rVert_{H^{2\alpha + 2 - \kappa}} + \lVert f(\Theta_{2}^{n} - \Theta_{2})  \rVert_{H^{2 \alpha + 2 - \kappa}}    + \lVert \mathcal{R} (f, \sigma(D) (\eta_{n} - \eta), \eta_{n}) \rVert_{H^{2 \alpha + 2 - \kappa}}  \nonumber \\
& \hspace{10mm} + \lVert \mathcal{R} (f, \sigma(D)\eta, \eta_{n} - \eta) \rVert_{H^{2 \alpha + 2 - \kappa}} + \lVert f^{\sharp} \circ (\eta_{n} - \eta) \rVert_{H^{2 \alpha + 2 - \kappa}}   \overset{\eqref{est 169} \eqref{est 170}}{\to} 0 
\end{align}
as $n \nearrow 0$ which verifies \eqref{est 171}. 
\end{proof} 

\begin{proposition}\label{Proposition 6.2} 
\rm{(Cf. \cite[Lemma 4.12]{AC15})} Let $\alpha \in (-\frac{4}{3}, -1), \eta \in \mathcal{C}^{\alpha}$, and $\gamma \in (\frac{2}{3}, \alpha + 2)$. Then $\mathcal{D}_{\eta}^{\gamma}$ is dense in $L^{2}$. 
\end{proposition}

\begin{proof}[Proof of Proposition \ref{Proposition 6.2} ]
We fix an arbitrary $g \in C^{\infty}$ and define a Fourier multiplier 
\begin{equation}\label{Define sigma a}
\sigma_{a}(k) \triangleq - \frac{1}{1+ a + \lvert k \rvert^{2}} \text{ for } a > 0
\end{equation} 
(cf. $\sigma(D)$ in \eqref{Define K alpha}) and consider a map 
\begin{equation}\label{Define Gamma} 
\Gamma: \hspace{1mm} H^{\gamma}\mapsto H^{\gamma} \text{ defined by } \Gamma(f) \triangleq \sigma_{a}(D) (f\prec \eta) + g. 
\end{equation}
For any $k\in \mathbb{Z}^{d}$, multi-index $r \in \mathbb{N}_{0}^{d}$ and $\vartheta \in [0, 1]$, we have 
\begin{equation}\label{est 172} 
\lvert D^{r} \sigma_{a}(k) \rvert \lesssim \frac{a^{\vartheta -1}}{(1+ \lvert k \rvert)^{2 \vartheta + r}} \hspace{1mm} \text{ and } \hspace{1mm} \lvert D^{r} ( \sigma_{a} - \sigma) (k) \rvert \lesssim \frac{a^{\vartheta}}{1+ \lvert k \rvert^{2+ 2 \vartheta + r}}. 
\end{equation} 
Because $\gamma < \alpha + 2$ by hypothesis, we can find $\epsilon_{1} > 0$ sufficiently small so that 
\begin{equation}\label{Define first epsilon} 
\epsilon_{1} < \alpha + 2 - \gamma
\end{equation} 
and estimate via Lemma \ref{Lemma A.1} for $\vartheta =  \frac{\gamma + \epsilon_{1} - \alpha}{2} \in [0,1]$, for $a \geq A$ sufficiently large, 
\begin{equation}\label{est 184} 
 \lVert \Gamma(f_{1}) - \Gamma(f_{2}) \rVert_{H^{\gamma}} \overset{\eqref{Define first epsilon} \eqref{Schauder}}{\lesssim} a^{\frac{ \gamma + \epsilon_{1} - \alpha}{2} -1} \lVert (f_{1} - f_{2}) \prec \eta \rVert_{H^{ \alpha - \epsilon_{1}}} \overset{\eqref{Bony 1} \eqref{Define first epsilon}}{\ll} \lVert f_{1} - f_{2} \rVert_{L^{2}}, 
\end{equation} 
and thus $\Gamma$ is a contraction for all such large $a$ and therefore admits a unique fixed point $f_{a}$. An identical estimate in \eqref{est 184} shows that the fixed point $f_{a}$ satisfies 
\begin{equation}\label{est 185} 
\lVert f_{a} - g \rVert_{H^{\gamma}} \overset{\eqref{Schauder}}{\lesssim} a^{\frac{ \gamma + \epsilon_{1} - \alpha}{2} - 1} \lVert f_{a} \prec \eta \rVert_{H^{\alpha - \epsilon_{1}}} \overset{\eqref{Bony 1}}{\lesssim}  a^{ \frac{ \gamma + \epsilon_{1} - \alpha}{2} - 1} \lVert f_{a} \rVert_{H^{\gamma}} \lVert \eta  \rVert_{\mathcal{C}^{\alpha}};
\end{equation}
taking $a \geq A$ for $A \gg \lVert \eta \rVert_{\mathcal{C}^{\alpha}}^{\frac{1}{1- \frac{\gamma + \epsilon_{1} - \alpha}{2}}}$ in this inequality gives us $\sup_{a \geq A} \lVert f_{a} \rVert_{H^{\gamma}} \leq 2 \lVert g\rVert_{H^{\gamma}}$ and plugging this inequality back into the upper bound in \eqref{est 185} finally shows $\lVert f_{a} - g \rVert_{H^{\gamma}} \lesssim a^{ \frac{ \gamma + \epsilon_{1} - \alpha}{2} - 1} \lVert g \rVert_{H^{\gamma}} \lVert \eta  \rVert_{\mathcal{C}^{\alpha}}$ which allows us to conclude that $f_{a}$ converges to $g$ in $H^{\gamma}$ as $n\nearrow \infty$ and hence in $L^{2}$. To show that $f_{a} - f_{a} \prec \sigma(D) \eta \in H^{2\gamma}$, we first rewrite via \eqref{Define Gamma}, 
\begin{equation}\label{est 186}
 f_{a} - f_{a} \prec \sigma(D) \eta = \sigma_{a}(D) (f_{a} \prec \eta) - f_{a} \prec \sigma_{a}(D) \eta + f_{a} \prec \left( \sigma_{a}(D) - \sigma(D) \right)\eta + g 
\end{equation} 
and conclude via the estimates of 
\begin{subequations}\label{est 187}
\begin{align}
& \lVert \sigma_{a}(D) (f_{a} \prec \eta) - f_{a} \prec \sigma_{a}(D) \eta \rVert_{H^{2\gamma}}  \lesssim \lVert f_{a} \rVert_{H^{\gamma}} \lVert \eta \rVert_{\mathcal{C}^{\alpha}}, \\
&  \lVert f_{a} \prec \left(\sigma_{a} (D) -\sigma(D)\right) \eta \rVert_{H^{2\gamma}} \lesssim_{a} \lVert f_{a} \rVert_{H^{\gamma}} \lVert \eta \rVert_{\mathcal{C}^{\alpha}},
\end{align}
\end{subequations} 
where we used Lemma \ref{Lemma A.2} in the first estimate while \eqref{Bony 3} and \eqref{est 172} in the second estimate. Considering \eqref{est 186}-\eqref{est 187} allows us to conclude that $f_{a} - f_{a} \prec \sigma(D) \eta \in H^{2\gamma}$ and therefore $f_{a} \in \mathcal{D}_{\eta}^{\gamma}$ so that $\mathcal{D}_{\eta}^{\gamma}$ is dense in $C^{\infty}$, which implies the claim.  
\end{proof} 

\begin{proposition}\label{Proposition 6.3}
\rm{(Cf. \cite[Proposition 4.13]{AC15})} Define $\mathcal{H}$ by \eqref{Define H}. Let $\alpha \in (-\frac{4}{3}, -1), \gamma \in (\frac{2}{3}, \alpha + 2)$, 
\begin{equation}\label{Define rho}
\rho \in \left( \gamma - \frac{\alpha +2}{2}, 1 + \frac{\alpha}{2}\right), 
\end{equation} 
and $\Theta = ( \eta, \Theta_{2}) \in \mathcal{K}^{\alpha}$. Then there exists $A = A( \lVert \Theta \rVert_{\mathcal{E}^{\alpha}})$ such that for all $a \geq A$ and $g \in H^{2 \gamma - 2}$, 
\begin{equation}\label{est 188}
(-\mathcal{H} + a) f = g 
\end{equation} 
admits a unique solution $f_{a} \in \mathcal{D}_{\eta}^{\gamma}$. Additionally, the mapping 
\begin{equation}\label{Define mathcal G}
\mathcal{G}_{a}: \hspace{1mm} L^{2} \mapsto \mathcal{D}_{\eta}^{\gamma} \text{ for } a \geq A, \text{ defined by } \mathcal{G}_{a} g \triangleq f_{a}, 
\end{equation}  
is uniformly bounded; in fact, for all $g \in H^{-\delta}$, all $\delta \in [0, 2- 2 \gamma]$, 
\begin{subequations}\label{est 202} 
\begin{align}
&\lVert f_{a} \rVert_{H^{\gamma}} + a^{-\rho} \lVert f_{a}^{\sharp} \rVert_{H^{2\gamma}} \lesssim \left( a^{\frac{\gamma + \delta}{2} -1} + a^{-\rho + \gamma + \frac{\delta}{2} -1} \right) \lVert g \rVert_{H^{-\delta}}, \label{est 202a} \\
&\lVert \mathcal{G}_{a} g \rVert_{\mathcal{D}_{\eta}^{\gamma}} \lesssim \left(a^{\rho + \frac{\gamma + \delta}{2} - 1} + a^{-1 + \gamma + \frac{\delta}{2}} \right) \lVert g \rVert_{H^{-\delta}}. \label{est 202b} 
\end{align}
\end{subequations}
\end{proposition}

\begin{proof}[Proof of Proposition \ref{Proposition 6.3}] 
For any $A > 0$, we define the Banach space  
\begin{subequations}
\begin{align}
& \tilde{\mathcal{D}}_{\eta}^{\gamma, \rho, A} \triangleq \{ (f_{a}, f_{a}')_{a \geq A} \in C([A, \infty), H^{\gamma})^{2}:  \hspace{1mm} f_{a} \in \mathcal{D}_{\eta}^{\gamma}, \lVert (f, f') \rVert_{\tilde{D}_{\eta}^{\gamma, \rho, A}} < \infty \}, \label{Define tilde D}\\
\text{where } &\lVert (f, f') \rVert_{\tilde{D}_{\eta}^{\gamma, \rho, A}} \triangleq \sup_{a \geq A} \lVert f_{a}' \rVert_{H^{\gamma}} + \sup_{a\geq A} a^{-\rho} \lVert f_{a}^{\sharp} \rVert_{H^{2\gamma}}+ \sup_{a \geq A} \lVert f_{a} \rVert_{H^{\gamma}}, \label{Define tilde D norm}
\end{align}
\end{subequations}
and 
\begin{equation}\label{Define mathcal M}
\mathcal{M}(f, f') \triangleq (M(f, f'), f), \hspace{3mm} M(f, f')_{a} \triangleq \tilde{\sigma}_{a}(D) (f_{a} \eta - g) \hspace{3mm} \forall \hspace{1mm} (f, f') \in \tilde{D}_{\eta}^{\gamma, \rho, A}, 
\end{equation} 
where 
\begin{subequations}\label{Define tilde sigma a}
\begin{align}
&\tilde{\sigma}_{a}(D) \triangleq -\frac{1}{a+ \lvert k \rvert^{2}} \hspace{2mm} \text{ for } a > 2 \\
\text{that satisfies } &\lvert D^{r} \tilde{\sigma}_{a}(k) \rvert \lesssim \frac{a^{\vartheta -1}}{(1+ \lvert k \rvert)^{2 \vartheta + r}},  \hspace{5mm} \lvert D^{r} ( \tilde{\sigma}_{a} - \sigma) (k) \rvert \lesssim \frac{a^{\vartheta}}{1+ \lvert k \rvert^{2+ 2 \vartheta + r}}
\end{align}
\end{subequations} 
similarly to \eqref{est 172}, and the product $f_{a} \eta$ is justified via \eqref{est 165}. To find a solution to \eqref{est 188}, it suffices to prove that $\mathcal{M}$ admits a unique fixed point in $\tilde{\mathcal{D}}_{\eta}^{\gamma, \rho, A}$. To do so, the idea is to show that if $(f, f') \in \tilde{\mathcal{D}}_{\eta}^{\gamma, \rho, A}$, then 
\begin{equation}\label{est 194} 
M(f, f')_{a} \in H^{\gamma} \hspace{1mm} \text{ and } \hspace{1mm} M(f, f')^{\sharp} \triangleq M(f, f') - f \prec \sigma(D) \eta \in H^{2\gamma} 
\end{equation} 
so that $\mathcal{M}(f, f') \in \tilde{D}_{\eta}^{\gamma, \rho, A}$ allowing us to conclude that $\mathcal{M} (\tilde{\mathcal{D}}_{\eta}^{\gamma, \rho, A}) \subset \tilde{\mathcal{D}}_{\eta}^{\gamma, \rho, A}$. For this purpose, first, we need to extend the $\epsilon_{1} > 0$ from \eqref{Define first epsilon}. Because $\gamma < \alpha + 2$ by hypothesis and $\rho > \gamma - \frac{\alpha + 2}{2}$ by \eqref{Define rho},  we can find $\epsilon_{2} > 0$ sufficiently small such that 
\begin{equation}\label{Define epsilon two}
\epsilon_{2} < \min\left\{ \alpha + 2 - \gamma,  2 \left( 1+ \frac{\alpha}{2} + \rho - \gamma\right), \frac{1}{3}\right\}.
\end{equation} 
Then, with choices of 
\begin{equation*}
\vartheta = \frac{\gamma + \epsilon_{2} - \alpha}{2}, \hspace{2mm} \vartheta = - \frac{\alpha}{2}, \hspace{2mm} \vartheta = \frac{\gamma + \delta}{2} \in [0,1] 
\end{equation*} 
for any $\delta \in [0, 2 - 2 \gamma]$, we deduce 
\begin{subequations}\label{est 173}   
\begin{align}
& \lVert \tilde{\sigma}_{a}(D) (f_{a} \prec \eta) \rVert_{H^{\gamma}} \overset{\eqref{Define tilde sigma a} \eqref{Schauder}}{\lesssim} a^{\frac{ \gamma + \epsilon_{2} - \alpha}{2} -1} \lVert f_{a} \prec \eta \rVert_{H^{\alpha - \epsilon_{2}}} \overset{\eqref{Bony 3}}{\lesssim} a^{\frac{ \gamma + \epsilon_{2} - \alpha}{2} -1} \lVert f_{a} \rVert_{H^{\gamma}} \lVert \eta \rVert_{\mathcal{C}^{\alpha}}, \label{est 173a} \\
& \lVert \tilde{\sigma}_{a}(D) ( f_{a} \circ \eta + f_{a} \succ \eta ) \rVert_{H^{\gamma}}  \overset{\eqref{Define tilde sigma a} \eqref{Schauder}\eqref{Bony 3}}{\lesssim}  a^{- \frac{\alpha}{2} -1} \left( \lVert f_{a} \circ \eta \rVert_{H^{\gamma + \alpha}} + \lVert f_{a}  \rVert_{H^{\gamma}} \lVert \eta \rVert_{\mathcal{C}^{\alpha}}  \right),  \label{est 173b} \\
& \lVert \tilde{\sigma}_{a}(D) g \rVert_{H^{\gamma}} \overset{\eqref{Define tilde sigma a} \eqref{Schauder}}{\lesssim}  a^{\frac{\gamma + \delta}{2}  - 1} \lVert g \rVert_{H^{-\delta}}; \label{est 173c} 
\end{align}  
\end{subequations} 
here, in contrast to \cite[Equation (32c)]{AC15}, we crucially do not  bound $\lVert f_{a} \circ \eta \rVert_{H^{\gamma + \alpha}}$ within \eqref{est 173b} by $\lVert f_{a} \circ \eta \rVert_{H^{2 \alpha + 2}}$. To treat $\lVert f_{a} \circ \eta \rVert_{H^{\gamma + \alpha}}$ in \eqref{est 173b}, we write $f_{a} \circ \eta = f_{a} \circ \eta - f_{a}^{\sharp} \circ \eta + f_{a}^{\sharp} \circ \eta$ and estimate 
\begin{align}
&\lVert f_{a} \circ \eta - f_{a}^{\sharp} \circ \eta \rVert_{H^{\gamma + \alpha}} \lesssim \lVert f_{a}' \prec\Theta_{2} \rVert_{H^{\gamma + \alpha}} + \lVert f_{a}' \succ \Theta_{2} \rVert_{H^{\gamma + \alpha}} + \lVert f_{a}' \circ \Theta_{2} \rVert_{H^{\gamma + \alpha}} \nonumber \\
&  \hspace{18mm} + \lVert \mathcal{R} ( f_{a}', \sigma(D)\eta, \eta)  \rVert_{H^{\gamma + \alpha}}   \overset{\eqref{Bony 3}-\eqref{Bony 5} \eqref{Estimate on R}}{\lesssim} \lVert f_{a}' \rVert_{H^{\gamma}} ( \lVert \Theta_{2} \rVert_{\mathcal{C}^{2 \alpha + 2}} + \lVert \eta \rVert_{\mathcal{C}^{\alpha}}^{2}); \label{est 191}
\end{align}
applying \eqref{Bony 5} in the other term $f_{a}^{\sharp} \circ \eta$, we have shown in sum
\begin{equation}\label{est 193}
\lVert f_{a} \circ \eta \rVert_{H^{\gamma + \alpha}} \lesssim  \lVert f_{a}' \rVert_{H^{\gamma}} ( \lVert \Theta_{2} \rVert_{\mathcal{C}^{2 \alpha + 2}} + \lVert \eta \rVert_{\mathcal{C}^{\alpha}}^{2})  + a^{\rho} \left( \frac{ \lVert f_{a}^{\sharp} \rVert_{H^{2\gamma}}}{a^{\rho}} \right) \lVert \eta \rVert_{\mathcal{C}^{\alpha}}. 
\end{equation} 
Applying \eqref{est 173}, \eqref{est 193}, \eqref{Define tilde D norm} to \eqref{Define mathcal M} allows us to deduce 
\begin{align}
& \lVert M(f, f')_{a} \rVert_{H^{\gamma}} \nonumber \overset{\eqref{Define mathcal M}}{=} \lVert \tilde{\sigma}_{a}(D) (f_{a} \eta - g) \rVert_{H^{\gamma}} \nonumber \\ 
\lesssim& a^{\max \{ \frac{ \gamma + \epsilon_{2} - \alpha}{2} -1, \rho - \frac{\alpha}{2} -1 \}} \lVert (f, f') \rVert_{\tilde{\mathcal{D}}_{\eta}^{\gamma, \rho, A}} ( \lVert \eta \rVert_{\mathcal{C}^{\alpha}} + \lVert \eta \rVert_{\mathcal{C}^{\alpha}}^{2} + \lVert \Theta_{2} \rVert_{\mathcal{C}^{2 \alpha + 2}} )+a^{\frac{\gamma + \delta}{2} -1} \lVert g\rVert_{H^{-\delta}},  \label{est 198}
\end{align}
which implies $M(f, f')_{a} \in H^{\gamma}$, the first claim in \eqref{est 194}. 

Next, to show the second claim in \eqref{est 194}, namely that $M(f, f')^{\sharp} \in H^{2\gamma}$, we write from \eqref{est 194} and \eqref{Define mathcal M}, 
\begin{equation}\label{est 197}
M(f, f')_{a}^{\sharp} = \sum_{k=1}^{6} \RomanIV_{k} 
\end{equation} 
where  
\begin{subequations} 
\begin{align}
&\RomanIV_{1} \triangleq \mathcal{C}_{a}(f_{a}, \eta), \hspace{9mm} \RomanIV_{2} \triangleq  \tilde{\sigma}_{a}(D) (f_{a} \circ \eta - f_{a}^{\sharp} \circ \eta),  \hspace{1mm} \RomanIV_{3} \triangleq  \tilde{\sigma}_{a}(D) (f_{a}^{\sharp} \circ \eta), \label{Define IV1, IV2, and IV3}\\
& \RomanIV_{4} \triangleq \tilde{\sigma}_{a}(D) (f_{a}\succ \eta), \hspace{1mm} \RomanIV_{5} \triangleq - \tilde{\sigma}_{a}(D) g, \hspace{19mm} \RomanIV_{6} \triangleq f_{a} \prec \left( \tilde{\sigma}_{a}- \sigma \right)(D) \eta, \label{Define IV3, IV4, and IV5}
\end{align}
\end{subequations} 
and
\begin{equation}\label{Define mathcal Ca}
\mathcal{C}_{a}(f_{a}, \eta) \triangleq \tilde{\sigma}_{a}(D) (f_{a} \prec \eta) - f_{a}\prec \tilde{\sigma}_{a}(D) \eta
\end{equation} 
analogously to \eqref{Define mathcal C}. We can estimate using \eqref{est 195}, \eqref{Schauder}, \eqref{est 191}, and \eqref{Bony in Sobolev},
\begin{subequations}\label{est 196} 
\begin{align}
&\lVert \RomanIV_{1} \rVert_{H^{2\gamma}} \lesssim a^{\frac{ \gamma + \epsilon_{2} - \alpha}{2} - 1} \lVert f_{a} \rVert_{H^{\gamma}} \lVert \eta \rVert_{\mathcal{C}^{\alpha}},  \hspace{2mm} \lVert \RomanIV_{2}  \rVert_{H^{2\gamma}} \lesssim a^{\frac{\gamma - \alpha}{2} -1} \lVert f_{a}' \rVert_{H^{\gamma}} \left( \lVert \Theta_{2} \rVert_{\mathcal{C}^{2\alpha + 2}} + \lVert \eta \rVert_{\mathcal{C}^{\alpha}}^{2} \right), \label{est 196a}\\
&\lVert \RomanIV_{3} \rVert_{H^{2\gamma}}  \lesssim a^{-\frac{\alpha}{2} -1} \lVert f_{a}^{\sharp} \rVert_{H^{2\gamma}} \lVert \eta \rVert_{\mathcal{C}^{\alpha}}, \hspace{4mm}  \lVert \RomanIV_{4} \rVert_{H^{2\gamma}} \lesssim a^{\frac{\gamma - \alpha}{2} -1} \lVert f_{a} \rVert_{H^{\gamma}} \lVert \eta \rVert_{\mathcal{C}^{\alpha}}, \label{est 196b}\\
& \lVert \RomanIV_{5} \rVert_{H^{2\gamma}} \lesssim a^{\frac{ 2 \gamma + \delta}{2} - 1}\lVert g \rVert_{H^{-\delta}}, \hspace{12mm}  \lVert \RomanIV_{6} \rVert_{H^{2\gamma}} \lesssim a^{\gamma + \frac{\epsilon_{2} - \alpha}{2} -1} \lVert f_{a} \rVert_{H^{\gamma}} \lVert \eta \rVert_{\mathcal{C}^{\alpha}}, \label{est 196c}
\end{align} 
\end{subequations} 
where additionally \eqref{Define tilde sigma a} is used with choices of 
\begin{itemize}
\item $\vartheta = \frac{\gamma + \epsilon_{2} - \alpha}{2}$ for $\RomanIV_{1}$, 
\item $\vartheta = \frac{\gamma - \alpha}{2}$ for $\RomanIV_{2}$ and $\RomanIV_{4}$, 
\item $\vartheta = -\frac{\alpha}{2}$ for $\RomanIV_{3}$, 
\item $\vartheta = \frac{2 \gamma + \delta}{2}$ where $\delta \in [0, 2-2\gamma]$ for $\RomanIV_{5}$, 
\item $\vartheta = \gamma + \frac{\epsilon_{2} - \alpha}{2} - 1$ for $\RomanIV_{6}$, 
\end{itemize}
all of which lie in $[0,1]$. Applying \eqref{est 196} to  \eqref{est 197}, and using \eqref{Define tilde D norm} give us
\begin{align}
& a^{-\rho} \lVert M(f, f')_{a}^{\sharp} \rVert_{H^{2\gamma}}  \label{est 199}\\
\lesssim& \lVert (f, f' ) \rVert_{\tilde{\mathcal{D}}_{\eta}^{\gamma, \rho, A}} \Bigg[ a^{\frac{ \gamma - \alpha}{2} - 1 - \rho} ( \lVert \Theta_{2} \rVert_{\mathcal{C}^{2 \alpha + 2}} +\lVert \eta \rVert_{\mathcal{C}^{\alpha}}^{2} ) + a^{ \max \{ \gamma + \frac{\epsilon_{2} - \alpha}{2} - 1 - \rho, - \frac{\alpha}{2}  -1 \}} \lVert \eta \rVert_{\mathcal{C}^{\alpha}} \Bigg] + a^{\gamma + \frac{\delta}{2} - 1-  \rho } \lVert g \rVert_{H^{-\delta}}. \nonumber 
\end{align}
Consequently, applying \eqref{est 198} and \eqref{est 199}, using \eqref{Define epsilon two}, and taking $\delta = 0$ for convenience leads us to 
\begin{align}
& \lVert \mathcal{M} (f, f') \rVert_{\tilde{\mathcal{D}}_{\eta}^{\gamma, \rho, A}} \lesssim  \sup_{a \geq A} \lVert f_{a} \rVert_{H^{\gamma}}  \label{est 200} \\
&\hspace{3mm}+  \sup_{a \geq A} a^{- \lambda} \lVert (f, f') \rVert_{\tilde{\mathcal{D}}_{\eta}^{\gamma, \rho, A}} ( 1+ \lVert \Theta \rVert_{\mathcal{E}^{\alpha}}^{2})  + \left( A^{\gamma- 1 - \rho} + A^{\frac{\gamma}{2} - 1} \right) \lVert g \rVert_{L^{2}},  \nonumber 
\end{align}
where 
\begin{equation}\label{Define lambda}
\lambda \triangleq \min \left\{\rho + 1 - \gamma - \frac{\epsilon_{2} - \alpha}{2},  1 + \frac{\alpha - \gamma - \epsilon_{2}}{2}, 1 + \frac{\alpha}{2} - \rho \right\} > 0
\end{equation} 
due to \eqref{Define rho} and \eqref{Define epsilon two}, and therefore we conclude that $\mathcal{M}(f,f') \in \tilde{\mathcal{D}}_{\eta}^{\gamma, \rho, A}$. Similarly to \eqref{est 200}, we can show 
\begin{align}
 \lVert \mathcal{M} (f, f') - \mathcal{M} (h, h') \rVert_{\tilde{\mathcal{D}}_{\eta}^{\gamma, \rho, A}} \lesssim&  \sup_{a \geq A} \lVert f_{a} - h_{a} \rVert_{H^{\gamma}}  \label{est 201} \\
&+  A^{- \lambda} \lVert (f, f') - (h, h') \rVert_{\tilde{\mathcal{D}}_{\eta}^{\gamma, \rho, A}} ( 1+ \lVert \Theta \rVert_{\mathcal{E}^{\alpha}}^{2}).  \nonumber 
\end{align}
We can make use of \eqref{est 201} and analogous computations to \eqref{est 198} to obtain 
\begin{align}
& \lVert \mathcal{M}^{2} (f, f') - \mathcal{M}^{2} (h, h') \rVert_{\tilde{\mathcal{D}}_{\eta}^{\gamma, \rho, A}} \nonumber \\
\lesssim& A^{- \lambda} \lVert (f, f') - (h, h') \rVert_{\tilde{\mathcal{D}}_{\eta}^{\gamma, \rho, A}} \Bigg[ 1+\lVert \Theta \rVert_{\mathcal{E}^{\alpha}}^{4} \Bigg] \ll \lVert (f, f') - (h, h') \rVert_{\tilde{\mathcal{D}}_{\eta}^{\gamma, \rho, A}}   \label{est 203}
\end{align}
for 
\begin{equation}\label{large A}
A \gg [1 + \lVert \Theta \rVert_{\mathcal{E}^{\alpha}}^{4} ]^{\frac{1}{ \lambda}} 
\end{equation} 
and therefore the mapping $\mathcal{M}^{2}: \hspace{1mm} \tilde{\mathcal{D}}_{\eta}^{\gamma, \rho, A} \mapsto \tilde{\mathcal{D}}_{\eta}^{\gamma, \rho, A}$ is a contraction. Consequently, the fixed point theorem gives us unique $(f, f') \in \tilde{\mathcal{D}}_{\eta}^{\gamma, \rho, A}$ such that $\mathcal{M}(f,f') = (f,f')$. Finally, making use of $M(f,f') = f$ and $f = f'$ and computations that led to \eqref{est 198} and \eqref{est 199} lead to   
\begin{equation}\label{est 204}
\lVert f_{a} \rVert_{H^{\gamma}} + a^{-\rho} \lVert f_{a}^{\sharp} \rVert_{H^{2\gamma}} \leq \frac{1}{2} [\lVert f_{a} \rVert_{H^{\gamma}} + a^{-\rho} \lVert f_{a}^{\sharp} \rVert_{H^{2\gamma}}] + C \left( a^{\frac{\gamma + \delta}{2} -1} + a^{-\rho + \gamma + \frac{\delta}{2} -1} \right) \lVert g \rVert_{H^{-\delta}}
\end{equation} 
for all $A \geq 1$ sufficiently large; subtracting$\frac{1}{2} [\lVert f_{a} \rVert_{H^{\gamma}} + a^{-\rho} \lVert f_{a}^{\sharp} \rVert_{H^{2\gamma}}]$ from both sides leads to \eqref{est 202a} because $\mathcal{G}_{a} g = f_{a}$ by \eqref{Define mathcal G}. Now \eqref{est 202b} follows immediately as  
\begin{align*}
\lVert \mathcal{G}_{a} g \rVert_{\mathcal{D}_{\gamma}^{\eta}} \overset{\eqref{Define mathcal G} \eqref{D eta gamma norm}}{\lesssim} a^{\rho} [ \lVert f_{a} \rVert_{H^{\gamma}} + a^{-\rho} \lVert f_{a}^{\sharp} \rVert_{H^{2\gamma}} ] \lesssim \left(a^{\rho + \frac{\gamma + \delta}{2} - 1} + a^{-1 + \gamma + \frac{\delta}{2}} \right) \lVert g \rVert_{H^{-\delta}}. 
\end{align*}
\end{proof} 

\begin{proposition}\label{Proposition 6.4}
\rm{(Cf. \cite[Lemma 4.15]{AC15})} Let $\alpha \in ( -\frac{4}{3}, -1)$,  $\gamma \in (\frac{2}{3}, \alpha + 2)$ and define $\lambda$ by \eqref{Define lambda}. Then there exists a constant $C > 0$ such that for all $\Theta = (\eta, \Theta_{2})$, $\tilde{\Theta} = (\tilde{\eta}, \tilde{\Theta}_{2}) \in \mathcal{K}^{\alpha}$, and $a \geq C[1+\lVert \Theta \rVert_{\mathcal{E}^{\alpha}}^{4}]^{\frac{1}{\lambda}}$ from \eqref{large A}, we have the following bounds:
\begin{align}
\lVert \mathcal{G}_{a}(\Theta) g - \mathcal{G}_{a} ( \tilde{\Theta}) g \rVert_{H^{\gamma}} \leq& \lVert (\mathcal{G}_{a}(\Theta) g - \mathcal{G}_{a} (\tilde{\Theta}) g,\mathcal{G}_{a}(\Theta) g - \mathcal{G}_{a} (\tilde{\Theta}) g) \rVert_{\tilde{\mathcal{D}}_{\eta}^{\gamma, \rho, A}}  \nonumber\\
\lesssim& \lVert g \rVert_{L^{2}} \lVert \Theta - \tilde{\Theta} \rVert_{\mathcal{E}^{\alpha}} ( 1+ \lVert \Theta \rVert_{\mathcal{E}^{\alpha}} + \lVert \tilde{\Theta} \rVert_{\mathcal{E}^{\alpha}}),  \label{est 267}
\end{align} 
where $\mathcal{G}_{a}(\Theta), \mathcal{G}_{a}(\tilde{\Theta}): \hspace{1mm} L^{2} \mapsto \mathcal{D}_{\eta}^{\gamma}$ are the resolvent operators associated to the rough distributions $\Theta, \tilde{\Theta} \in \mathcal{K}^{\alpha}$ constructed in Proposition \ref{Proposition 6.3}. 
\end{proposition} 

\begin{proof}[Proof of Proposition \ref{Proposition 6.4}]
We take $a \geq A( \lVert \Theta \rVert_{\mathcal{E}^{\alpha}}) + A ( \lVert \tilde{\Theta} \rVert_{\mathcal{E}^{\alpha}})$ according to \eqref{large A} so that 
\begin{equation}\label{est 220}
f_{a} \triangleq \mathcal{G}_{a} (\Theta) g \hspace{1mm} \text{ and } \hspace{1mm} \tilde{f}_{a} \triangleq \mathcal{G}_{a} (\tilde{\Theta}) g 
\end{equation} 
are well-defined by Proposition \ref{Proposition 6.3}. We can verify from \eqref{est 188}, \eqref{Define H}, and \eqref{Define tilde sigma a} that 
\begin{subequations}
\begin{align}
f_{a} - \tilde{f}_{a}=& \tilde{\sigma}_{a}(D) [ f_{a} \prec \eta - \tilde{f}_{a} \prec \tilde{\eta} + f_{a} \succ \eta - \tilde{f}_{a} \succ \tilde{\eta} + f_{a} \Theta_{2} - \tilde{f}_{a} \tilde{\Theta}_{2} + f_{a}^{\sharp} \circ \eta - \tilde{f}_{a}^{\sharp} \circ \tilde{\eta} ]  \nonumber \\
& \hspace{20mm} + \tilde{\sigma}_{a}(D) [ \mathcal{R} (f_{a}, \sigma(D) \eta, \eta) - \mathcal{R} (\tilde{f}_{a}, \sigma(D) \tilde{\eta}, \tilde{\eta} ) ],  \label{est 205a} \\
f_{a}^{\sharp} - \tilde{f}_{a}^{\sharp} =&  \tilde{\sigma}_{a}(D) [ f_{a} \succ \eta - \tilde{f}_{a} \succ \tilde{\eta}  + f_{a} \Theta_{2} - \tilde{f}_{a} \tilde{\Theta}_{2} + f_{a}^{\sharp} \circ \eta - \tilde{f}_{a}^{\sharp} \circ \eta]  \nonumber \\
&+ \tilde{\sigma}_{a}(D) [ \mathcal{R} (f_{a}, \sigma(D) \eta, \eta) - \mathcal{R} (\tilde{f}_{a}, \sigma(D) \tilde{\eta}, \tilde{\eta} ) ]  \nonumber \\
&+\mathcal{C}_{a} (f_{a}, \eta) + f_{a} \prec (\tilde{\sigma}_{a} - \sigma) (D) \eta  - \mathcal{C}_{a} (\tilde{f}_{a}, \tilde{\eta})- \tilde{f}_{a} \prec (\tilde{\sigma}_{a} - \sigma) (D) \tilde{\eta}, \label{est 205b}  
\end{align}
\end{subequations} 
where $\mathcal{C}_{a}(f_{a}, \eta)$ was defined in \eqref{Define mathcal Ca}. With the same $\rho$ from \eqref{Define rho}, because $\gamma < \alpha + 2$ by hypothesis, we can find the same $\epsilon_{2} \in (0, \frac{1}{3})$ in \eqref{Define epsilon two} and estimate similarly to \eqref{est 196a} with $\vartheta = \frac{\gamma + \epsilon_{2} - \alpha}{2} \in [0,1]$ in \eqref{est 195}, 
\begin{subequations}\label{est 206}
\begin{align}
&a^{- \rho} \lVert \mathcal{C}_{a} (f_{a} - \tilde{f}_{a},\eta) \rVert_{H^{2\gamma}}  \lesssim  a^{-( \rho  + \frac{2+ \alpha - \gamma - \epsilon_{2}}{2} )} \lVert f_{a} - \tilde{f}_{a} \rVert_{H^{\gamma}} \lVert \eta \rVert_{\mathcal{C}^{\alpha}},  \label{est 206a}\\
&a^{-\rho} \lVert \mathcal{C}_{a} (\tilde{f}_{a}, \eta - \tilde{\eta}) \rVert_{H^{2\gamma}} \lesssim a^{- (\rho +\frac{2+ \alpha - \gamma - \epsilon_{2}}{2})} \lVert \tilde{f}_{a} \rVert_{H^{\gamma}} \lVert \eta - \tilde{\eta} \rVert_{\mathcal{C}^{\alpha}}. \label{est 206b} 
\end{align}
\end{subequations} 
Additionally, as $\rho > \gamma - \frac{\alpha+2}{2}$ from \eqref{Define rho}, we can find 
\begin{equation}\label{Define epsilon three}
\epsilon_{3} \in \left(0, \rho - \gamma + \frac{\alpha + 2}{2}\right)
\end{equation} 
and estimate by \eqref{Bony 3} and \eqref{Define tilde sigma a} with $\vartheta = \gamma - \frac{\alpha + 2}{2} + \epsilon_{3} \in [0,1]$, 
\begin{subequations}\label{est 207} 
\begin{align}
&a^{-\rho}\lVert ( f_{a} - \tilde{f}_{a}) \prec (\tilde{\sigma}_{a} - \sigma) (D) \eta \rVert_{H^{2\gamma}} \lesssim  a^{-(\rho - \gamma + \frac{\alpha + 2}{2} - \epsilon_{3})} \lVert f_{a} - \tilde{f}_{a} \rVert_{H^{\gamma}} \lVert \eta \rVert_{\mathcal{C}^{\alpha}}, \label{est 207a}  \\
& a^{-\rho} \lVert \tilde{f}_{a} \prec ( \tilde{\sigma}_{a} - \sigma) (D) (\eta - \tilde{\eta} ) \rVert_{H^{2\gamma}} \lesssim a^{-(\rho - \gamma + \frac{\alpha+2}{2} - \epsilon_{3})} \lVert \tilde{f}_{a} \rVert_{H^{\gamma}} \lVert \eta - \tilde{\eta} \rVert_{\mathcal{C}^{\alpha}}.  \label{est 207b}
\end{align} 
\end{subequations} 
Applying \eqref{est 206} and \eqref{est 207}, and making use of $\rho - \gamma + \frac{\alpha +2}{2} - \epsilon_{3} > 0$ due to \eqref{Define epsilon three} lead to 
\begin{align}
a^{-\rho} \lVert f_{a}^{\sharp} - \tilde{f}_{a}^{\sharp} \rVert_{H^{2\gamma}} \lesssim& (a^{\rho + \frac{\gamma}{2} - 1} + a^{-1+ \gamma}) \lVert g \rVert_{L^{2}} \lVert \Theta - \tilde{\Theta} \rVert_{\mathcal{E}^{\alpha}}  \nonumber \\
&+ a^{-(\rho - \gamma + \frac{\alpha + 2}{2} - \epsilon_{3})} \lVert f_{a} - \tilde{f}_{a} \rVert_{H^{\gamma}} (1+ \lVert \Theta \rVert_{\mathcal{E}^{\alpha}})^{2}. \label{est 222}
\end{align} 
Next, with the same $\epsilon_{2}$ from \eqref{Define epsilon two}, due to \eqref{Schauder} and \eqref{Bony 3}, we estimate 
\begin{equation}\label{est 223}
\lVert \tilde{\sigma}_{a}(D) [ (f_{a} - \tilde{f}_{a}) \prec \eta ] \rVert_{H^{\gamma}} \lesssim a^{\frac{ \gamma + \epsilon_{2} - \alpha}{2} -1} \lVert (f_{a} - \tilde{f}_{a}) \prec \eta \rVert_{H^{\alpha - \epsilon_{2}}} \lesssim a^{- ( \frac{ \alpha + 2 - \gamma - \epsilon_{2}}{2})} \lVert f_{a} - \tilde{f}_{a} \rVert_{H^{\gamma}} \lVert \eta \rVert_{\mathcal{C}^{\alpha}}, 
\end{equation} 
which leads to 
\begin{align}
\lVert f_{a} - \tilde{f}_{a}& \rVert_{H^{\gamma}}   \lesssim a^{-( \frac{ \alpha + 2 - \gamma - \epsilon_{2}}{2})} \left(a^{\frac{\gamma}{2} -1} + a^{-\rho + \gamma - 1} \right) \lVert g \rVert_{L^{2}} \lVert \Theta - \tilde{\Theta} \rVert_{\mathcal{E}^{\alpha}} (1+ \lVert \Theta \rVert_{\mathcal{E}^{\alpha}} + \lVert \tilde{\Theta} \rVert_{\mathcal{E}^{\alpha}})  \nonumber \\
&+ a^{-( \frac{\gamma + \alpha}{2} + 1)} \lVert f_{a}^{\sharp} - \tilde{f}_{a}^{\sharp} \rVert_{H^{2\gamma}} \lVert \eta \rVert_{\mathcal{C}^{\alpha}} + a^{-( \frac{\gamma + \alpha}{2} + 1)}  \left(a^{\rho + \frac{\gamma}{2} -1} + a^{\gamma -1} \right)\lVert g \rVert_{L^{2}} \lVert \Theta - \tilde{\Theta} \rVert_{\mathcal{E}^{\alpha}}. \label{est 234} 
\end{align}
Summing \eqref{est 234} and \eqref{est 222}, making use of $\rho - \gamma + \frac{\alpha+2}{2} - \epsilon_{3} > 0$ from \eqref{Define epsilon three}, we obtain
\begin{equation}\label{est 235}
 \lVert f_{a} - \tilde{f}_{a} \rVert_{H^{\gamma}} + a^{-\rho} \lVert f_{a}^{\sharp} - \tilde{f}_{a}^{\sharp} \rVert_{H^{2\gamma}} \lesssim (a^{\rho + \frac{\gamma}{2} -1} + a^{\gamma -1}) \lVert g \rVert_{L^{2}} \lVert \Theta- \tilde{\Theta} \rVert_{\mathcal{E}^{\alpha}}( 1+ \lVert \Theta \rVert_{\mathcal{E}^{\alpha}} +\lVert \tilde{\Theta} \rVert_{\mathcal{E}^{\alpha}}), 
\end{equation} 
which allows us to conclude \eqref{est 267}. 
\end{proof} 

\begin{define}\label{Definition 6.4}
\rm{(Cf. \cite[Definition 4.17]{AC15})} Let $\alpha \in (-\frac{4}{3}, -1), \gamma \in (-\frac{\alpha}{2}, \alpha + 2)$, and $\Theta = (\eta, \Theta_{2}) \in \mathcal{K}^{\alpha}$. We define a bilinear operator $B: \hspace{1mm} H^{\gamma} \times \mathcal{K}^{\alpha} \mapsto H^{2\gamma}$ by 
\begin{equation}\label{Define B}
B(f,\Theta) \triangleq \sigma(D) [ - 2 \nabla f \prec \nabla \sigma(D) \eta - (1+ \Delta) f \prec \sigma(D) \eta + f \succ \eta + f \Theta_{2} ]. 
\end{equation} 
Then we define 
\begin{equation}\label{Define D Theta gamma}
\mathcal{D}_{\Theta}^{\gamma} \triangleq \{ f \in H^{\gamma}: \hspace{1mm} f^{\flat} \triangleq f - f \prec \sigma(D) \eta -B(f, \Theta) \in H^{2} \} 
\end{equation} 
with an inner product 
\begin{equation}\label{Norm D Theta gamma}
\langle f, g \rangle_{\mathcal{D}_{\Theta}^{\gamma}} \triangleq \langle f, g \rangle_{H^{\gamma}} + \langle f^{\flat}, g^{\flat} \rangle_{H^{2}}. 
\end{equation} 
\end{define} 

\begin{remark}
We observe that $\mathcal{D}_{\Theta}^{\gamma} \subset \mathcal{D}_{\eta}^{\gamma}$ and that any $f \in \mathcal{D}_{\Theta}^{\gamma}$ has the regularity of $H^{(\alpha + 2)-}$ which motivates us to define for any $\alpha \in \left( - \frac{4}{3}, -1 \right)$, $\Theta = ( \eta, \Theta_{2} ) \in \mathcal{K}^{\alpha}$, and 
\begin{equation}\label{Define new kappa}
\kappa \in \left( 0, (4+ 3 \alpha) \wedge \frac{2}{3} \right) 
\end{equation} 
sufficiently small fixed, 
\begin{equation}\label{Define D Theta}
\mathcal{D}_{\Theta} \triangleq \{ f \in H^{\alpha + 2 - \kappa}: \hspace{1mm} f^{\flat} \triangleq f - f \prec \sigma(D) \eta - B(f, \Theta) \in H^{2} \}. 
\end{equation} 
\end{remark} 

\begin{proposition}\label{Proposition 6.5}
\rm{(Cf. \cite[Lemma 4.19]{AC15})} Let $\alpha \in (-\frac{4}{3}, -1)$ and $\gamma \in (\frac{2}{3}, \alpha + 2)$. Then, for any $a \geq 2$, we define 
\begin{equation}\label{Define Ba}
B_{a} (f, \Theta) \triangleq \sigma_{a}(D) [ -2 \nabla f \prec \nabla \sigma(D) \eta - (1+ \Delta) f \prec \sigma(D) \eta + f \succ \eta + f \Theta_{2} ], 
\end{equation} 
where $\sigma_{a}$ was defined in \eqref{Define sigma a}. Then 
\begin{equation}\label{est 240} 
\lVert B_{a} (f, \Theta) \rVert_{H^{2\gamma}} \lesssim a^{- ( \frac{2- \gamma + \alpha}{2})} \lVert f \rVert_{H^{\gamma}} \lVert \Theta \rVert_{\mathcal{E}^{\alpha}}, \hspace{5mm} \lVert (B- B_{a}) (f, \Theta) \rVert_{H^{2 \gamma + 2}} \lesssim a^{\frac{\gamma - \alpha}{2}} \lVert f \rVert_{H^{\gamma}} \lVert \Theta \rVert_{\mathcal{E}^{\alpha}}. 
\end{equation} 
\end{proposition}

\begin{proof}[Proof of Proposition \ref{Proposition 6.5}]
Applying the two inequalities in \eqref{est 172} with $\vartheta = \frac{\gamma - \alpha}{2}$ gives us the desired result \eqref{est 240}. 
\end{proof}  

\begin{proposition}\label{Proposition 6.6}
\rm{(Cf. \cite[Proposition 4.20]{AC15})} Let $\alpha \in (-\frac{4}{3}, -1)$, $\Theta = ( \eta, \Theta_{2} ) \in \mathcal{K}^{\alpha}$, and $\kappa$ from \eqref{Define new kappa} sufficiently small fixed. If $f \in \mathcal{D}_{\Theta}$ where $\mathcal{D}_{\Theta}$ is defined in \eqref{Define D Theta}, then $\mathcal{H} f \in L^{2}$. 
\end{proposition}

\begin{proof}[Proof of Proposition \ref{Proposition 6.6}]
The claim follows by writing 
\begin{equation}\label{est 251}
\mathcal{H} f = \Delta f^{\flat} +  B(f, \Theta) - \mathcal{R} (f, \sigma(D) \eta, \eta) - \eta \circ \left(B(f, \Theta) + f^{\flat} \right)
\end{equation} 
and making use of \eqref{Define D Theta}, \eqref{Define B}, \eqref{Estimate on R}, and \eqref{Bony 5} to verify that each term is in $L^{2}$. The hypothesis that $\kappa < 4 + 3 \alpha$ due to \eqref{Define new kappa} is used upon verifying that $\mathcal{R} (f, \sigma(D) \eta, \eta) \in L^{2}$. 
\end{proof}

\begin{proposition}\label{Proposition 6.7}
\rm{(Cf. \cite[Proposition 4.21]{AC15})} Let $\alpha \in (-\frac{4}{3}, -1), \gamma \in (-\frac{\alpha}{2}, \alpha + 2)$, and $\Theta = (\eta, \Theta_{2}), \tilde{\Theta} = (\tilde{\eta}, \tilde{\Theta}_{2}) \in \mathcal{K}^{\alpha}$. Then, for all $f \in \mathcal{D}_{\Theta}$, there exists $g \in \mathcal{D}_{\tilde{\Theta}}$ such that 
\begin{equation}\label{est 250} 
\lVert f-g \rVert_{H^{\gamma}} + \lVert f^{\flat} - g^{\flat} \rVert_{H^{2}} \lesssim ( \lVert f \rVert_{H^{\gamma}} + \lVert g \rVert_{H^{\gamma}} ) (1+ \lVert \tilde{\Theta} \rVert_{\mathcal{E}^{\alpha}}) \lVert \Theta - \tilde{\Theta} \rVert_{\mathcal{E}^{\alpha}}, 
\end{equation} 
where 
\begin{equation}\label{Define g flat}
g^{\flat} \triangleq g - g \prec \sigma(D) \tilde{\eta} - B(g, \tilde{\Theta}). 
\end{equation} 
In particular, if $(\eta_{n}, c_{n})_{n\in\mathbb{N}} \subset C^{\infty} \times \mathbb{R}$ is a family such that  $(\eta_{n}, -\eta_{n} \circ \sigma(D) \eta_{n} - c_{n}) \to \Theta$ in $\mathcal{E}^{\alpha}$ as $n \nearrow 0$, then there exists a family $\{f_{n} \}_{n \in\mathbb{N}} \subset H^{2}$ such that  
\begin{equation}\label{est 254} 
\lim_{n \nearrow \infty} \lVert f_{n} - f \rVert_{H^{\gamma}} + \lVert f_{n}^{\flat} - f^{\flat} \rVert_{H^{2}} = 0, 
\end{equation} 
where $f_{n}^{\flat} \triangleq f_{n} - f_{n}\prec \sigma(D) \eta_{n} - B(f_{n}, \Theta^{n})$ and $\Theta^{n} \triangleq (\eta_{n}, -\eta_{n} \circ \sigma(D) \eta_{n} - c_{n})$. 
\end{proposition} 

\begin{proof}[Proof of Proposition \ref{Proposition 6.7}]
Let $f \in \mathcal{D}_{\Theta}$ and define $\Gamma: \hspace{1mm} H^{\gamma} \mapsto H^{\gamma}$ by 
\begin{equation}\label{est 175}
\Gamma(g) \triangleq g \prec \sigma_{a}(D) \tilde{\eta} + B_{a} (g, \tilde{\Theta}) + f - [f \prec \sigma_{a}(D) \eta+ B_{a} (f, \Theta) ], 
\end{equation} 
where $B_{a}(g, \tilde{\Theta})$ and $B_{a} (f, \Theta)$ are defined according to \eqref{Define Ba}. Then, for all $g_{1}, g_{2} \in H^{\gamma}$, because $\gamma < \alpha + 2$ by hypothesis, we can find $\epsilon_{2} > 0$ that satisfies \eqref{Define epsilon two} and rely on \eqref{Bony 1} and \eqref{Define tilde sigma a} to deduce 
\begin{equation}\label{est 177}
\lVert (g_{1} - g_{2}) \prec \sigma_{a}(D) \tilde{\eta} \rVert_{H^{\gamma}} \lesssim a^{\frac{ \gamma + \epsilon_{2} - \alpha}{2} -1} \lVert \tilde{\Theta} \rVert_{\mathcal{E}^{\alpha}} \lVert g_{1} - g_{2} \rVert_{H^{\gamma}},
\end{equation} 
which leads to 
\begin{equation}\label{est 179}
 \lVert \Gamma (g_{1}) - \Gamma(g_{2}) \rVert_{H^{\gamma}} \lesssim  a^{\frac{\gamma + \epsilon_{2} - \alpha}{2} -1} \lVert g_{1} - g_{2}  \rVert_{H^{\gamma}} [ \lVert \Theta \rVert_{\mathcal{E}^{\alpha}} + \lVert \tilde{\Theta} \rVert_{\mathcal{E}^{\alpha}} ].
\end{equation} 
Considering \eqref{Define epsilon two}, we deduce that $\Gamma$ is a contraction for $a \gg 1$ so that there exists a unique $g$ such that 
\begin{subequations}\label{est 180}
\begin{align}
g \overset{\eqref{est 175}}{=}& g \prec \sigma_{a}(D) \tilde{\eta} + B_{a} (g, \tilde{\Theta}) + f - [f \prec \sigma_{a}(D) \eta+ B_{a} (f, \Theta) ] \label{est 180a}\\
\overset{\eqref{Define D Theta gamma}}{=}&g \prec \sigma_{a}(D) \tilde{\eta} + B_{a} (g, \tilde{\Theta}) + f^{\flat} + f \prec ( \sigma - \sigma_{a})(D) \eta + B(f, \Theta) - B_{a}(f, \Theta). \label{est 180b}
\end{align}
\end{subequations} 
Next, for any $\kappa > 0$, we can estimate 
\begin{equation}
\lVert g \prec \sigma_{a}(D) \tilde{\eta} \rVert_{H^{2+ \alpha - \kappa}} \overset{\eqref{Bony 1}}{\lesssim} \lVert g \rVert_{L^{2}} \lVert \sigma_{a}(D) \tilde{\eta} \rVert_{\mathcal{C}^{2+ \alpha}} \lesssim \lVert g \rVert_{H^{\gamma}} \lVert \tilde{\eta} \rVert_{\mathcal{C}^{\alpha}} \lesssim 1, 
\end{equation} 
and deduce $g \in H^{2 + \alpha - \kappa}$. Additionally, we can show that $g^{\flat} \in H^{2}$ so that $g \in \mathcal{D}_{\tilde{\Theta}}$ by \eqref{Define D Theta}. Next, because $\gamma < \alpha + 2$ by hypothesis, we can find $\epsilon_{2} > 0$ that satisfies \eqref{Define epsilon two} and estimate 
\begin{equation}\label{est 247}
 \lVert (g-f) \prec \sigma_{a}(D) \tilde{\eta} \rVert_{H^{\gamma}} \overset{\eqref{Bony 1}}{\lesssim}\lVert g-f \rVert_{L^{2}} \lVert \sigma_{a}(D) \tilde{\eta} \rVert_{\mathcal{C}^{\gamma + \epsilon_{2}}} \overset{\eqref{est 172}}{\lesssim}a^{\frac{ \gamma + \epsilon_{2} - \alpha}{2} -1} \lVert g -f \rVert_{H^{\gamma}}\lVert \tilde{\Theta} \rVert_{\mathcal{E}^{\alpha}}
\end{equation} 
which leads to 
\begin{align*}
\lVert f-g \rVert_{H^{\gamma}} \lesssim& [a^{\frac{\gamma + \epsilon_{2} - \alpha}{2} -1} \lVert \tilde{\Theta} \rVert_{\mathcal{E}^{\alpha}} + a^{-\frac{\alpha}{2} -1} \lVert \Theta \rVert_{\mathcal{E}^{\alpha}} ] \lVert f-g \rVert_{H^{\gamma}}  \\
&+ [a^{\frac{ \gamma + \epsilon_{2} - \alpha}{2} -1} \lVert f \rVert_{H^{\gamma}} + a^{-\frac{\alpha}{2} -1} \lVert g \rVert_{H^{\gamma}} ] \lVert \Theta - \tilde{\Theta} \rVert_{\mathcal{E}^{\alpha}}  
\end{align*}
so that making use of \eqref{Define epsilon two} leads to, for all sufficiently large $a \gg 1$, 
\begin{equation}\label{est 249}
 \lVert f-g \rVert_{H^{\gamma}} \lesssim [a^{\frac{ \gamma + \epsilon_{2} - \alpha}{2} -1} \lVert f \rVert_{H^{\gamma}} + a^{-\frac{\alpha}{2} -1} \lVert g \rVert_{H^{\gamma}} ] \lVert \Theta - \tilde{\Theta} \rVert_{\mathcal{E}^{\alpha}}. 
\end{equation}  
We can estimate $f^{\flat} - g^{\flat}$ similarly, make use of \eqref{est 249}, and conclude \eqref{est 250}.  
\end{proof} 

\begin{proposition}\label{Proposition 6.8}
\rm{(Cf. \cite[Proposition 4.22]{AC15})} Let $\alpha \in (-\frac{4}{3}, -1), \gamma \in (\frac{2}{3},\alpha +2)$, and $\Theta = (\eta, \Theta_{2}), \tilde{\Theta} = (\tilde{\eta}, \tilde{\Theta}_{2}) \in \mathcal{K}^{\alpha}$. Define $\mathcal{H}^{\Theta} \triangleq  \Delta - \eta$ and $\mathcal{H}^{\tilde{\Theta}} \triangleq \Delta - \tilde{\eta}$ with respective domains denoted by $\mathcal{D}_{\Theta}$ and $\mathcal{D}_{\tilde{\Theta}}$ defined according to \eqref{Define D Theta}. Then 
\begin{align}
\lVert \mathcal{H}^{\Theta} f - \mathcal{H}^{\tilde{\Theta}} g \rVert_{L^{2}} \lesssim& ( \lVert f-g \rVert_{H^{\gamma}} + \lVert f^{\flat} - g^{\flat} \rVert_{H^{2}} + \lVert \Theta - \tilde{\Theta} \rVert_{\mathcal{E}^{\alpha}})  \nonumber \\
& \times \left( 1+ \lVert \Theta \rVert_{\mathcal{E}^{\alpha}} + \lVert \tilde{\Theta} \rVert_{\mathcal{E}^{\alpha}} \right) \left( 1+ \lVert \Theta \rVert_{\mathcal{E}^{\alpha}} + \lVert f \rVert_{H^{\gamma}} + \lVert g \rVert_{H^{\gamma}} + \lVert f^{\flat} \rVert_{H^{2}} \right).  \label{est 252}
\end{align}
Moreover, the operator $\mathcal{H}: \hspace{1mm} \mathcal{D}_{\Theta} \mapsto L^{2}$ is symmetric in $L^{2}$ so that 
\begin{equation}
\langle \mathcal{H}f, g \rangle_{L^{2}} = \langle f, \mathcal{H} g \rangle_{L^{2}} \hspace{3mm} \forall \hspace{1mm} f, g \in \mathcal{D}_{\Theta}. 
\end{equation} 
\end{proposition} 

\begin{proof}[Proof of Proposition \ref{Proposition 6.8}]
We can deduce \eqref{est 252} starting from \eqref{est 251} and using \eqref{Define B}, \eqref{est 172} and \eqref{Bony in Sobolev}. Next, we let 
\begin{equation*}
\{ \Theta^{n} \}_{n \in \mathbb{N}} = \{ (\eta_{n}, -\eta_{n} \circ \sigma(D) \eta_{n} - c_{n}) \}_{n\in\mathbb{N}} \subset C^{\infty} 
\end{equation*}
satisfy $\Theta^{n} \to \Theta$ in $\mathcal{E}^{\alpha}$ as $n \nearrow \infty$. Then, by Proposition \ref{Proposition 6.7}, there exists $\{f_{n}\}_{n \in\mathbb{N}} \subset H^{2}$ such that \eqref{est 254} is satisfied. The smoothness of $\eta_{n}$ allows us to define $\mathcal{H}_{n} \triangleq  \Delta - \eta_{n}$ on $H^{2}$. Moreover, applying the assumption of $\Theta^{n} \to \Theta$ in $\mathcal{E}^{\alpha}$ as $n \nearrow \infty$ and \eqref{est 254} to \eqref{est 252} shows that $\mathcal{H}_{n} f_{n} \to \mathcal{H} f$ in $L^{2}$. Additionally, if $\{g_{n}\}_{n \in \mathbb{N}} \subset H^{2}$ satisfies $\mathcal{H}_{n}g_{n} \to \mathcal{H} g$ in $L^{2}$ as $n \nearrow \infty$ and $\lim_{n \nearrow \infty} \lVert g_{n} - g \rVert_{H^{\gamma}} + \lVert g_{n}^{\flat} - g^{\flat} \rVert_{H^{2}} = 0$ similarly to \eqref{est 254}, then $\langle \mathcal{H}_{n} f_{n}, g_{n} \rangle_{L^{2}}  = \langle f_{n}, \mathcal{H}_{n} g_{n} \rangle_{L^{2}}$; in turn, this implies 
\begin{align*}
\lvert \langle \mathcal{H}f, g \rangle_{L^{2}} - \langle f, \mathcal{H} g \rangle_{L^{2}}& \rvert \leq \lVert \mathcal{H} f - \mathcal{H} f_{n} \rVert_{L^{2}} \lVert g \rVert_{L^{2}} + \lVert \mathcal{H} f_{n} \rVert_{L^{2}} \lVert g - g_{n} \rVert_{L^{2}} \\
&+ \lVert f_{n} - f \rVert_{L^{2}} \lVert \mathcal{H}_{n} g_{n} \rVert_{L^{2}} + \lVert f \rVert_{L^{2}} \lVert \mathcal{H}_{n}g_{n} - \mathcal{H} g \rVert_{L^{2}}  \to 0 \text{ as } n \nearrow \infty. 
\end{align*}
\end{proof} 

\begin{proposition}\label{Proposition 6.9}
\rm{(Cf. \cite[Proposition 4.23]{AC15})} Let $\alpha \in (-\frac{4}{3}, -1), \gamma \in (\frac{2}{3}, \alpha + 2),$ and $A = A( \lVert \Theta \rVert_{\mathcal{E}^{\alpha}})$ from Proposition \ref{Proposition 6.3}. Then, for all $a \geq A$, $-\mathcal{H} + a: \hspace{1mm} \mathcal{D}_{\Theta} \mapsto L^{2}$ is invertible with inverse $\mathcal{G}_{a}: \hspace{1mm} L^{2} \mapsto \mathcal{D}_{\Theta}$. Additionally, $\mathcal{G}_{a}: \hspace{1mm} L^{2} \mapsto L^{2}$ is bounded, self-adjoint, and compact. 
\end{proposition} 

\begin{proof}[Proof of Proposition \ref{Proposition 6.9}]
Let $g \in L^{2}$. By Proposition \ref{Proposition 6.3}, there exists a unique $f_{a} \in \mathcal{D}_{\eta}^{\gamma}$ such that $f_{a} = \mathcal{G}_{a} g$; i.e., $(-\mathcal{H} + a) f_{a} = g$, and consequently due to \eqref{Define H} and \eqref{est 165}, 
\begin{align}
 (1 - \Delta) f_{a}^{\sharp} =& f_{a}^{\sharp} + g - af_{a} + 2 \nabla f_{a} \prec \nabla \sigma(D) \eta + (1+ \Delta) f_{a} \prec \sigma(D) \eta - f_{a} \succ \eta   \nonumber \\
& \hspace{20mm} - \left( f_{a} \Theta_{2} + \mathcal{R} (f_{a}, \sigma(D) \eta, \eta ) + f_{a}^{\sharp} \circ \eta \right).  \label{est 257}
\end{align}
Resultantly, 
\begin{align}
f_{a}^{\flat} \triangleq f_{a}^{\sharp} - B(f_{a}, \Theta) \overset{\eqref{est 257} \eqref{Define B}}{=}  \sigma(D) [-f_{a}^{\sharp} - g + a f_{a} + \mathcal{R} (f_{a}, \sigma(D) \eta, \eta) + f_{a}^{\sharp} \circ \eta]   \label{est 258}
\end{align}
where we can show that $f_{a}^{\flat} \in H^{2}$ 
\begin{equation}\label{est 259}
\lVert f_{a}^{\flat} \rVert_{H^{2}} \overset{\eqref{Estimate on R} \eqref{Bony 5}}{\lesssim}  a \lVert f_{a} \rVert_{\mathcal{D}_{\eta}^{\gamma}} (1+ \lVert \eta \rVert_{\mathcal{C}^{\alpha}}^{2}) + \lVert g \rVert_{L^{2}}  \lesssim 1, 
\end{equation} 
and consequently $f_{a} \in H^{\alpha + 2 - \kappa}$ so that $f_{a} \in \mathcal{D}_{\Theta}$ by \eqref{Define D Theta}. Moreover, 
\begin{equation}\label{est 260} 
\lVert \mathcal{G}_{a} g \rVert_{\mathcal{D}_{\Theta}^{\gamma}} 
\overset{\eqref{Norm D Theta gamma}}{\leq}  \lVert \mathcal{G}_{a} g \rVert_{H^{\gamma}} + a^{-\rho} \lVert ( \mathcal{G}_{a} g)^{\sharp} \rVert_{H^{2\gamma}} + \lVert f_{a}^{\flat} \rVert_{H^{2}}  \overset{\eqref{est 259}\eqref{est 202}}{\lesssim} a^{\gamma} (1+\lVert \eta \rVert_{\mathcal{C}^{\alpha}}^{2} ) \lVert g\rVert_{L^{2}}.
\end{equation} 
Next, the fact that $\mathcal{G}_{a}: \hspace{1mm} L^{2} \mapsto L^{2}$ is  self-adjoint follows from the symmetry of $\mathcal{H}$. Finally, writing $\mathcal{G}_{a}: \hspace{1mm} L^{2} \mapsto L^{2}$ as a composition of $\mathcal{G}_{a}:\hspace{1mm} L^{2} \mapsto H^{\gamma}$ and an embedding operator $i:  \hspace{1mm} H^{\gamma} \mapsto L^{2}$ shows that $\mathcal{G}_{a}: \hspace{1mm} L^{2} \mapsto L^{2}$ is compact. 
\end{proof}

\appendix
\section{Further preliminaries}\label{Appendix A}
\begin{lemma}\label{Lemma A.1} 
\rm{(\hspace{1sp}\cite[Proposition 3.3]{AC15})} Let $\alpha, n \in \mathbb{R}$ and $\sigma: \hspace{1mm} \mathbb{R}^{d} \setminus \{0\} \mapsto \mathbb{R}$ be an infinitely differentiable function such that $\lvert D^{k} \sigma(x) \rvert \lesssim (1+ \lvert x \rvert)^{-n - k }$ for all $x \in \mathbb{R}^{d}$. For $f \in H^{\alpha}(\mathbb{T}^{d})$, we define $\sigma(D) f$ by \eqref{Define sigma D}. Then $\sigma(D) f \in H^{\alpha + n} (\mathbb{T}^{d})$ and 
\begin{equation}\label{Schauder}
\lVert \sigma(D) f \rVert_{H^{\alpha + n}} \lesssim_{\alpha, n} \lVert f \rVert_{H^{\alpha}}. 
\end{equation} 
\end{lemma} 

\begin{lemma}\label{Lemma A.2} 
\rm{(\hspace{1sp}\cite[Proposition A.2]{AC15} and \cite[Lemma A.8]{GUZ20})} Let $\alpha \in (0,1), \beta \in \mathbb{R}$, and $f \in H^{\alpha} (\mathbb{T}^{d}), g \in \mathcal{C}^{\beta} (\mathbb{T}^{d})$, and $\sigma: \hspace{1mm} \mathbb{R}^{d} \setminus \{0\} \mapsto \mathbb{R}$ be infinitely differentiable function such that $\lvert D^{k} \sigma(x) \rvert \lesssim (1+ \lvert x \rvert)^{-n-k}$ for all $x \in \mathbb{R}^{d}$ and $k \in \mathbb{N}_{0}^{d}$. Define 
\begin{equation}\label{Define mathcal C}
\mathcal{C} ( f,g) \triangleq \sigma(D) (f\prec g) - f \prec \sigma(D) g.  
\end{equation} 
Then 
\begin{align}\label{est 195}
\lVert \mathcal{C} (f,g) \rVert_{H^{\alpha + \beta + n - \delta}} \lesssim \lVert f \rVert_{H^{\alpha}} \lVert g \rVert_{\mathcal{C}^{\beta}} \hspace{3mm} \forall \hspace{1mm} \delta > 0. 
\end{align}
\end{lemma} 
 
\begin{lemma}\label{Lemma A.3} 
\rm{(\hspace{1sp}\cite[Proposition 4.3]{AC15} and \cite[Proposition A.2]{GUZ20})} Let $\alpha \in (0,1), \beta, \gamma \in \mathbb{R}$ such that $\beta + \gamma < 0$ and $\alpha + \beta + \gamma > 0$. Define 
\begin{equation}\label{Define R}
\mathcal{R}(f,g,h) \triangleq (f \prec g) \circ h - f (g \circ h)
\end{equation} 
for smooth functions. Then this trilinear operator can be extended to the product space $H^{\alpha} \times \mathcal{C}^{\beta} \times \mathcal{C}^{\gamma}$ and 
\begin{equation}\label{Estimate on R} 
\lVert \mathcal{R} (f,g.h) \rVert_{H^{\alpha + \beta + \gamma - \delta}} \lesssim \lVert f \rVert_{H^{\alpha}} \lVert g \rVert_{\mathcal{C}^{\beta}} \lVert h \rVert_{\mathcal{C}^{\gamma}} \hspace{3mm} \forall \hspace{1mm} f \in H^{\alpha}, g \in \mathcal{C}^{\beta}, h \in \mathcal{C}^{\gamma}, \text{ and } \delta > 0. 
\end{equation} 
\end{lemma} 

\section{Further details}\label{Appendix B} 

\subsection{Proof of Lemma \ref{Product estimate lemma}}\label{Section B.1}
We sketch the proof of Lemma \ref{Product estimate lemma}. First, we estimate 
\begin{align*} 
\lVert \Delta_{m} (f \prec g)\rVert_{L^{2}} &\overset{\eqref{est 0}}{\lesssim}\sum_{l \leq m-2} \lVert \Delta_{l} f\rVert_{L^{\infty}} \lVert \Delta_{m} g\rVert_{L^{2}}  \\
&\lesssim \lVert f\rVert_{\dot{B}_{2,\infty}^{\sigma_{1}}}\left(\sum_{l\leq m-2}2^{(l-m)(\frac{d}{2} - \sigma_{1})}\right) 2^{m(\frac{d}{2} - \sigma_{1})}\lVert \Delta_{m} g\rVert_{L^{2}} \lesssim \lVert f\rVert_{\dot{B}_{2,\infty}^{\sigma_{1}}}2^{m(\frac{d}{2} - \sigma_{1})}\lVert\Delta_{m} g\rVert_{L^{2}}, 
\end{align*}
where we used Bernstein's inequality and the hypothesis that $\sigma_{1} < \frac{d}{2}$. Multiplying by $2^{m(\sigma_{1} + \sigma_{2} - \frac{d}{2})}$ and taking $l^{2}$-norm in $m$ give us 
\begin{align}\label{est 18}
\left\lVert 2^{m(\sigma_{1} + \sigma_{2} - \frac{d}{2})}\lVert \Delta_{m} (f \prec g)\rVert_{L^{2}} \right\rVert_{l^{2}} \lesssim& \lVert f\rVert_{\dot{B}_{2,\infty}^{\sigma_{1}}} \left\lVert 2^{m\sigma_{2}}\lVert \Delta_{j} g\rVert_{L^{2}} \right\rVert_{l^{2}} \lesssim \lVert f\rVert_{\dot{H}^{\sigma_{1}}} \lVert g\rVert_{\dot{H}^{\sigma_{2}}}.
\end{align}
Similarly, relying on the hypothesis that $\sigma_{2} < \frac{d}{2}$ and Bernstein's inequality leads us to 
\begin{align}\label{est 19}
\left\lVert 2^{m(\sigma_{1} + \sigma_{2} - \frac{d}{2})}\lVert \Delta_{m} (g \prec f)\rVert_{L^{2}} \right\rVert_{l^{2}} \lesssim \lVert g\rVert_{\dot{B}_{2,\infty}^{\sigma_{2}}} \Bigg\lVert 2^{m\sigma_{1}}\lVert \Delta_{m} f\rVert_{L^{2}} \Bigg\rVert_{l^{2}} \lesssim \lVert g\rVert_{\dot{H}^{\sigma_{2}}}\lVert f\rVert_{\dot{H}^{\sigma_{1}}}.
\end{align}
Finally, Bernstein's and H$\ddot{\mathrm{o}}$lder's inequalities lead us to 
\begin{align*}
&\lVert \Delta_{m} (f \circ g)\rVert_{L^{2}} \overset{\eqref{est 0}}{\lesssim}  2^{m(\frac{d}{2})} \sum_{j: \lvert j \rvert \leq 1} \sum_{i \geq m- N_{2}} \lVert \Delta_{i} f\Delta_{i+j} g\rVert_{L^{1}}\\
\lesssim&  \sum_{j: \lvert j \rvert \leq 1} \sum_{i \geq m- N_{2}} 2^{(m-i)(\sigma_{1} + \sigma_{2})} 2^{m(\frac{d}{2}) + (i-m)(\sigma_{1} + \sigma_{2})} \lVert \Delta_{i} f\rVert_{L^{2}} \lVert \Delta_{i+j} g\rVert_{L^{2}}.
\end{align*}
Multiplying by $2^{m(\sigma_{1} + \sigma_{2} - \frac{d}{2})}$, taking $l^{2}$-norm in $m$, and using Young's inequality for convolution give us  
\begin{align}
\left\lVert 2^{m(\sigma_{1} + \sigma_{2} - \frac{d}{2})}\lVert \Delta_{m}(f \circ g)\rVert_{L^{2}} \right\rVert_{l^{2}}  \lesssim \lVert f\rVert_{\dot{H}^{\sigma_{1}}} \lVert g\rVert_{\dot{H}^{\sigma_{2}}}. \label{est 20}
\end{align}
Summing up \eqref{est 18}, \eqref{est 19}, and \eqref{est 20} implies \eqref{Product estimate}. 

\subsection{Proof of Proposition \ref{Proposition 4.11}}\label{Section B.2}
Similarly to \eqref{est 61} and \eqref{est 33}, we see that \eqref{est 32} gives us 
\begin{equation}\label{est 105}
\partial_{t} \lVert w^{\mathcal{L}} (t) \rVert_{\dot{H}^{\epsilon}}^{2}  = \sum_{k=1}^{4} \RomanII_{k}
\end{equation} 
where 
\begin{subequations}
\begin{align}
\RomanII_{1} \triangleq& 2 \langle (-\partial_{x}^{2})^{\epsilon} w^{\mathcal{L}}, \nu \partial_{x}^{2} w^{\mathcal{L}} -  \partial_{x} (w^{\mathcal{L}} \mathcal{L}_{\lambda_{t}}X) \rangle_{L^{2}}(t), \label{Define II1}\\
\RomanII_{2} \triangleq& -2 \langle (-\partial_{x}^{2})^{\epsilon} w^{\mathcal{L}}, \partial_{x} (w^{\mathcal{L}} \mathcal{H}_{\lambda_{t}}X - w^{\mathcal{L}} \prec \mathcal{H}_{\lambda_{t}}X ) \rangle_{L^{2}}(t), \label{Define II2} \\
\RomanII_{3} \triangleq& -2 \langle (-\partial_{x}^{2})^{\epsilon} w^{\mathcal{L}}, \partial_{x} (w^{\mathcal{H}} X - w^{\mathcal{H}} \prec \mathcal{H}_{\lambda_{t}}X) \rangle_{L^{2}} (t),  \label{Define II3}\\
\RomanII_{4} \triangleq& - \langle (-\partial_{x}^{2})^{\epsilon} w^{\mathcal{L}}, \partial_{x} (w^{2} + 2wY - C^{\prec}(w, Q^{\mathcal{H}}) + Y^{2} \rangle_{L^{2}}(t).  \label{Define II4}
\end{align}
\end{subequations} 
For $\RomanII_{1}$, we first compute 
\begin{align}
\lVert w^{\mathcal{L}} \prec \mathcal{L}_{\lambda_{t}}X(t) \rVert_{\dot{H}^{2\epsilon}}  \overset{\eqref{Bony 3}}{\lesssim}  \lVert & w^{\mathcal{L}}(t) \rVert_{\dot{H}^{2\epsilon}}  \lambda_{t}^{\frac{1}{3}} \lVert X(t) \rVert_{\mathcal{C}^{-\kappa}}  \nonumber \\
\overset{\eqref{Define Lt and Nt} \eqref{Define M} \eqref{Higher frequency estimate}}{\lesssim}& \lVert w^{\mathcal{L}} (t) \rVert_{\dot{H}^{2\epsilon}} (1+ M + N_{t}^{\kappa}) N_{t}^{\kappa}  \lesssim \lVert w^{\mathcal{L}}(t) \rVert_{\dot{H}^{2\epsilon}} C(M, N_{t}^{\kappa}),  \label{est 81}
\end{align}
and similarly
\begin{align}
\lVert w^{\mathcal{L}} \succ \mathcal{L}_{\lambda_{t}} X(t) \rVert_{\dot{H}^{2\epsilon}} +& \lVert w^{\mathcal{L}} \circ \mathcal{L}_{\lambda_{t}}X(t) \rVert_{\dot{H}^{2\epsilon}} \overset{\eqref{Bony 2} \eqref{Bony 5}}{\lesssim} \lVert w^{\mathcal{L}} (t) \rVert_{H^{2\epsilon}} \lVert \mathcal{L}_{\lambda_{t}} X(t) \rVert_{\mathcal{C}^{\frac{1}{3} - \kappa}} \nonumber\\
& \hspace{10mm}  \lesssim \lVert w^{\mathcal{L}}(t) \rVert_{\dot{H}^{2\epsilon}} \lambda_{t}^{\frac{1}{3}} \lVert X(t) \rVert_{\mathcal{C}^{-\kappa}}  \lesssim \lVert w^{\mathcal{L}}(t) \rVert_{\dot{H}^{2\epsilon}} C(M, N_{t}^{\kappa}). \label{est 82}
\end{align}
Therefore, we estimate $\RomanII_{1}$ from \eqref{Define II1} by 
\begin{align}
\RomanII_{1} &\overset{\eqref{est 81} \eqref{est 82}}{\leq}- 2 \nu \lVert w^{\mathcal{L}}(t) \rVert_{\dot{H}^{1+ \epsilon}}^{2} +   \lVert w^{\mathcal{L}} (t) \rVert_{\dot{H}^{1}} \lVert w^{\mathcal{L}}(t) \rVert_{\dot{H}^{2\epsilon}} C(M, N_{t}^{\kappa}) \nonumber \\
\overset{\eqref{Define M}}{\leq}&- 2 \nu \lVert w^{\mathcal{L}}(t) \rVert_{\dot{H}^{1+ \epsilon}}^{2} + C(M, N_{t}^{\kappa}) \lVert w^{\mathcal{L}}(t) \rVert_{H^{1+ \epsilon}}^{\frac{1+ 2 \epsilon}{1+ \epsilon}} \leq - \frac{31 \nu}{16} \lVert w^{\mathcal{L}} (t) \rVert_{\dot{H}^{1+ \epsilon}}^{2} +C(M, N_{t}^{\kappa}). \label{est 106}
\end{align}
Next, we first rewrite $\RomanII_{2}$ from \eqref{Define II2} using Bony's paraproducts and estimate as follows:
\begin{align}
\RomanII_{2} =& -2 \langle ( -\partial_{x}^{2})^{\epsilon} w^{\mathcal{L}}, \partial_{x} ( w^{\mathcal{L}} \succ \mathcal{H}_{\lambda_{t}}X + w^{\mathcal{L}} \circ \mathcal{H}_{\lambda_{t}}X) \rangle_{L^{2}} (t)  \nonumber \\
\overset{\eqref{Bony 4} \eqref{Bony 5}}{\lesssim}& \lVert w^{\mathcal{L}}(t) \rVert_{\dot{H}^{1+ \epsilon}}  \lVert w^{\mathcal{L}}(t) \rVert_{\dot{H}^{\epsilon + \kappa}}  \lVert \mathcal{H}_{\lambda_{t}}X(t) \rVert_{\mathcal{C}^{-\kappa}} \leq \frac{\nu}{16} \lVert w^{\mathcal{L}}(t) \rVert_{\dot{H}^{1+ \epsilon}}^{2} + C(M, N_{t}^{\kappa}). \label{est 107} 
\end{align}
Concerning $\RomanII_{3}$ in \eqref{Define II3}, we first rewrite 
\begin{equation}\label{est 263}
\RomanII_{3}  = -2 \langle (-\partial_{x}^{2})^{\epsilon} w^{\mathcal{L}}, \partial_{x} ( w^{\mathcal{H}} \mathcal{L}_{\lambda_{t}}X + w^{\mathcal{H}}  \succ \mathcal{H}_{\lambda_{t}}X+ w^{\mathcal{H}} \circ \mathcal{H}_{\lambda_{t}}X ) \rangle_{L^{2}} (t), 
\end{equation} 
make use of \eqref{est 36}-\eqref{est 35} and estimate 
\begin{align}
\RomanII_{3} \lesssim& \lVert w^{\mathcal{L}} (t) \rVert_{\dot{H}^{\frac{2}{3} + 2 \kappa + 2 \epsilon}} [ \lVert w^{\mathcal{H}} \mathcal{L}_{\lambda_{t}} X(t) \rVert_{\dot{H}^{\frac{1}{3} - 2 \kappa}} + \lVert w^{\mathcal{H}} \succ \mathcal{H}_{\lambda_{t}}X(t) \rVert_{\dot{H}^{\frac{1}{3} - 2 \kappa}}+ \lVert w^{\mathcal{H}} \circ \mathcal{H}_{\lambda_{t}}X(t) \rVert_{\dot{H}^{\frac{1}{3} - 2 \kappa}} ]  \nonumber \\
\overset{\eqref{est 36} \eqref{est 35}}{\lesssim}& \lVert w^{\mathcal{L}}(t) \rVert_{\dot{H}^{\frac{2}{3} + 2 \kappa + 2 \epsilon}} [ \lambda_{t}^{\frac{1}{3}} (N_{t}^{\kappa})^{2} + (N_{t}^{\kappa})^{2} ] \overset{ \eqref{Higher frequency estimate} \eqref{Define M}}{\leq} \frac{\nu}{16} \lVert w^{\mathcal{L}}(t) \rVert_{\dot{H}^{1+ \epsilon}}^{2} + C(M, N_{t}^{\kappa}). \label{est 108}
\end{align}
Within $\RomanII_{4}$, we first work on 
\begin{equation}\label{est 83}
- \langle (-\partial_{x}^{2})^{\epsilon} w^{\mathcal{L}}, \partial_{x} w^{2} \rangle_{L^{2}}(t) = - \langle (-\partial_{x}^{2})^{\epsilon} w^{\mathcal{L}}, \partial_{x} \left( (w^{\mathcal{L}})^{2} + 2 w^{\mathcal{L}} w^{\mathcal{H}} + (w^{\mathcal{H}})^{2}  \right) \rangle_{L^{2}} (t). 
\end{equation}  
For the product of the lower order terms within \eqref{est 83}, we do a commutator type estimate despite absence of divergence-free property as follows: 
\begin{align}
& - \langle (-\partial_{x}^{2})^{\epsilon} w^{\mathcal{L}}, \partial_{x} (w^{\mathcal{L}})^{2} \rangle_{L^{2}} (t)  \nonumber \\
=& -2 \int_{\mathbb{T}} \left( [ (-\partial_{x}^{2})^{\frac{\epsilon}{2}},w^{\mathcal{L}} \partial_{x} ] w^{\mathcal{L}} \right) (-\partial_{x}^{2})^{\frac{\epsilon}{2}} w^{\mathcal{L}}(t) dx+ \int_{\mathbb{T}} \partial_{x} w^{\mathcal{L}} \lvert (-\partial_{x}^{2})^{\frac{\epsilon}{2}} w^{\mathcal{L}} (t)\rvert^{2}dx \nonumber  \\
\lesssim&  \lVert (-\partial_{x}^{2})^{\frac{\epsilon}{2}} w^{\mathcal{L}}(t) \rVert_{L^{4}} \lVert \partial_{x} w^{\mathcal{L}}(t) \rVert_{L^{2}} \lVert (-\partial_{x}^{2})^{\frac{\epsilon}{2}} w^{\mathcal{L}}(t) \rVert_{L^{4}} \leq \frac{\nu}{48} \lVert w^{\mathcal{L}}(t) \rVert_{\dot{H}^{1+ \epsilon}}^{2} + C(M, N_{t}^{\kappa}),  \label{est 84} 
\end{align}
where $[a,b] = ab - ba$. Next, for the product of the higher order terms within \eqref{est 83}, we rely on \eqref{est 42} and estimate 
\begin{align}
 - \langle (-\partial_{x}^{2})^{\epsilon} w^{\mathcal{L}}, \partial_{x} (w^{\mathcal{H}})^{2} \rangle_{L^{2}} (t) \overset{\eqref{est 42}}{\lesssim}& \lVert w^{\mathcal{L}}(t) \rVert_{\dot{H}^{1+ \epsilon}}^{\frac{1+ 6 \epsilon}{2(1+ \epsilon)}} \lVert w^{\mathcal{L}}(t) \rVert_{L^{2}}^{\frac{1- 4 \epsilon}{2(1+ \epsilon)}} (N_{t}^{\kappa})^{2}  \nonumber\\
\overset{\eqref{Define M}}{\leq}& \frac{\nu}{48} \lVert w^{\mathcal{L}}(t) \rVert_{\dot{H}^{1+ \epsilon}}^{2} + C(M, N_{t}^{\kappa}).\label{est 85}
\end{align} 
Next, the product of the higher and lower order terms within \eqref{est 83} is treated as follows:
\begin{align}
-2 \langle (-\partial_{x}^{2})^{\epsilon} w^{\mathcal{L}}, \partial_{x} (w^{\mathcal{L}} w^{\mathcal{H}}) &\rangle_{L^{2}} (t) \overset{\eqref{Product estimate}}{\lesssim} \lVert w^{\mathcal{L}}(t) \rVert_{\dot{H}^{\frac{1}{2} + 2 \epsilon + 2 \kappa}} \lVert w^{\mathcal{L}} (t) \rVert_{\dot{H}^{\frac{1}{2} - \kappa}} \lVert w^{\mathcal{H}}(t) \rVert_{\dot{H}^{\frac{1}{2} - \kappa}} \nonumber \\
\overset{\eqref{Define M} \eqref{Higher frequency estimate}}{\lesssim}&  \lVert w^{\mathcal{L}}(t) \rVert_{\dot{H}^{1+ \epsilon}}^{\frac{1+ 2 \epsilon + \kappa}{1+ \epsilon}} C(M, N_{t}^{\kappa}) \leq \frac{\nu}{48} \lVert w^{\mathcal{L}}(t) \rVert_{\dot{H}^{1+ \epsilon}}^{2}  + C(M, N_{t}^{\kappa}). \label{est 86}
\end{align}
Applying \eqref{est 84}, \eqref{est 85}, and \eqref{est 86} to \eqref{est 83} gives us 
\begin{equation}\label{est 100}
- \langle (-\partial_{x}^{2})^{\epsilon} w^{\mathcal{L}}, \partial_{x} w^{2} \rangle_{L^{2}}(t) \leq \frac{\nu}{16} \lVert w^{\mathcal{L}}(t) \rVert_{\dot{H}^{1+ \epsilon}}^{2} + C(M, N_{t}^{\kappa}).
\end{equation} 
Next, we estimate the rest of the non-commutator terms in $\RomanII_{4}$ of \eqref{Define II4}:
\begin{align}
& - \langle (-\partial_{x}^{2})^{\epsilon} w^{\mathcal{L}}, \partial_{x} (2w Y + Y^{2}) \rangle_{L^{2}} (t) \nonumber \\
\overset{\eqref{Bony 3} \eqref{Bony 4}}{\lesssim}& \lVert w^{\mathcal{L}} (t) \rVert_{\dot{H}^{1+\epsilon}}^{\frac{1- \frac{3\kappa}{2} + 2 \epsilon}{1+ \epsilon}} \lVert w^{\mathcal{L}}(t) \rVert_{L^{2}}^{\frac{ \frac{3\kappa}{2} - \epsilon}{1+ \epsilon}} [ \lVert w(t) \rVert_{H^{2\kappa}} + \lVert Y \rVert_{C_{t} \mathcal{C}^{2\kappa}} ] \lVert Y \rVert_{C_{t} \mathcal{C}^{2\kappa}}  \nonumber \\
\overset{\eqref{Define M} \eqref{Higher frequency estimate}}{\lesssim}&C(M, N_{t}^{\kappa}) [ \lVert w^{\mathcal{L}} (t) \rVert_{\dot{H}^{1+ \epsilon}}^{\frac{1 + \frac{\kappa}{2} + 2 \epsilon}{1+ \epsilon}} + \lVert w^{\mathcal{L}} (t) \rVert_{\dot{H}^{1+ \epsilon}}^{\frac{ 1 - \frac{3\kappa}{2} + 2 \epsilon}{1+ \epsilon}} ] \leq \frac{\nu}{16} \lVert w^{\mathcal{L}}(t) \rVert_{\dot{H}^{1+ \epsilon}}^{2} + C(M, N_{t}^{\kappa}). \label{est 101}
\end{align}
Finally, we work on the commutator term within $\RomanII_{4}$ of \eqref{Define II4}: 
\begin{align}
& \langle (-\partial_{x}^{2})^{\epsilon} w^{\mathcal{L}}, \partial_{x} C^{\prec} (w, Q^{\mathcal{H}}) \rangle_{L^{2}} (t)  \label{est 90} \\
\overset{\eqref{est 46b} }{=}& - \left\langle (-\partial_{x}^{2})^{\epsilon} w^{\mathcal{L}}, \partial_{x}\left( \frac{1}{2}  \partial_{x} (w^{2} + 2wY + 2wX+ Y^{2})  \prec Q^{\mathcal{H}} + 2 \nu \partial_{x} w \prec \partial_{x} Q^{\mathcal{H}} \right) \right\rangle_{L^{2}} (t). \nonumber 
\end{align}
Concerning the first term in \eqref{est 90}, computations in \eqref{est 51} shows 
\begin{align}
\lVert (\partial_{x} w^{2}) \prec  Q^{\mathcal{H}} (t) \rVert_{\dot{H}^{1- 2 \kappa - \gamma}} \lesssim N_{t}^{\kappa} (1+ \lVert w(t) \rVert_{L^{2}})^{-3 \gamma} \left(\lVert w^{\mathcal{L}} (t) \rVert_{\dot{H}^{\frac{1}{4} - \frac{\kappa}{4}}} + N_{t}^{\kappa} \right)^{2} \hspace{3mm} \forall \hspace{1mm} \gamma \geq 0,  \label{est 87}
\end{align}
so that for any $\gamma \in [0, \frac{3}{2} - \frac{3\kappa}{2})$ and $\eta \in \left( \max\{ \frac{1}{4}, 2 \kappa + \gamma + 2 \epsilon\}, 1+ \epsilon \right)$, we can compute 
\begin{align}
& - \frac{1}{2} \langle (-\partial_{x}^{2})^{\epsilon} w^{\mathcal{L}}, \partial_{x} [ ( \partial_{x} w^{2}) \prec Q^{\mathcal{H}} ] \rangle_{L^{2}} (t) \lesssim \lVert w^{\mathcal{L}}(t) \rVert_{\dot{H}^{2 \kappa + \gamma + 2 \epsilon}} \lVert (\partial_{x} w^{2}) \prec Q^{\mathcal{H}}(t) \rVert_{\dot{H}^{1- 2 \kappa - \gamma}}  \nonumber \\
\lesssim& C(M, N_{t}^{\kappa}) \Bigg(\lVert w^{\mathcal{L}}(t) \rVert_{\dot{H}^{1+ \epsilon}}^{\frac{3 \kappa + 2 \gamma + 4 \epsilon + 1}{2(1+\epsilon)}} +  \lVert w^{\mathcal{L}}(t) \rVert_{\dot{H}^{1+ \epsilon}}^{ \frac{2\kappa+ \gamma + 2 \epsilon}{1+ \epsilon}}   \Bigg). \label{est 89}
\end{align} 
For simplicity, we choose 
\begin{equation}\label{est 88}
\gamma = \frac{1}{2} \text{ and } \eta = \frac{3}{4}
\end{equation} 
to conclude 
\begin{equation}\label{est 95} 
 - \frac{1}{2} \langle (-\partial_{x}^{2})^{\epsilon} w^{\mathcal{L}}, \partial_{x} [ ( \partial_{x} w^{2}) \prec Q^{\mathcal{H}} ] \rangle_{L^{2}} (t)   \leq \frac{\nu}{80} \lVert w^{\mathcal{L}}(t) \rVert_{\dot{H}^{1+ \epsilon}}^{2} + C(M, N_{t}^{\kappa}).
\end{equation} 
Concerning the second term in \eqref{est 90}, using the fact that
\begin{align} 
\lVert [ \partial_{x} (wY) ] \prec Q^{\mathcal{H}}(t) \rVert_{\dot{H}^{1-\eta}} \lesssim (N_{t}^{\kappa})^{2} [ \lVert w^{\mathcal{L}}(t) \rVert_{H^{\kappa}} + N_{t}^{\kappa}]  \hspace{1mm} \forall \hspace{1mm} \eta \in (\kappa, 1) \label{est 91} 
\end{align}
due to \eqref{est 48} and \eqref{est 52}, we can deduce 
\begin{equation}\label{est 96} 
- \langle (-\partial_{x}^{2})^{\epsilon} w^{\mathcal{L}}, \partial_{x} ( [ \partial_{x} (wY) ] \prec Q^{\mathcal{H}}) \rangle_{L^{2}} (t) \leq \frac{\nu}{80} \lVert w^{\mathcal{L}} (t) \rVert_{\dot{H}^{1+ \epsilon}}^{2} + C(M, N_{t}^{\kappa}). 
\end{equation} 
Concerning the third term in \eqref{est 90}, making use of \eqref{est 50} and \eqref{est 53} that implies 
\begin{align}
 \lVert [ \partial_{x} (wX) ] \prec Q^{\mathcal{H}}(t) \rVert_{\dot{H}^{1-\eta}} \lesssim  (N_{t}^{\kappa})^{2} [ \lVert w^{\mathcal{L}}(t) \rVert_{\dot{H}^{\frac{3\kappa}{2}}} + N_{t}^{\kappa} ] \hspace{3mm} \forall \hspace{1mm} \eta \in (\kappa, 1),  \label{est 92} 
\end{align}
we can deduce 
\begin{equation}\label{est 97} 
- \langle (-\partial_{x}^{2})^{\epsilon} w^{\mathcal{L}}, \partial_{x} \left( [ \partial_{x} (wX) ] \prec Q^{\mathcal{H}} \right) \rangle_{L^{2}}(t)\leq \frac{\nu}{80} \lVert w^{\mathcal{L}}(t) \rVert_{\dot{H}^{1+ \epsilon}}^{2} + C(M, N_{t}^{\kappa}). 
\end{equation}   
Concerning the fourth term in \eqref{est 90}, we can rely on estimates from \eqref{est 54a} that implies 
\begin{align}\label{est 93} 
\lVert ( \partial_{x} Y^{2}) \prec Q^{\mathcal{H}}(t) \rVert_{\dot{H}^{1-\eta}}   \overset{\eqref{Bony 3} \eqref{Estimate on Q} \eqref{Define Lt and Nt}}{\lesssim} (N_{t}^{\kappa})^{3} \hspace{3mm} \forall \hspace{1mm} \eta \in (0,1), 
\end{align}
to deduce 
\begin{equation}\label{est 98}
- \frac{1}{2} \langle (-\partial_{x}^{2})^{\epsilon} w^{\mathcal{L}}, \partial_{x} \left( [ \partial_{x} Y^{2} ] \prec Q^{\mathcal{H}} \right) \rangle_{L^{2}} (t)   \leq \frac{\nu}{80} \lVert w^{\mathcal{L}} (t) \rVert_{\dot{H}^{1+ \epsilon}}^{2} + C(M, N_{t}^{\kappa}). 
\end{equation} 
Finally, concerning the fifth term in \eqref{est 90}, making use of \eqref{est 54b} that implies 
\begin{equation}\label{est 94} 
\lVert \partial_{x} w \prec \partial_{x} Q^{\mathcal{H}} (t) \rVert_{\dot{H}^{1-\eta}} \lesssim N_{t}^{\kappa} [ \lVert w^{\mathcal{L}}(t) \rVert_{\dot{H}^{\eta}} + N_{t}^{\kappa}] \hspace{3mm} \forall \hspace{1mm} \eta \geq \frac{1}{2} + \frac{3\kappa}{4}, 
\end{equation} 
leads us to 
\begin{equation}\label{est 99}
2 \left\langle ( - \partial_{x}^{2})^{\epsilon} w^{\mathcal{L}}, \partial_{x} \left( \nu \partial_{x} w \prec \partial_{x} Q^{\mathcal{H}} \right) \right\rangle_{L^{2}} (t)  \leq \frac{\nu}{80} \lVert w^{\mathcal{L}}(t) \rVert_{\dot{H}^{1+ \epsilon}}^{2} + C(M, N_{t}^{\kappa}). 
\end{equation} 
Applying \eqref{est 95}, \eqref{est 96}, \eqref{est 97}, \eqref{est 98}, and \eqref{est 99} to \eqref{est 90} gives us 
\begin{equation}\label{est 102}
 \langle (-\partial_{x}^{2})^{\epsilon} w^{\mathcal{L}}, \partial_{x} C^{\prec} (w, Q^{\mathcal{H}}) \rangle_{L^{2}} (t)\leq \frac{\nu}{16} \lVert w^{\mathcal{L}}(t) \rVert_{\dot{H}^{1+ \epsilon}}^{2} +C(M, N_{t}^{\kappa}).
\end{equation} 
In conclusion, we finally deduce by applying \eqref{est 100}, \eqref{est 101}, and \eqref{est 102} to \eqref{Define II4}, 
\begin{equation}\label{est 109}
\RomanII_{4} \leq \frac{3\nu}{16} \lVert w^{\mathcal{L}}(t) \rVert_{\dot{H}^{1+ \epsilon}}^{2} +C(M, N_{t}^{\kappa}). 
\end{equation} 
By applying \eqref{est 106}, \eqref{est 107}, \eqref{est 108}, and \eqref{est 109} to \eqref{est 105}, we obtain 
\begin{equation}
\partial_{t} \lVert w^{\mathcal{L}}(t) \rVert_{\dot{H}^{\epsilon}}^{2}  \leq - \frac{13 \nu}{8} \lVert w^{\mathcal{L}} (t) \rVert_{\dot{H}^{1+ \epsilon}}^{2} +C(M, N_{t}^{\kappa})
\end{equation} 
so that applying Gronwall's inequality completes the proof of Proposition \ref{Proposition 4.11}. 

\subsection{Proof of Proposition \ref{Proposition 4.12}}\label{Section B.3} 
The hypothesis of $\theta^{\text{in}} \in L^{2} (\mathbb{T})$ implies $\theta^{\text{in}} \in \mathcal{C}^{-1+ \kappa}(\mathbb{T})$ for $\kappa \in (0, \frac{1}{2})$ due to Bernstein's inequality, and therefore allows us via Proposition \ref{Proposition 4.2} to obtain $T^{\max} ( \{ L_{t}^{\kappa} \}_{t \geq 0}, \theta^{\text{in}}) \in (0, \infty]$ and a unique mild solution $w \in \mathcal{M}_{T^{\max}}^{\frac{\gamma}{2}} \mathcal{C}^{\frac{3\kappa}{4}}$ over $[0, T^{\max})$ with $\gamma = 1 - \frac{\kappa}{4}$ such that $\sup_{t \in [0, T^{\max} ]} t^{\frac{1}{2} - \frac{\kappa}{8}} \lVert w^{\mathcal{L}}(t) \rVert_{\mathcal{C}^{\frac{3\kappa}{4}}} < \infty$. It follows that 
\begin{equation}\label{est 104}
\lVert w^{\mathcal{L}}(t) \rVert_{H^{\zeta}} < \infty \hspace{3mm} \forall \hspace{1mm} t \in [0, T^{\max}) \hspace{1mm} \forall \hspace{1mm} \zeta < 1 - \kappa. 
\end{equation} 
Suppose that there exists some $i_{\max} \in \mathbb{N}_{0}$ such that $T_{i} = T^{\max}$ for all $i \geq i_{\max}$. Then, due to the Besov embedding of $H^{\kappa} (\mathbb{T}) \hookrightarrow \mathcal{C}^{-1+ 2\kappa}(\mathbb{T})$ and Proposition \ref{Proposition 4.11}, we can reach a contradiction to $T^{\max}$ and conclude that $T_{i} < T^{\max}$ for all $i \in \mathbb{N}$. 
 
\subsection{Proof of convergence $\mathbb{P}$-a.s. for Proposition \ref{Proposition 4.13}}\label{Section B.4}
We can compute for $\{ \lambda^{n} \}_{n \in \mathbb{N}}$, similarly to \eqref{est 157}, 
\begin{align}
& \mathbb{E} \left[ \lvert \Delta_{m} [ ( \partial_{x} \mathcal{L}_{\lambda^{n}} X \circ P^{\lambda^{n}} ) (t) - r_{\lambda^{n}} (t) - ( \partial_{x} \mathcal{L}_{\lambda^{n+1}} X \circ P^{\lambda^{n+1}} ) (t) + r_{\lambda^{n+1}} (t) ] \rvert^{2} \right]  \label{est 161}\\
=& \sum_{k, k' \in \mathbb{Z} \setminus \{0\}} e^{i 4 \pi (k+k')} \rho_{m}^{2} (k+k') \lvert \psi_{0} (k, k') \rvert^{2} (1+ \nu \lvert k' \rvert^{2})^{-1} \lvert k \rvert^{3} \lvert k'\rvert^{3} \nonumber\\
& \hspace{15mm} \times \int_{0}^{t} e^{-2 \nu \lvert k \rvert^{2} (t-s)} ds \int_{0}^{t} e^{- 2 \nu \lvert k' \rvert^{2} (t-s')} ds' \nonumber\\
& \hspace{15mm} \times \left[ \left[ \mathfrak{l} \left( \frac{\lvert k \rvert}{\lambda^{n}} \right) - \mathfrak{l} \left( \frac{\lvert k \rvert}{\lambda^{n+1}} \right) \right] \mathfrak{l} \left( \frac{\lvert k' \rvert}{\lambda^{n}} \right) + \mathfrak{l} \left( \frac{\lvert k \rvert}{\lambda^{n+1}} \right) \left[ \mathfrak{l} \left( \frac{\lvert k' \rvert}{\lambda^{n}} \right) - \mathfrak{l} \left( \frac{\lvert k' \rvert}{\lambda^{n+1}} \right) \right] \right] \nonumber\\
& \hspace{15mm} \times \Bigg[ (1+ \nu \lvert k \rvert^{2})^{-1} \left[ \mathfrak{l} \left( \frac{ \lvert k' \rvert}{\lambda^{n}} \right) - \mathfrak{l} \left( \frac{ \lvert k' \rvert}{\lambda^{n+1}} \right) \right] \mathfrak{l} \left( \frac{\lvert k \rvert}{\lambda^{n}} \right) + \mathfrak{l} \left( \frac{\lvert k' \rvert}{\lambda^{n+1}} \right) \left[ \mathfrak{l} \left( \frac{\lvert k \rvert}{\lambda^{n}} \right) - \mathfrak{l} \left( \frac{\lvert k \rvert}{\lambda^{n+1}} \right)  \right]  \nonumber \\
& \hspace{20mm} + (1+ \nu \lvert k' \rvert^{2})^{-1} \left[ \mathfrak{l} \left( \frac{ \lvert k \rvert}{\lambda^{n}} \right) - \mathfrak{l} \left( \frac{ \lvert k \rvert}{\lambda^{n+1}} \right) \right] \mathfrak{l} \left( \frac{\lvert k' \rvert}{\lambda^{n}} \right) + \mathfrak{l} \left( \frac{\lvert k \rvert}{\lambda^{n+1}} \right) \left[ \mathfrak{l} \left( \frac{\lvert k' \rvert}{\lambda^{n}} \right) - \mathfrak{l} \left( \frac{\lvert k' \rvert}{\lambda^{n+1}} \right)  \right] \Bigg].  \nonumber 
\end{align}
We estimate similarly to \eqref{est 158}  
\begin{align}
& \mathbb{E} \left[ \lvert \Delta_{m} [ ( \partial_{x} \mathcal{L}_{\lambda^{n}} X \circ P^{\lambda^{n}} ) (t) - r_{\lambda^{n}} (t) - ( \partial_{x} \mathcal{L}_{\lambda^{n+1}} X \circ P^{\lambda^{n+1}} ) (t) + r_{\lambda^{n+1}} (t) ] \rvert^{2} \right]   \nonumber \\
\overset{\eqref{Define psi 0}}{\lesssim}&  \sum_{k, k' \in \mathbb{Z} \setminus \{0\}: \lvert k \rvert \approx 2^{m}, \lvert k' \rvert \gtrsim 2^{m}} \lvert k' \rvert^{\frac{1}{2}} \left( \sum_{c: m \lesssim c} \frac{1}{2^{\frac{c}{4}}} \right)^{2} \left( \frac{1}{1+ \lvert k' \rvert^{2}}\right)^{2} \lvert k-k' \rvert \lvert k' \rvert \nonumber \\
& \hspace{30mm} \times  \left[ 1_{[ \lambda^{n}, \lambda^{n+1} ]} (\lvert k-k' \rvert) + 1_{[\lambda^{n}, \lambda^{n+1} ]} (\lvert k' \rvert) \right] \lesssim (\lambda^{n})^{-\frac{\kappa}{4}} 2^{\frac{m\kappa}{4}}.  \label{est 162}
\end{align}
We conclude that for all $p \in [2,\infty)$, due to Gaussian hypercontractivity theorem again, 
\begin{align*}
& \mathbb{E} \left[ \lVert ( \partial_{x} \mathcal{L}_{\lambda^{n}} X \circ P^{\lambda^{n}} ) (t) - r_{\lambda^{n}} (t) - [ ( \partial_{x} \mathcal{L}_{\lambda^{n+1}} X \circ P^{\lambda^{n+1}} ) (t) - r_{\lambda^{n+1}} (t) ] \rVert_{B_{p,p}^{-\kappa}}^{p}  \right]   \\
& \hspace{30mm} \overset{\eqref{est 162}}{\lesssim} \sum_{m=-1}^{\infty} 2^{- \kappa pm} \int_{\mathbb{T}} \lvert ( \lambda^{n})^{-\frac{\kappa}{4}} 2^{\frac{m \kappa}{4}} \rvert^{\frac{p}{2}} dx  \lesssim (\lambda^{n})^{-\frac{\kappa p}{8}}.
\end{align*}
  
\section*{Acknowledgments}
The author expresses deep gratitude to Prof. Jiahong Wu and Prof. Carl Mueller for valuable discussions.

\end{document}